\documentclass[graybox]{svmult}

\usepackage{type1cm}        
\usepackage{makeidx}         
\usepackage{graphicx}        
\usepackage{multicol}        
\usepackage[bottom]{footmisc}

\usepackage{newtxtext}       %
\usepackage[varvw]{newtxmath}       

\usepackage[normalem]{ulem}

\usepackage[capitalize]{cleveref}
\usepackage{bbm}
\usepackage{nicematrix}
\usepackage[normalem]{ulem}

\newcommand{\R}{\mathbb{R}}
\newcommand{\N}{\mathbb{N}}
\newcommand{\A}{\mathcal{A}}
\newcommand{\B}{\mathcal{B}}
\newcommand{\M}{\mathcal{M}}

\DeclareMathOperator{\rank}{rank} 
\renewcommand{\vec}{\operatorname{vec}} 
\DeclareMathOperator{\tr}{trace} 

\newcommand{\rowvec}[1]{\rule[2pt]{1em}{0.4pt}~#1~\rule[2pt]{1em}{0.4pt}}

\makeindex             

\begin{document}

\title*{Uniqueness of size-2 positive semidefinite matrix factorizations}
\author{Kristen Dawson \\  Serkan Ho\c{s}ten \\  Kaie Kubjas \\ Lilja Metsälampi}

\institute{Kristen Dawson \at Department of Mathematics,  UC Davis \email{kndawson@ucdavis.edu}
  \and Serkan Ho\c{s}ten \at Department of Mathematics, San Francisco State University \email{serkan@sfsu.edu}
  \and Kaie Kubjas \at Department of Mathematics and Systems Analysis, Aalto University \email{kaie.kubjas@aalto.fi}
  \and Lilja Metsal\"ampi \at Department of Mathematics and Systems Analysis, Aalto University \email{lilja.metsalampi@aalto.fi}}
%
%
\maketitle



\abstract{We characterize when a size-$2$ positive semidefinite (psd) factorization of a positive matrix of rank $3$ and psd rank $2$ is unique. The characterization is obtained using tools from rigidity theory. In the first step, we define $s$-infinitesimally rigid psd factorizations and characterize  $1$- and $2$-infinitesimally rigid size-$2$ psd factorizations. In the second step, we connect $1$- and $2$-infinitesimal rigidity of size-$2$ psd factorizations to uniqueness via  global rigidity.  We also prove necessary conditions on a positive matrix 
of rank $3$ and psd rank $2$ to be on the topological boundary of all nonnegative matrices with the same rank conditions.}

\section{Introduction}
Let $M \in \R^{p \times q}_+$ be a $p \times q$ real matrix with nonnegative entries, and let $\mathcal{S}^k$ and $\mathcal{S}_+^k$ be sets of $k \times k$ symmetric and symmetric positive semidefinite (psd) matrices, respectively.   For $k \in \N$, a size-$k$ psd factorization of $M$ is given by matrices $A_1,\ldots,A_p,B_1,\ldots,B_q \in \mathcal{S}^k_+$ satisfying $M_{ij} = \langle A_i, B_j\rangle:=\tr(A_i B_j)$. The psd rank of $M$ is the smallest $k \in \N$ such that $M$ has a size-$k$ psd factorization \cite{gouveia2013lifts}. The goal of this article is to study the uniqueness of size-2 psd factorizations of nonnegative matrices of rank three whose psd rank is two. 

The size-$k$ psd factorizations are never completely unique. Let $A_1,\ldots,A_p$, $B_1,\ldots,B_q$ be a size-$k$ psd factorization of a matrix $M$.  
Then, for any $S \in \operatorname{GL}(k)$, $S^TA_iS$, $i \in [p]$, and
$S^{-1}B_jS^{-T}$, $j\in[q]$ also give a psd factorization of $M$; see \cite[Section 6]{Fawzi:2015} and Lemma \ref{lemma:rotation-gives-psd-factorization}.  It is not clear whether a group larger than $\operatorname{GL}(k)$ acts on size-$k$ psd factorizations. However, at least for $k=2$, there are examples of matrices whose space of factorizations equals just one orbit of the $\operatorname{GL}(k)$-action \cite[Example 11]{Fawzi:2015}. 
Therefore we study uniqueness of psd factorizations up to the action of $\operatorname{GL}(k)$. 

A comprehensive survey of psd factorizations and psd rank is given in~\cite{Fawzi:2015}. Psd factorizations were originally introduced in~\cite{gouveia2013lifts} to study extension complexity of semidefinite programs: The psd rank of the slack matrix of a polytope is equal to the smallest $k$ where there exists a spectrahedron
in $\mathcal{S}_+^k$
with a linear projection on the polytope
\cite{Fawzi:2015, Fawzi:2022}. Optimizing a linear function over a complicated polytope translates into a semidefinite program over a potentially less complicated spectrahedron.
Uniqueness (up to the $\operatorname{GL}$-action) of psd factorizations of the slack matrix implies that there is essentially one spectrahedral lift of the corresponding polytope. 
 
Psd rank has applications also in quantum information theory~\cite{jain2013efficient,fiorini2012linear}. 
Consider the correlation generation game \cite{jain2013efficient} where two parties want to generate samples from a pair of correlated random variables $X$ and $Y$ following the joint distribution of $(X,Y)$.
The quantum protocol to achieve this goal can be constructed from a psd factorization of the matrix $M = P(X=i,Y=j)$ and a certain amount of shared information.
Each size-$k$ psd factorization of $M$ gives rise to a quantum protocol for the correlation generation game using $\log k$ qubits (and vice versa).
Hence, if the matrix $M$ has a unique size-$k$ psd factorization up to the $\operatorname{GL}$-action it follows there is a unique way to construct the quantum protocol for the correlation generation game using $\log k$ qubits.

The main result of the paper is the 
following.
\begin{theorem} \label{main-theorem}
Let $M\in \R_{+}^{p \times q}$ be a matrix of rank $3$ and psd rank $2$, and consider a size-$2$ psd factorization of $M$ given by the factors $A^{(1)}$, $\dots$, $A^{(p)}$, $B^{(1)}$, $\dots$, $B^{(q)} \in \mathcal{S}^2_+$.  Let $A^{(1)}$, $\dots$, $A^{(\bar{p})}$, $B^{(1)}$, $\dots$, $B^{(\bar{q})}$  be the factors that have rank one. For these,  let $A^{(i)} = a_i a_i^T$ and $B^{(j)} = b_j b_j^T$ for $a_i, b_j \in \R^2$. Assume that all $\det(a_i,a_j) \neq 0, \det(b_i,b_j) \neq 0$.
\begin{enumerate}
\item  If $M$ has positive entries, then $\left(A^{(1)}, \dots, A^{(p)}, B^{(1)}, \dots, B^{(q)}\right)$ is the unique size-2 psd factorization of $M$ up to $\text{GL}(2)$-action if and only if  there exist rank-one factors $A^{(i_1)},A^{(i_2)},A^{(i_3)}$, $B^{(j_1)},B^{(j_2)},B^{(j_3)}$ such that $\langle a_{i_k},b_{j_l} \rangle \neq 0$ and  
\begin{equation}\label{eqn:uniqueness-conditions}
\begin{split}    
\frac{\det(a_{i_1},a_{i_2})\det(a_{i_1},a_{i_3})\langle a_{i_1},b_{j_1} \rangle \langle a_{i_1},b_{j_2} \rangle}{\det(b_{j_1},b_{j_3})\det(b_{j_2},b_{j_3})\langle a_{i_2},b_{j_3} \rangle \langle a_{i_3},b_{j_3} \rangle} \geq 0, \\
-\frac{\det(a_{i_1},a_{i_2})\det(a_{i_2},a_{i_3})\langle a_{i_2},b_{j_1} \rangle \langle a_{i_2},b_{j_2} \rangle}{\det(b_{j_1},b_{j_3})\det(b_{j_2},b_{j_3})\langle a_{i_1},b_{j_3} \rangle \langle a_{i_3},b_{j_3} \rangle} \geq 0, \\
\frac{\det(a_{i_1},a_{i_3})\det(a_{i_2},a_{i_3})\langle a_{i_3},b_{j_1} \rangle \langle a_{i_3},b_{j_2} \rangle}{\det(b_{j_1},b_{j_3})\det(b_{j_2},b_{j_3})\langle a_{i_1},b_{j_3} \rangle \langle a_{i_2},b_{j_3} \rangle} \geq 0,\\
\frac{\det(b_{j_1},b_{j_2})\langle a_{i_1},b_{j_1} \rangle \langle a_{i_2},b_{j_1} \rangle \langle a_{i_3},b_{j_1} \rangle}{\det(b_{j_2},b_{j_3}) \langle a_{i_1},b_{j_3} \rangle \langle a_{i_2},b_{j_3} \rangle \langle a_{i_3},b_{j_3} \rangle} \geq 0, \\
-\frac{\det(b_{j_1},b_{j_2})\langle a_{i_1},b_{j_2} \rangle \langle a_{i_2},b_{j_2} \rangle \langle a_{i_3},b_{j_2} \rangle}{\det(b_{j_1},b_{j_3}) \langle a_{i_1},b_{j_3} \rangle \langle a_{i_2},b_{j_3} \rangle \langle a_{i_3},b_{j_3} \rangle} \geq 0.
\end{split}
\end{equation}
\item If $M$ has one zero entry and $(A^{(1)}$, $\dots$, $A^{(p)}$, $B^{(1)}$, $\dots$, $B^{(q)})$ is the unique size-2 psd factorization of $M$ up to $\text{GL}(2)$-action, then there exist rank-one factors $A^{(i_1)},A^{(i_2)},A^{(i_3)}$, $B^{(j_1)},B^{(j_2)},B^{(j_3)}$ such that $\langle a_{i_k},b_{j_l} \rangle \neq 0$ except  for $\langle a_{i_1},b_{j_1} \rangle$ and~\eqref{eqn:uniqueness-conditions} holds.
\end{enumerate}
\end{theorem}

To derive our main result, we use tools from rigidity theory \cite{Roth:81} where the main goal is to determine whether there is a unique realization of a set of points up to rigid transformations given  some pairwise distances between the points. The idea to apply rigidity theory to study uniqueness of matrix factorizations is not new and has been previously put into use for low-rank matrix completion~\cite{singer2010uniqueness}, nonnegative matrix factorizations~\cite{krone2021uniqueness} and general classes of matrix and tensor decompositions given by algebraic constraints~\cite{cruickshank2023identifiability}. Our setting is not covered by~\cite{cruickshank2023identifiability} because their factorizations are defined by equality constraints only, while psd factorizations require also inequality constraints to guarantee positive semidefiniteness of the factors.

Uniqueness of matrix factorizations corresponds to the notion of global rigidity in rigidity theory. To approach this question, we define weaker versions such as local rigidity and $s$-infinitesimal rigidity for $1 \leq s \leq k$. A psd factorization of a fixed matrix $M \in \R^{p \times q}_+$ is locally rigid if it is unique in some neighborhood up to the $\operatorname{GL}$-action, although the matrix can have other factorizations. Global and local rigidity are difficult to study both in rigidity theory and for matrix factorizations. $1$-Infinitesimal rigidity is a notion that implies local rigidity; see  \cref{lem:inf-rigidity-implies-local-rigidity}. Furthermore, $s$-infinitesimal rigidity is usually easier to study. 

Proving our main results consists of two main steps. In the first step we characterize when a size-$2$ psd factorization of a positive matrix of rank $3$ and psd rank $2$ is $1$- and $2$-infinitesimally rigid (\Cref{thm:no-orthogonal-pairs}). We also characterize $2$-infinitesimal rigidity for matrices with one or more zeros (\Cref{thm:one-orthogonal-pair} and \Cref{thm:two-orth-inf-rigid}). In the second step we show that for a positive matrix of rank $3$  and psd rank $2$,  $1$- and $2$-infinitesimal, local, and global rigidity coincide; see \Cref{thm:equivalence-of-rigidities}.

The outline of the paper is the following. In~\Cref{sec:infinitesimal-and-trivial-motions}, we introduce $s$-infinitesimal and $s$-trivial motions and define $s$-infinitesimal rigidity in the context of positive semidefinite factorizations. $1$-Trivial motions for size-$2$ factorizations and $k$-trivial motions for general size-$k$ factorizations are characterized in~\Cref{thm:1-trivial} and \Cref{thm:trivial-general-case}, respectively. In~\Cref{sec:prelim-results}, we prove preliminary results about positive semidefinite factorizations, $s$-infinitesimal motions and $s$-infinitesimal rigidity used in the rest of the paper. \Cref{sec:infinitesimal-rigidity} is where we focus on the characterization of $1$- and $2$-infinitesimally rigid size-$2$ psd factorizations. This section consists of three subsections~\Cref{sec:no-orthogonal-pairs},~\Cref{sec:one-orthogonal-pair} and~\Cref{sec:two-orthogonal-pairs} in which we consider the cases when $M$ has no zeros, one zero and two or more zeros. These correspond to the cases when there are no orthogonal pairs among the pairs of factors $(A_i,B_j)$, there is one orthogonal pair or there are two or more orthogonal pairs. Finally, in~\Cref{sec:rigidity-and-boundaries}, we introduce the notions of local and global rigidity. The main result in this section in which we characterize uniqueness of psd factorizations is 
\Cref{thm:equivalence-of-rigidities} where
we prove that for positive matrices of rank $3$ and psd rank $2$, local, global, and $1$- and $2$-infinitesimal rigidity coincide. This culminates in proving~\Cref{main-theorem}. Finally, in~\Cref{section:boundaries} we give necessary conditions for a positive 
matrix to be on the boundary of all matrices 
of rank $3$ and psd rank $2$. The code for the computations in this paper is available at GitHub:
\begin{center}
\url{https://github.com/kaiekubjas/uniqueness-of-size-2-psd-factorizations/}
\end{center}

\section{$s$-Infinitesimal and $s$-trivial motions of psd factorizations}\label{sec:infinitesimal-and-trivial-motions}

In~\Cref{sec:infinitestimal-motions}, we introduce $s$-infinitesimal motions. Then we define $s$-trivial motions and $s$-infinitesimal rigidity in~\Cref{section:trivial-motions-infinitesimal-rigidity}. \Cref{section:1-trivial-motions} and
\Cref{sec:k-trivial-motions} characterize 
$1$-trivial motions of size-$2$ psd factorizations of matrices of rank $3$ and psd rank $2$, and $k$-trivial motions of size-$k$ psd factorizations of matrices of rank $\binom{k+1}{2}$ and psd rank $k$, respectively.

\subsection{$s$-Infinitesimal motions} \label{sec:infinitestimal-motions}

For a matrix $X \in \mathbb{R}^{n \times n}$ and a set $I \subseteq [n]$, we denote by $X_I$ the submatrix of $X$ with row and column indices in the set $I$. We denote the determinant of $X_I$ by $[X]_I$. For a polynomial $f \in \mathbb{R}[t]$, we denote its Taylor approximation  of order $s$ by $T^s(f)$. 

Let $M \in \R_+^{p \times q}$ be a nonnegative matrix of positive semidefinite rank $k$, and let a size-$k$ positive semidefinite factorization of $M$ be given by  $A^{(1)}, \dots, A^{(p)}$, $B^{(1)}, \dots, B^{(q)} \in \mathcal{S}_+^k.$
Let $(A^{(1)}(t), \dots, A^{(p)}(t)$, $B^{(1)}(t), \dots, B^{(q)}(t) )$  be a smooth function on $[0, 1]$ such that $A^{(i)}(0) = A^{(i)}$, $B^{(j)}(0) = B^{(j)}$, $A^{(i)}(t),B^{(j)}(t) \in \mathcal{S}_+^k$ and $\langle A^{(i)}(t), B^{(j)}(t) \rangle = M_{ij}$ for all $i\in [p]$, $j \in [q]$ and $t \in [0,1]$.
Taking derivatives with respect to $t$ and substituting $t=0$, we get 
\begin{equation*}
    \langle \dot{A}^{(i)},B^{(j)}  \rangle + \langle A^{(i)},\dot{B}^{(j)}  \rangle = 0 \text{ for all } (i,j) \in [p] \times [q],
\end{equation*}
where $\dot A^{(i)} := \frac{dA^{(i)}}{dt}(0)$ and $\dot B^{(j)} := \frac{dB^{(j)}}{dt}(0)$.  Furthermore, $[A^{(i)} + t \dot A^{(i)}]_I$ and $[B^{(j)} + t \dot B^{(j)}]_J$ are polynomials in $t$ for all $I, J \subseteq [k]$. One would like      $A^{(i)} + t \dot A^{(i)}$ and $B^{(j)} + t \dot B^{(j)}$ to be  psd matrices as well. This is equivalent for the above determinants to be nonnegative for all $I, J \subset [k]$. However, we will require a slightly weaker condition.  
This motivates the next definition.

\begin{definition}[$s$-Infinitesimal motion of a psd factorization]\label{def:psd-infinitesimal} 
Let $M \in \R_+^{p \times q}$ be a nonnegative matrix of positive semidefinite rank $k$. 
Let a size-$k$ positive semidefinite factorization of $M$ be given by  $A^{(1)}, \dots, A^{(p)}, B^{(1)}, \dots, B^{(q)} \in \mathcal{S}_+^k$.
A collection of real symmetric $k \times k$ matrices
\begin{align*}
    \left(\dot A^{(1)}, \dots, \dot A^{(p)}, \dot B^{(1)}, \dots , \dot B^{(q)} \right)
\end{align*} 
is called an \textit{$s$-infinitesimal motion}, for $1 \leq s \leq k$, of the factorization if 
\begin{enumerate}
    \item $\langle \dot{A}^{(i)},B^{(j)}  \rangle + \langle A^{(i)},\dot{B}^{(j)}  \rangle = 0$ for all $(i,j) \in [p] \times [q]$, and \label{cond:psd-equality}

    \item there exists $\varepsilon >0$ such that for $t \in [0,\varepsilon)$, $T^s([A^{(i)} + t \dot{A}^{(i)}]_I) \geq 0$ for all $i \in [p]$ and $I \subseteq [k]$ and $T^s([B^{(j)} + t \dot{B}^{(j)}]_J) \geq 0$ for all $j \in [q]$ and $J \subseteq [k]$.
\end{enumerate}
\end{definition}

\begin{remark}
For a $k$-infinitesimal motion, 
the second condition in~\Cref{def:psd-infinitesimal} is 
equivalent to 
 $A^{(i)} + t \dot{A}^{(i)} \in \mathcal{S}^k_+$ for all $i \in [p]$ and $B^{(j)} + t \dot{B}^{(j)} \in \mathcal{S}^k_+$ for all $j \in [q]$. 
\end{remark}

\begin{remark}\label{rmk:only-zeros-matter}
If $[A^{(i)}]_I > 0$, then $T^s([A^{(i)} + t \dot{A}^{(i)}]_I) \geq 0$ is automatically satisfied for small enough $t>0$, and similarly for the matrices $B^{(j)}$. Therefore, the inequalities of the second condition in~\Cref{def:psd-infinitesimal} are only relevant for the factors and submatrices such that $[A^{(i)}]_I = 0$ or $[B^{(j)}]_J = 0$.
\end{remark}

\begin{remark}
The work in this paper is similar to the study of uniqueness of nonnegative matrix factorizations using rigidity theory in~\cite{krone2021uniqueness}. Nevertheless, an important difference is that in the case of nonnegative matrix factorizations, all the inequality conditions on the factors are linear in the entries of factors, which is not the case for psd factorizations. Therefore, for nonnegative factorizations there are only infinitesimal motions instead of $s$-infinitesimal motions for $s \in \{1,\ldots,k\}$, and the definition includes only linear inequalities. As a result, the study of infinitesimal rigidity of nonnegative factorizations is in the field of polyhedral geometry and that of psd factorizations is in the field of real algebraic geometry, which can be considerably more challenging. Sometimes a result for psd factorizations holds only for a specific $s$, e.g., 1-infinitesimal rigidity implies local rigidity by~\cref{lem:inf-rigidity-implies-local-rigidity}. When rank-3 psd rank-2 matrices have zeros, then 2-infinitesimal rigidity is more interesting than 1-infinitesimal rigidity as we discuss in the next section.
\end{remark}

\begin{lemma}\label{lemma:t-inf-motion-is-s-inf-motion}
Let $M \in \R_+^{p \times q}$ be a nonnegative matrix of positive semidefinite rank $k$. 
Let a size-$k$ positive semidefinite factorization of $M$ be given by  $A^{(1)}, \dots, A^{(p)}, \allowbreak B^{(1)}, \dots, B^{(q)} \in \mathcal{S}_+^k$. For $s<r$, any $r$-infinitesimal motion is an $s$-infinitesimal motion.
\end{lemma}

\begin{proof}
We have $T^r([A^{(i)} + t \dot{A}^{(i)}]_I) = T^{r-1}([A^{(i)} + t \dot{A}^{(i)}]_I) + t^r \gamma_I$, where $\gamma_I \in \mathbb{R}$.  If for every $\varepsilon>0$ there exists $t \in (0, \varepsilon)$ such that $T^{r-1}([A^{(i)} + t \dot{A}^{(i)}]_I)<0$, then for every $\varepsilon>0$ there exists $t \in (0, \varepsilon)$ such that $T^r([A^{(i)} + t \dot{A}^{(i)}]_I) <0$. Hence if a motion is not a $(r-1)$-infinitesimal motion, then it is not an $r$-infinitesimal motion. This proves that an $r$-infinitesimal motion is a $(r-1)$-infinitesimal motion. 
Now induction gives the result. 
\end{proof}

\begin{remark} \label{rmk:inner-product-not-preserved}
In the rigidity theory of bar-and-joint frameworks an infinitesimal motion $\dot p$ of a framework $(G,p)$ need not correspond to an actual smooth path that preserves edge lengths; see \cite[Definition 4.1]{Roth:81} for reference.
Instead, infinitesimal motions correspond to smooth paths that instantaneously preserve edge lengths of the framework.
Analogously to this, an $s$-infinitesimal motion of a factorization is only required to instantaneously preserve the inner products, which is expressed by the first condition of \Cref{def:psd-infinitesimal}.
In particular, an $s$-infinitesimal motion of a psd factorization of $M$ need not correspond to a smooth path $(A^{(1)}(t), \dots, A^{(p)}(t)$, $B^{(1)}(t), \dots, B^{(q)}(t))$ on $[0,1]$ such that
\begin{align}
\label{eqn:inner-pr-preserving-factors}
    M_{ij} = \langle A^{(i)}(t), B^{(j)}(t) \rangle \text{ for all } t \in [0,1].
\end{align} 
\end{remark}

 Throughout the paper, we will consider nonnegative matrices $M$ with rank $\binom{k+1}{2}$ and psd rank $k$. The set of such $p \times q$ matrices will be denoted by $\M_{\binom{k+1}{2},k}^{p \times q}$.
We use the following vectorization of a symmetric matrix as in \cite[Proposition 2]{Fawzi:2015}.  For $X\in \mathcal{S}^k$ we let  $\text{vec}(X) \in \R^{\binom{k+1}{2}}$ where 
\[\underline{x} := \text{vec}(X) = (X_{11},\ldots,X_{kk},\sqrt{2}X_{12},\ldots,\sqrt{2}X_{1k},\sqrt{2}X_{23},\ldots,\sqrt{2}X_{(k-1)k}).\]
Furthermore we define $\A$ as a $p\times \binom{k+1}{2}$ matrix with rows $\underline{a}^{(i)}$ and $\B$ as a $\binom{k+1}{2}\times q$ matrix with the columns $\underline{b}^{(j)T}$, that is,  
\begin{equation}\label{eqn:calA-and-calB}
    \A = \begin{pmatrix}
\rowvec{\underline{a}^{(1)}}\\
 \vdots  \\
\rowvec{\underline{a}^{(p)}} 
\end{pmatrix} \  \text{ and } \  \B = \begin{pmatrix}
\vert & & \vert \\
\underline{b}^{(1)T} & \dots & \underline{b}^{(q)T} \\
\vert & & \vert
\end{pmatrix}.
\end{equation}

 Then $M = \A \B$ since $\mathrm{trace}(XY) = \underline{x} \, \underline{y}^T$ for matrices $X,Y \in \mathcal{S}^k$. 
 Recalling that $\rank(M)=\binom{k+1}{2}$, we get  $\rank(M)=\rank(\A)=\rank(\B)$. 
With this notation, condition  (1)  in \Cref{def:psd-infinitesimal} can be written in matrix form as $\dot\A \B + \A \dot\B = 0$, where the rows of $\dot \A$ and columns of $\dot \B$ are $\dot{\underline{a}}^{(i)}$ and $\dot{\underline{b}}^{(j)T},$ respectively. 
As in the case of nonnegative factorizations \cite[Section 3]{krone2021uniqueness}, for this equation to hold, the column span of $\dot\A$ must be contained in the column span of $\A$, and the row span of $\dot\B$ must be contained in the row span of $\B$.  This means that $\dot\A= \A D_1$ and $\dot \B = -D_2\B$ where $D_1$ and $D_2$ are $\binom{k+1}{2} \times \binom{k+1}{2}$ matrices. Now $\A D_1 \B - \A D_2 \B = 0$, and because both $\A$ and $\B$ have full rank, we conclude that $D_1=D_2=D$. 
We will often need the entries of $D$ be written in vector form, so we define
$$ \underline{D} = 
(D_{11},\ldots,D_{1\binom{k+1}{2}},D_{21},\ldots,D_{2\binom{k+1}{2}},\ldots,D_{\binom{k+1}{2}1}, \ldots, D_{\binom{k+1}{2}\binom{k+1}{2}}).$$

\subsection{$s$-Trivial motions and $s$-infinitesimal rigidity} \label{section:trivial-motions-infinitesimal-rigidity}

The discussion above implies that for each $s$-infinitesimal motion of a size-$k$ psd factorization of a matrix $M \in \mathcal{M}_{\binom{k+1}{2},k}^{p \times q}$, there is a corresponding matrix $D$ such that $\dot \A = \A D$ and $\dot \B = -D \B$, that is,  $\dot A^{(i)}= \operatorname{vec}^{-1}(\underline{a}^{(i)}D)$ and $\dot B^{(i)}= \operatorname{vec}^{-1}(-\underline{b}^{(j)}D^T)$.

In the next example we show that $D=dI$, where  
$I$ is the $\binom{k+1}{2} \times \binom{k+1}{2}$ identity matrix, can be used to induce an $s$-infinitesimal motion for {\it any} psd factorization for all $s \in [k]$.

\begin{example}\label{ex:trivial-motions}
Let $M \in \R_+^{p \times q}$ and consider a size-$k$ psd factorization of $M$ given by the positive semidefinite matrices $A^{(1)}, \dots, A^{(p)}, B^{(1)}, \dots, B^{(q)}$.
Define a smooth path of the factorization on $[0,1]$  by 
\begin{equation}
\label{eqn:trivial-motion}
A^{(i)} (t) = (1+td)A^{(i)}\mbox{  and  } 
        B^{(j)} (t) = (1+ td)^{-1}B^{(j)} \text{for $t \in [0,1]$}.
\end{equation}
These are symmetric psd matrices for $t \in [0, |\frac{1}{d}|)$, and $A^{(i)}(0) = A^{(i)}$ and $B^{(j)}(0) = B^{(j)}$. 
Furthermore, by differentiating with respect to $t$ and substituting $t = 0$, we get 
$\dot A^{(i)} = dA^{(i)}$ and  
    $\dot B^{(j)} = -d B^{(j)}$,
which immediately implies that $\langle A^{(i)}, \dot B^{(j)} \rangle + \langle \dot A^{(i)}, B^{(j)} \rangle = 0$ for all pairs $(i,j) \in [p] \times [q]$.
Clearly, $A^{(i)} + t \dot A^{(i)}$ and $B^{(j)} + t \dot B^{(j)}$ are also psd matrices for 
all $t \in [0, |\frac{1}{d}|)$.
Hence \eqref{eqn:trivial-motion} defines a $k$-infinitesimal motion $(\dot A^{(1)}, \dots, \dot A^{(p)}, \dot B^{(1)}, \dots, \dot B^{(q)})$ of the factorization 
corresponding to the matrix $D=dI$. 
By~\Cref{lemma:t-inf-motion-is-s-inf-motion}, the $k$-infinitesimal motion~\eqref{eqn:trivial-motion} is an $s$-infinitesimal motion for every $s \in [k]$.
\end{example}

\begin{definition}[$s$-Trivial motion] \label{def:trivial-motion} 
An $s$-infinitesimal motion of a size-$k$ psd factorization of $M \in \M_{\binom{k+1}{2},k}^{p \times q}$  is called an $s$-\textit{trivial motion}, if the corresponding matrix $D$ gives rise to an $s$-infinitesimal motion of any size-$k$ psd factorization of all matrices in $\M_{\binom{k+1}{2},k}^{p \times q}$.
\end{definition}

\begin{definition}[$s$-Infinitesimally rigid factorization]\label{def:infinitesimally-rigid-factorization}     
A size-$k$ psd factorization of $M \in \M_{\binom{k+1}{2},k}^{p \times q}$ is $s$-\textit{infinitesimally rigid} if all its $s$-infinitesimal motions are $s$-trivial motions. A factorization that is not $s$-infinitesimally rigid is called {\it $s$-infinitesimally flexible}.
\end{definition}

As we noted above $\dot{\mathcal{A}}$ and $\dot{\mathcal{B}}$ of the form $\dot{\A} = \A D$ and $\dot{\B} = -D \B$ always satisfy condition (1) of \Cref{def:psd-infinitesimal} for any size-$k$ psd factorization of a rank at most $\binom{k+1}{2}$ matrix. But it could be that for matrices of rank less than $\binom{k+1}{2}$ there are other solutions to the equation in condition (1) in \Cref{def:psd-infinitesimal} that cannot be characterized by a matrix $D$. Therefore $s$-infinitesimal rigidity cannot be defined for matrices of rank strictly less than $\binom{k+1}{2}$ as the definition of $s$-trivial motions uses matrix $D$.

\color{black}
\begin{lemma}\label{lem:sub-facorization-rigid}
    Let $M \in \M_{\binom{k+1}{2},k}^{p \times q}$ 
     and consider a size-$k$ psd factorization of $M$ given by the factors $A^{(1)}$, $\dots$, $A^{(p)}$, $B^{(1)}$, $\dots$, $B^{(q)}$.
    If there exists a subset of the factors which corresponds to an $s$-infinitesimally rigid factorization of a full-rank submatrix $M'$ of $M$, then the whole factorization is $s$-infinitesimally rigid.

    \begin{proof}
        Let $A^{(i_1)}$, $\dots$, $A^{(i_{p'})}$, $B^{(j_1)}$, $\dots$, $B^{(j_{q'})}$ be a subset of the factors of the psd factorization corresponding to an $s$-infinitesimally rigid psd factorization of a $p' \times q'$   full-rank  submatrix $M'$ of $M$.
        Then, the only $s$-infinitesimal motions of the factorization given by the factors $A^{(i_1)}$, $\dots$, $A^{(i_{p'})}$, $B^{(j_1)}$, $\dots$, $B^{(j_{q'})}$ are the $s$-trivial motions.
        Since the $s$-infinitesimal motions of the whole factorization need to satisfy the constraints imposed by the subset $A^{(i_1)}$, $\dots$, $A^{(i_{p'})}$, $B^{(j_1)}$, $\dots$, $B^{(j_{q'})}$ of the factors, it follows that the whole factorization is $s$-infinitesimally rigid. 
    \end{proof}
\end{lemma}

\subsection{1-Trivial motions} \label{section:1-trivial-motions}

In this subsection we will give a characterization of 1-trivial motions.  To do this, we recall the action of $\operatorname{GL}(k)$ on psd factorizations. 

\begin{lemma}\label{lemma:rotation-gives-psd-factorization}
Let $M \in \R_+^{p \times q}$ and let $A^{(i)} \in \mathcal{S}_+^k, \, i\in [p]$ and $B^{(j)} \in \mathcal{S}_+^k, \, j \in [q]$ be a size-$k$ psd factorization of $M$. 
For any $S \in \operatorname{GL}(k)$, $\Tilde{A}^{(i)} = S^TA^{(i)}S$ and
$\Tilde{B}^{(j)} = S^{-1}B^{(j)}S^{-T}$ is also a psd factorization of $M$.
\end{lemma}
\begin{proof}
For any $y \in \R^k$, $y^TS^TA^{(i)}Sy= (Sy)^TA^{(i)}(Sy) \geq 0$, because $A^{(i)}$ is psd. This shows $\Tilde{A}^{(i)}$ is psd. Same argument works for $\Tilde{B}^{(j)}$. 
Moreover $\tr(\Tilde{A}^{(i)}\Tilde{B}^{(j)}) = \tr(S^TA^{(i)}B^{(j)}S^{-T}) = \tr(A^{(i)}B^{(j)})$.
This shows that we get a
psd factorization of $M$.
\end{proof}

Given a matrix  $M \in \M_{\binom{k+1}{2},k}^{p \times q}$  and a size-$k$ psd factorization
$(A^{(1)}, \dots, A^{(p)}$, $B^{(1)}, \dots, B^{(q)})$ of $M$, 
the set of all size-$k$ psd factorizations of $M$ contains the set of size-$k$ psd factorizations obtained from the above psd factorization by the  $\operatorname{GL}(k)$-action. 
Consider the matrix multiplication map 
\[
\mu: \mathbb{R}^{p \times \binom{k+1}{2}} \times \mathbb{R}^{\binom{k+1}{2} \times q} \rightarrow \mathbb{R}^{p \times q}, \quad (\mathcal{A}, \mathcal{B}) \mapsto \mathcal{A} \mathcal{B},
\]
where the domain is restricted to pairs of matrices of rank $\binom{k+1}{2}$. The image of this map is the set of $p \times q$ matrices of rank $\binom{k+1}{2}$. Define the set
\[
F := \left \{ (C,C^{-1}) \in \mathbb{R}^{\binom{k+1}{2} \times \binom{k+1}{2}} \times \mathbb{R}^{\binom{k+1}{2} \times \binom{k+1}{2}}: C \text{ is invertible}\right \}.
\]
A factorization $(\mathcal{A} C,C^{-1} \mathcal{B})$ is a psd factorization  if $\mathcal{A} C \in (\mathcal{S}_+^k)^p$ and $C^{-1} \mathcal{B} \in (\mathcal{S}_+^k)^q$. 
Therefore, after fixing a size-$k$ psd factorization of $M$, we can identify  all its size-$k$ psd factorizations 
with a subset of $F$. Similarly, all size-$k$ psd factorizations of $M$ obtained by the  $\operatorname{GL}(k)$-action can be identified with a further subset of $F$. We denote the latter set by $\mathcal{M}_{GL(k)}$ and call its elements rigid motions of $\mathcal{S}_+^k$. 

Let $I$ denote the $\binom{k+1}{2} \times \binom{k+1}{2}$ identity matrix.
Then $(I,I) \in F$, and we may consider the tangent space of $F$ at the point $(I,I)$, which
is the affine space 
\[
T_{(I,I)} F:=\{(I+D,I-D)\, : \,  D \in \mathbb{R}^{\binom{k+1}{2} \times \binom{k+1}{2}}\}.
\]
The tangent space $T_{(I,I)} F$ is the translate of the linear space
\[
L_{(I,I)} F:=\{(D,-D) \in \mathbb{R}^{\binom{k+1}{2} \times \binom{k+1}{2}} \times \mathbb{R}^{\binom{k+1}{2} \times \binom{k+1}{2}}\}.
\]
We call the elements of $L_{(I,I)} F$ the tangent directions. We extend these definitions to $T_{(I,I)}\M_{\operatorname{GL}(k)}$ and
$L_{(I,I)}\M_{\operatorname{GL}(k)}$ 
as well. We will index the rows and the columns of  matrices $C$ and $D$ in $\mathbb{R}^{\binom{k+1}{2} \times \binom{k+1}{2}}$ by $(1,1), (2,2), \ldots, (k,k), (1,2), \ldots, (1,k), (2,3)$,   $(k-1,k)$.

\begin{proposition} \label{prop:tangent-space-at-M}
 The linear space $L_{(I,I)}\M_{\operatorname{GL}(k)}$  has dimension $k^2$ and it  
 is generated by $(D, 
  -D)$ where $D$ comes in two types:
 \begin{enumerate}
 \item for each $i\in [k]$, the matrix $D$ where $D_{(i,i),(i,i)} =2$, \\ $D_{(\min(i,j),\max(i,j)), (\min(i,j),\max(i,j))} = 1$ for $j\neq i$, and all other entries equal to zero, and
 \item for each $1 \leq i \neq j \leq k$, the matrix $D$ where $D_{(\min(i,j),\max(i,j)),(j,j)} = D_{(i,i),(\min(i,j),\max(i,j))} = \sqrt{2}$ and $D_{(\min(i,\ell),\max(i,\ell)),(\min(j,\ell),\max(j,\ell))} = 1$ for all $\ell\neq i,j$, and all other entries equal to zero. 
 \end{enumerate}
 This linear space always contains the set of diagonal matrices with equal entries on the diagonal.
 \end{proposition}
 \begin{proof} Since $ \dim(\M_{\operatorname{GL}(k)}) =\dim(\operatorname{GL}(k))=k^2$, the dimension of the tangent space is $k^2$.
 We note that the product $S^TA^{(i)}S$ is the same as the product of the $i$-th row of $\mathcal{A}$ with the matrix
\[
\begin{pmatrix}
s_{11}^2 & \cdots & s_{1k}^2 & \sqrt{2}s_{11}s_{12}& \cdots & \sqrt{2}s_{1k-1}s_{1k}\\
\vdots & \ddots & \vdots & \vdots & \ddots & \vdots \\
s_{k1}^2 & \cdots & s_{kk}^2 & \sqrt{2}s_{k1}s_{k2}& \cdots & \sqrt{2}s_{kk-1}s_{kk}\\
\sqrt{2}s_{11}s_{21} & \cdots & \sqrt{2}s_{1k}s_{2k} & s_{11}s_{22}+s_{12}s_{21} & \cdots & s_{1k-1}s_{2k}+s_{1k}s_{2k-1}\\
\vdots & \ddots & \vdots & \vdots & \ddots & \vdots\\
\sqrt{2}s_{11}s_{k1} & \cdots & \sqrt{2}s_{1k}s_{kk} & s_{11}s_{k2}+s_{12}s_{k1} & \cdots & s_{1k-1}s_{kk}+s_{1k}s_{kk-1}\\
\sqrt{2}s_{21}s_{31} & \cdots & \sqrt{2}s_{2k}s_{3k} & s_{21}s_{32}+s_{22}s_{31} & \cdots & s_{2k-1}s_{3k}+s_{2k}s_{3k-1}\\
\vdots & \ddots & \vdots & \vdots & \ddots & \vdots\\
\sqrt{2}s_{k-11}s_{k1} & \cdots & \sqrt{2}s_{k-1k}s_{kk} & s_{k-11}s_{k2}+s_{k-12}s_{k1} & \cdots & s_{k-1k-1}s_{kk}+s_{k-1k}s_{kk-1}\\
\end{pmatrix}
\]   
In the above matrix, the entry corresponding to index $((i,i),(j,j))$ is $s_{ij}^2$. The entry corresponding to index $((i,j),(\ell, \ell))$ is $\sqrt{2}s_{i\ell}s_{j\ell}$. The entry corresponding to index $((i,i),(j,\ell))$ is $\sqrt{2}s_{ij}s_{i\ell}$. The entry corresponding to index $((i,j),(\ell, o))$ is $s_{i\ell}s_{jo}+s_{io}s_{j\ell}$.
Taking the partial derivative with respect to $s_{ii}$ and then substituting the identity matrix gives a matrix $D$ of the first type. 
Taking partial derivative with respect to $s_{ij}$, $i \neq j$, and then substituting the identity matrix gives a matrix $D$ of the second type. Adding up all matrices of the first type produces a diagonal matrix with equal entries on the diagonal.
\end{proof}
We will see later that diagonal matrices with equal entries will  give $s$-trivial motions for $s\in[k]$.
\begin{example}
When $k=2$, the tangent space of $\mathcal{M}_{\operatorname{GL}(2)}$ at $(I,I)$ is the translate of the $4$-dimensional linear subspace $L_{(I,I)}\M_{\operatorname{GL}(2)}$ generated by 
$(D_i, -D_i)$ where 
\begin{equation} \label{eqn:1-trivial-motions-for-size-2-psd-factorizations}
D_1=\begin{pmatrix}
2 & 0 & 0\\
0 & 0 & 0\\
0 & 0 & 1
\end{pmatrix},
D_2=\begin{pmatrix}
0 & 0 & 0\\
0 & 2 & 0\\
0 & 0 & 1
\end{pmatrix},
D_3=\begin{pmatrix}
0 & 0 & 0\\
0 & 0 & \sqrt{2}\\
\sqrt{2} & 0 & 0
\end{pmatrix},
D_4=\begin{pmatrix}
0 & 0 & \sqrt{2}\\
0 & 0 & 0\\
0 & \sqrt{2} & 0
\end{pmatrix}.
\end{equation}
\end{example}

\begin{theorem} \label{thm:1-trivial}
The set of $1$-trivial motions of size-$2$ psd factorizations of matrices in $\mathcal{M}_{3,2}^{p \times q}$ is equal to $L_{(I,I)} \M_{\operatorname{GL}(2)}$, i.e., it is the linear subspace generated by the matrices~\eqref{eqn:1-trivial-motions-for-size-2-psd-factorizations}.
\end{theorem}

\begin{proof}
First we will show that a linear combination of matrices in~\eqref{eqn:1-trivial-motions-for-size-2-psd-factorizations} gives a 1-infinitesimal motion. Let $A \in \mathcal{S}^2_+$, $t \geq 0$, and $D= \sum_{i=1}^4 \lambda_i D_i$. Then
\[
 A + t \dot{A} =
\begin{pmatrix}
a_{11} + 2t(\lambda_1 a_{11} + \lambda_3 a_{12}) & 
  a_{12} + t(\lambda_1 a_{12} + \lambda_2 a_{12} + \lambda_3 a_{22} + \lambda_4 a_{11})\\
  a_{12} + 
   t(\lambda_1 a_{12} + \lambda_2 a_{12} + \lambda_3 a_{22} + \lambda_4 a_{11}) &
  a_{22} + 2t(\lambda_2 a_{22} + \lambda_4 a_{12})
  \end{pmatrix}.
\]
We want to show that $T^1([A + t \dot{A}]_I) = T^1([A + t \dot{A}]_I) \geq 0$ for all $I \subseteq [2]$. We first consider the diagonal entries. If $[A]_{\{i\}}=a_{ii}>0$, then $[A + t \dot{A}]_{\{i\}} >0$ for small enough $t>0$. If $[A]_{\{i\}}=a_{ii}=0$, then also $a_{12}=0$ because of positive semidefiniteness, and hence $[A + t \dot{A}]_{\{i\}} =0$ for any $t$. The linear part of the determinant is  
\[
T^1([A + t \dot{A}]_{\{1,2\}}) = a_{11} a_{22} -a_{12}^2 + 2 (a_{11} a_{22} - a_{12}^2 ) (\lambda_1 + \lambda_2) t.
\]
If $T^1([A]_{\{1,2\}})=a_{11} a_{22} -a_{12}^2>0$, then $T^1([A + t \dot{A}]_{\{1,2\}}) >0$ for small enough $t>0$. If $T^1([A]_{\{1,2\}})=a_{11} a_{22} -a_{12}^2=0$, then  $T^1([A + t \dot{A}]_{\{1,2\}}) =0$ for all $t$. Hence all $(D,-D)$ in the linear space generated by $(D_i,-D_i)$  in~\eqref{eqn:1-trivial-motions-for-size-2-psd-factorizations} are 1-trivial.

To show that there are no further 1-trivial motions, we consider the size-$2$ psd factorization given by the factors
    \begin{align*}
        A^{(1)} &= \begin{pmatrix}
            1 & 0 \\
            0 & 0
        \end{pmatrix}, \quad A^{(2)} = \begin{pmatrix}
            \frac{1}{4} & -\frac{1}{4} \\
            -\frac{1}{4} & \frac{1}{4}
        \end{pmatrix}, \quad A^{(3)} = \begin{pmatrix}
            0 & 0 \\
            0 & 1
        \end{pmatrix}, \\
        B^{(1)} &= \begin{pmatrix}
            \frac{1}{4} & \frac{3}{4} \\
            \frac{3}{4} & \frac{9}{4}
        \end{pmatrix}, \quad B^{(2)} = \begin{pmatrix}
            \frac{1}{4} & -\frac{1}{4} \\
            -\frac{1}{4} & \frac{1}{4}
        \end{pmatrix}, \quad B^{(3)} = \begin{pmatrix}
            1 & \frac{1}{4} \\
            \frac{1}{4} & \frac{1}{16}
        \end{pmatrix}.
    \end{align*}
One can explicitly compute that the set of its $1$-infinitesimal motions is precisely equal to the linear subspace spanned by the matrices in~\eqref{eqn:1-trivial-motions-for-size-2-psd-factorizations}. 
\end{proof}
We believe that  $L_{(I,I)} \M_{\operatorname{GL}(k)}$ at $(I,I)$ coincides with the set of $1$-trivial motions for any $k$. 
\begin{conjecture}
The set of $1$-trivial motions of size-$k$ psd factorizations of matrices in $\mathcal{M}_{ \binom{k+1}{2},k}^{p \times q}$ is equal to $L_{(I,I)}\M_{\operatorname{GL}(k)}$.
\end{conjecture}

\subsection{$k$-Trivial motions}
\label{sec:k-trivial-motions}

In this subsection we prove that each $k$-trivial motion of
a size-$k$ psd factorization is given by a matrix of the form $D = dI$. First we show the validity of a simple claim.
\begin{lemma} \label{lemma:diag-eq}
The inequalities
\begin{align*}
(D_{ii}+D_{jj}-2D_{kk})t+(D_{ii} D_{jj} - D_{kk}^2)t^2 &\geq 0,\\
-(D_{ii}+D_{jj}-2D_{kk})t+(D_{ii} D_{jj} - D_{kk}^2)t^2 & \geq 0
\end{align*}
where $t \in [0,\varepsilon)$ for some $\varepsilon>0$ imply that $D_{ii}=D_{jj}=D_{kk}$.
\end{lemma}

\begin{proof} 
Adding the inequalities gives us 
$2(D_{ii} D_{jj} - D_{kk}^2)t^2 \geq 0$, and  we conclude that $D_{ii} D_{jj} - D_{kk}^2 \geq 0$.
Factoring out $t$ from the two inequalities, we  see 
that  
$$-(D_{ii} D_{jj} - D_{kk}^2)t \leq  D_{ii}+D_{jj}-2D_{kk}\leq (D_{ii} D_{jj} - D_{kk}^2)t.$$ Since this must hold for all $t\in (0,\varepsilon)$ we conclude that $D_{ii}+D_{jj}-2D_{kk}=0$, so \begin{equation}\label{eq:1}
    D_{kk} = \frac{D_{ii}+D_{jj}}{2}.
\end{equation}
Making the above substitution for $D_{kk}$ in $D_{ii} D_{jj} - D_{kk}^2 \geq 0$ gives us $D_{ii} = D_{jj}$. Now together with equation \eqref{eq:1} we get the result.
\end{proof}

\begin{theorem}\label{thm:trivial-general-case}
The set of $k$-trivial motions of size-$k$ psd factorizations of matrices in $\M_{\binom{k+1}{2},k}^{p \times q}$ is given by $\binom{k+1}{2} \times \binom{k+1}{2}$ diagonal matrices $D=dI$. 
\end{theorem}

\begin{proof}
Example \ref{ex:trivial-motions} implies that $D=dI$ gives $k$-trivial motions. For the converse, we will construct a size-$k$ psd factorization of a particular matrix in $\M_{\binom{k+1}{2},k}^{p \times q}$ and show that the $k$-infinitesimal motions for this factorization are given by $D=dI$. Note that $p,q \geq \binom{k+1}{2}$ since the rank of the matrices in $\M_{\binom{k+1}{2},k}^{p \times q}$ is $\binom{k+1}{2}$. 
By Lemma \ref{lem:sub-facorization-rigid}, it is enough to construct a matrix in $\M_{\binom{k+1}{2}, k}^{p' \times q'}$ where 
$p'= q' = \binom{k+1}{2}$ with an $k$-infinitesimally rigid factorization, since we can embed such matrix as a submatrix of $M \in \M_{\binom{k+1}{2},k}^{p \times q}$
and extend the $k$-infinitesimally rigid factorization to one of $M$. Hence without loss of generality we can assume
$p=q=\binom{k+1}{2}$.

Now we let  $A^{(i)}=B^{(i)}$, $i \in [k]$, to be the $k \times k$ matrices with the $(i,i)$ entry equal to $1$, and all other entries equal to $0$. 
For $1 \leq i < j \leq k$,  we define $A^{(i,j)}$ to be  the $k \times k$ matrices with the $(i,i)$ and $(j,j)$ entries equal to $1$, the $(i,j)$ and $(j,i)$ entries equal to $-1$, and all other entries equal to $0$. Similarly, we define $B^{(i,j)}$ to be the $k \times k$ matrices with the $(i,i)$, $(j,j)$, $(i,j)$ and $(j,i)$ entries equal to $1$, and all other entries equal to $0$. All these matrices are symmetric psd, so 
they form a size-$k$ psd factorization of the matrix $\A \B =M\in \M_{\binom{k+1}{2},k}^{p \times q}$. Moreover, both $\A$ and $\B$ are square triangular matrices with nonzero diagonal entries. Hence they have full rank.

By definition 
\begin{equation*}
    \dot \A = \begin{pmatrix}
        \rowvec{\underline{\dot a}^{(1)T}}\\
        \vdots \\
    \rowvec{\underline{\dot a}^{(k)T}}\\
   \rowvec{\underline{\dot a}^{(1,2)T}} \\
     \vdots  \\
    \rowvec{\underline{\dot a}^{(k-1,k)T}} \\
    \end{pmatrix}, \text{ and }\dot \B = \begin{pmatrix}
\vert & & \vert & \vert & & \vert  \\
\underline{\dot b}^{(1)} & \cdots &\underline{\dot b}^{(k)} & \underline{\dot b}^{(1,2)} & \cdots & \underline{\dot b}^{(k-1,k)} \\
\vert & & \vert & \vert & & \vert
\end{pmatrix},
\end{equation*}
and $\vec(\dot A^{(i)})^T = (D_{i1}, \dots, D_{ik}, \dots, D_{i \binom{k+1}{2}})$ is the $i$th row of matrix $D$. Therefore
\begin{align*}
    A^{(i)} + t \dot A^{(i)} = \begin{pmatrix}
        D_{i1}t & \cdots & \cdots & \cdots & \frac{1}{\sqrt{2}}D_{i(2k-1))}t \\
          \vdots & \ddots  & &  &\vdots \\
        \vdots & & 1 + D_{ii}t & & \vdots \\
        \vdots & & & \ddots & \vdots\\
        \frac{1}{\sqrt{2}}D_{i(2k-1)}t & \cdots & \cdots & \cdots & D_{ik}t
    \end{pmatrix}.
\end{align*}
Similarly, $\text{vec}(\dot B^{(i)}) = (-D_{1i}, \dots, -D_{ki}, \dots, -D_{\binom{k+1}{2}}i)$,
and 
\begin{align*}
    B^{(i)} + t \dot B^{(i)} = \begin{pmatrix}
        -D_{1i}t & \cdots & \cdots & \cdots & -\frac{1}{\sqrt{2}}D_{(2k-1)i)}t \\
          \vdots & \ddots  & &  &\vdots \\
        \vdots & & 1 - D_{ii}t & & \vdots \\
        \vdots & & & \ddots & \vdots\\
        -\frac{1}{\sqrt{2}}D_{(2k-1)i}t & \cdots & \cdots  & \cdots & -D_{ki}t
    \end{pmatrix}.
\end{align*}
In order for the matrices $A^{(i)} + t \dot A^{(i)}$ and $B^{(i)} + t \dot B^{(i)}$ to be positive semidefinite, all their principal minors need to be nonnegative. These include the diagonal entries, and we get the following two sets of nonnegativity conditions: 
\[
D_{i1}t \geq 0, \ D_{i2}t \geq 0,\ \dots, \ D_{i(i-1)}t \geq 0, \ D_{i(i+1)}t \geq 0, \dots , \ D_{ik}t \geq 0,
\]  
\[
-D_{1i}t \geq 0, \ -D_{2i}t \geq 0,\ \dots, \ -D_{(i-1)i}t \geq 0, \ -D_{(i+1)i}t \geq 0, \dots , \ -D_{ki}t \geq 0.
\] 
Now consider these conditions for all $i \in [k]$. 
In particular, we have for all pairs $(i',j') \in [k] \times [k], i'\neq j'$, the conditions $D_{i'j'}t \geq 0$ and $-D_{i'j'}t \geq 0$, for some $t \in [0, \varepsilon)$, $\varepsilon > 0$. 
This implies that $D_{i'j'} = 0$. 
Therefore, the top left $k \times k$ block of matrix $D$ is diagonal.

Consider again $i \in [k]$. 
Denote by $[A^{(i)} + t \dot A^{(i)}]_{\{i',j'\}}$ the $2 \times 2$ principal minor of the matrix $A^{(i)} + t \dot A^{(i)}$ obtained from the rows and columns indexed by $\{i',j'\}$. 
From the above it follows that for  $i'= 1, \dots, k-1$ and $j' = i'+1, \dots , k$, we get the  non-negativity conditions
\begin{align*}
    -\frac{1}{2}D_{i(k+1)}^2t^2 \geq 0,\ \dots , \ -\frac{1}{2}D_{i\binom{k+1}{2}}^2t^2 \geq 0.
\end{align*}
This implies that $D_{i(k+1)} = 0, \ \ldots, D_{i\binom{k+1}{2}} = 0$.
Considering the same principal minors of the matrices $B^{(j)} + t \dot B^{(j)}$ for $j \in [k]$ gives similar non-negativity conditions
\begin{align*}
    -\frac{1}{2}D_{(k+1)i}^2t^2 \geq 0,\ \dots , \ -\frac{1}{2}D_{\binom{k+1}{2}i}^2t^2 \geq 0,
\end{align*}
from which it follows that  $D_{(k+1)i} = 0, \ldots, D_{\binom{k+1}{2}i} = 0$. From these conditions we conclude that 

\[
D = \begin{pmatrix}
D_{11} & \cdots & 0 & 0 & \cdots & 0 \\

\vdots & \ddots & \vdots & \vdots & \ddots & \vdots \\

0 & \cdots & D_{kk} & 0 & \cdots & 0 \\ 

0 & \cdots & 0 & D_{(k+1)(k+1)} & \cdots & D_{(k+1)\binom{k+1}{2}} \\

\vdots & \ddots & \vdots & \vdots & \ddots & \vdots \\ 
 
0 & \cdots & 0 & D_{\binom{k+1}{2}(k+1)} & \cdots & D_{\binom{k+1}{2}\binom{k+1}{2}}
\end{pmatrix}.
\]

Consider now the matrices $B^{(i,j)}+t \dot B^{(i,j)}$, for $i=1, \dots, k-1$ and $j = i+1, \dots, k$. 
From the above we see that they are the form
\begin{equation*}
    \begin{pNiceArray}{c@{}cw{c}{0em}cccw{c}{0em}c@{}c}[xdots/shorten=1em, xdots/inter=0.6em]
    
    0 & -D_{(k+1)m}t & \Cdots & & & & & & -D_{(2k-1)m}t \\
    
    -D_{(k+1)m}t & 0 & & & & & & & -D_{(3k-3)m}t \\
    
    \Vdots & & \ddots & & & & & & \\
    
    & & & 1 -D_{ii}t & \cdots & 1-D_{mm}t & & & \\
    
    & & & \vdots & \ddots & \vdots & & & \Vdots \\
    
    & & & 1-D_{mm}t & \cdots & 1 -D_{jj}t & & & \\
    
    & & & & & & \ddots & & \\
    
    -D_{(3k-3)m}t & & & & & & & 0 & -D_{\binom{k+1}{2}m}t \\
    
    -D_{(2k-1)m}t & \Cdots & & & & & & -D_{\binom{k+1}{2}m}t & 0 \\
    
    \end{pNiceArray}
\end{equation*}
where $m = ik+j-\binom{i+1}{2}$.
For $l = 1, \dots, \check{i}, \dots, k-1$ and $n = l+1, \dots, \check{j}, \dots, k$, 
$
[B^{(i,j)} + t \dot B^{(i,j)}]_{\{l,n\}} = -(D_{Lm}t)^2,
$
where $L = lk + n - \binom{l+1}{2}$.
We have that $k+1 = k+2- \binom{1+1}{2} \leq m \leq (k-1)k + k + \binom{(k-1)+1}{2} = \frac{k(k+1)}{2} = \binom{k+1}{2}$. Similarly $k + 1 \leq L \leq \binom{k+1}{2}$. In particular, $D_{Lm} = 0$. Therefore the off-diagonal entries of the bottom right $\binom{k}{2} \times \binom{k}{2}$ block of matrix $D$ are zero, that is, the matrix $D$ must be a diagonal matrix.

Now because
\begin{align*}
    [B^{(i,j)} + t \dot B^{(i,j)}]_{\{i,j\}} & = (1 -D_{ii}t)(1 -D_{jj}t) - (1-D_{mm}t)^2 \\ & = (-D_{ii}-D_{jj}+2D_{mm})t + (D_{ii}D_{jj}-D_{mm}^2)t^2, 
\end{align*}
and similarly 
\begin{align*}
    [A^{(i,j)} + t \dot A^{(i,j)}]_{\{i,j\}} & = (1 +D_{ii}t)(1 + D_{jj}t) - (1 + D_{mm}t)^2 \\ & = (D_{ii}+D_{jj}-2D_{mm})t + (D_{ii}D_{jj}-D_{mm}^2)t^2, 
\end{align*}
 by \Cref{lemma:diag-eq} it follows that $D_{ii} = D_{jj} = D_{mm}$. 
Thus $D = dI$.
\end{proof}

We note that our knowledge of $s$-trivial motions for $1<s<k$ is non-existent. While what we have above is sufficient for our case of $k=2$, it would be interesting to study the general case further.

\section{Further preliminary results} \label{sec:prelim-results}
In this section we include a few more preliminary results that will help us in the subsequent sections.

We remark that if $\Tilde{A}^{(i)}$ and $\Tilde{B}^{(j)}$ are obtained from a   size-$k$  psd factorization consisting of matrices $A^{(i)}$ and $B^{(j)}$ via the $\operatorname{GL}(k)$-action then $\rank(\Tilde{A}^{(i)}) = \rank(A^{(i)})$ and
$\rank(\Tilde{B}^{(j)}) = \rank(B^{(j)})$. In particular, if we have a psd factorization with 
$\rank(A^{(i)}) = \rank(B^{(j)}) = 1$
for all $i\in [p]$ and $j \in [q]$ 
then the action of $\operatorname{GL}(k)$ 
will give   a size-$k$ psd factorization with rank-one factors.   
In fact, if $A^{(i)} = a_ia_i^T$ and $B^{(j)} = b_jb_j^T$ then
$\Tilde{A}^{(i)} = \Tilde{a}_i\Tilde{a}_i^T$ and $\Tilde{B}^{(j)} = \Tilde{b}_j\Tilde{b}_j^T$
where $\Tilde{a}_i = S^Ta_i$ and $\Tilde{b}_j = S^{-1}b_j$.

\begin{proposition}\label{prop:inf-motion-rotation}
Let a size-$2$ psd factorization of  $M \in \R_+^{p \times q}$  be given by the matrices $A^{(1)}, \dots, A^{(p)}$,  $B^{(1)}, \dots, B^{(q)} \in \mathcal{S}_+^2$, and let $\dot A^{(1)}, \dots, \dot A^{(p)}, \dot B^{(1)}, \dots, \dot B^{(q)} \in \mathcal{S}^2$ 
be an $s$-infinitesimal motion of the psd factorization for $s=1$ or $s=2$. Define $\dot {\Tilde{A}}^{(i)} = S^T \dot A^{(i)} S$ and $\dot {\Tilde{B}}^{(j)} = S^{-1} \dot B^{(j)} S^{-T}$.  For any $S \in \operatorname{GL}(2)$, $\dot{\Tilde{A}}^{(1)}, \dots, \dot{\Tilde{A}}^{(p)}, \dot{\Tilde{B}}^{(1)}, \dots, \dot{\Tilde{B}}^{(q)}$ 
is also an $s$-infinitesimal motion for the factorization given by the matrices 
$\Tilde{A}^{(1)}, \dots, \Tilde{A}^{(p)}, \Tilde{B}^{(1)}, \dots, \Tilde{B}^{(q)}$.
\end{proposition}
\begin{proof}
We have  $\tr(\dot{\Tilde A }^{(i)} \Tilde{B}^{(j)}) = \tr(S^T \dot A^{(i)} B^{(j)}S^{-T}) = \tr(\dot A^{(i)} B^{(j)})$, and similarly
$\tr(\Tilde{A}^{(i)} \dot{\Tilde{B}}^{(j)}) = \tr(A^{(i)} \dot{B}^{(j)})$. Therefore
$$ \langle \dot{\Tilde{A}}^{(i)},\Tilde{B}^{(j)}  \rangle + \langle \Tilde{A}^{(i)},\dot{\Tilde{B}}^{(j)}  \rangle = \langle \dot{A}^{(i)},B^{(j)}  \rangle + \langle A^{(i)},\dot{B}^{(j)}  \rangle = 0.$$
When $s=1$, 
\begin{align*}
T^1([S^T (A^{(i)} + t\dot{A}^{(i)})S]_{\{\ell \}}) &=\begin{pmatrix}
s_{1\ell} & s_{2\ell}
 \end{pmatrix}
[A^{(i)} + t\dot{A}^{(i)}]
  \begin{pmatrix}
s_{1\ell} \\ 
s_{2\ell}
 \end{pmatrix},\\
T^1([S^T (A^{(i)} + t\dot{A}^{(i)})S]_{\{1,2\}}) & = T^1([A^{(i)} + t\dot{A}^{(i)}]_{\{1,2\}})\cdot (s_{11}s_{22}-s_{12}s_{21})^2.
\end{align*}
In the second case, since $(s_{11}s_{22}-s_{12}s_{21})^2 \geq 0$, then $T^1([A^{(i)} + t\dot{A}^{(i)}]_{\{1,2\}}) \geq 0$ implies $T^1([S^T (A^{(i)} + t\dot{A}^{(i)})S]_{\{1,2\}}) \geq 0$. In the first case, if $A^{(i)}$ is positive definite (rank two), then for $t>0$ small enough, $A^{(i)} + t\dot{A}^{(i)}$ is positive definite, and hence $T^1([S^T (A^{(i)} + t\dot{A}^{(i)})S]_{\{\ell\}})$ is a quadratic form in $s_{1\ell},s_{2\ell}$, and therefore nonnegative for all values of $s_{1\ell},s_{2\ell}$. If  $A^{(i)}$ is of rank one, then without loss of generality, we can assume that $a_{11} \neq 0$ and $a_{22}=\frac{a_{12}^2}{a_{11}}$. The kernel of $A^{(i)}$ is spanned by the vector $\begin{pmatrix}a_{12} \\ -a_{11} \end{pmatrix}$. If $\begin{pmatrix}s_{1\ell} \\ s_{2\ell} \end{pmatrix}$ is not in the kernel of $A^{(i)}$, then 
$$
\begin{pmatrix}
s_{1\ell} & s_{2\ell}
 \end{pmatrix}
A^{(i)}
  \begin{pmatrix}
s_{1\ell} \\ 
s_{2\ell}
 \end{pmatrix} >0
$$
and hence $T^1([S^T (A^{(i)} + t\dot{A}^{(i)})S]_{\{\ell\}}) \geq 0$ for $t>0$ small enough. If $\begin{pmatrix}s_{1\ell} \\ s_{2\ell} \end{pmatrix}$ is in the kernel of $A^{(i)}$, then we can write $\begin{pmatrix}s_{1\ell} \\ s_{2\ell} \end{pmatrix}=k\begin{pmatrix}a_{12} \\ -a_{11} \end{pmatrix}$. With this substitution, we get
$$
T^1([S^T (A^{(i)} + t\dot{A}^{(i)})S]_{\{\ell\}}) = a_{11}k^2 T^1([A^{(i)} + t\dot{A}^{(i)}]_{\{1,2\}}),
$$
and $T^1([S^T (A^{(i)} + t\dot{A}^{(i)})S]_{\{\ell\}})$ is therefore nonnegative. Similarly, if $T^1([B^{(j)} + t\dot{B}^{(j)}]_{J}) \geq 0$, then $T^1([S^T (B^{(j)} + t\dot{B}^{(j)})S]_J) \geq 0$.

When $s=2$, $A^{(i)} + t\dot{A}^{(i)}$ and $B^{(j)} + t\dot{B}^{(j)}$ are psd for $t \in [0, \varepsilon)$. Therefore $S^T (A^{(i)} + t\dot{A}^{(i)})S = \Tilde{A}^{(i)} + t \dot{\Tilde{A}}^{(i)}$ is also psd as well as
$\Tilde{B}^{(j)} + t \dot{\Tilde{B}}^{(j)} \in \mathcal{S}_+^2$.  
\end{proof}

\begin{corollary} \label{cor:bijection-between-inf-motions}
Let a size-$2$ positive semidefinite factorization of  $M \in \M_{3,2}^{p\times q}$  be given by the matrices $A^{(1)}, \dots, A^{(p)}, B^{(1)}, \dots, B^{(q)} \in \mathcal{S}_+^2$. For any $S \in \operatorname{GL}(2)$,
the set of all $s$-infinitesimal motions of this factorization
and
the set of all $s$-infinitesimal motions of 
$(\Tilde{A}^{(1)}, \dots, \Tilde{A}^{(p)}, \Tilde{B}^{(1)}, \dots, \Tilde{B}^{(q)})$ are linearly isomorphic for
$s=1,2$. This linear isomorphism restricts to the identity isomorphism on $2$-trivial motions. 
\end{corollary}
\begin{proof} 
The matrix $S$ induces an invertible $3\times 3$ matrix 
$\mathcal{S}$ such that 
$\Tilde{\mathcal{A}} = \mathcal{A} \mathcal{S}$ and
$\Tilde{\mathcal{B}} = \mathcal{S}^{-1}\mathcal{B}$. 
If $D$ gives an $s$-infinitesimal motion, then 
$\dot{\mathcal{A}} = \mathcal{A} D$ and $\dot{\mathcal{B}} = -D \mathcal{B}$.  
In  \Cref{prop:inf-motion-rotation} we defined  $\dot{\Tilde{\mathcal{A}}} = \dot{\mathcal{A}} \mathcal{S}$ and 
$\dot{\Tilde{\mathcal{B}}} = \mathcal{S}^{-1} \dot{\mathcal{B}}$. These give us
$\dot{\Tilde{\mathcal{A}}} = \mathcal{A}D\mathcal{S}= \mathcal{A} \mathcal{S} (\mathcal{S}^{-1}D\mathcal{S}) = \Tilde{\mathcal{A}}(\mathcal{S}^{-1}D\mathcal{S})$
and 
$\dot{\Tilde{\mathcal{B}}} = -\mathcal{S}^{-1}D\mathcal{B}=  -(\mathcal{S}^{-1}D\mathcal{S})\mathcal{S}^{-1}\mathcal{B} = -(\mathcal{S}^{-1}D\mathcal{S}) \Tilde{B}$. The claimed linear isomorphism is given by $D \mapsto \mathcal{S}^{-1}D\mathcal{S}$. By \Cref{thm:trivial-general-case}, $2$-trivial motions are given by $D = dI_3$, and these commute with $\mathcal{S}$. This proves the last statement.  
\end{proof}

The next result will be useful when 
we analyze size-$2$ psd factorizations of matrices in $\M_{3,2}^{p\times q}$. 
We note that in this situation
$$\det(A^{(i)} + t {\dot A}^{(i)} ) = \det(A^{(i)}) + t \alpha_i^A(D) + t^2 \beta_i^A(D)$$
where $\alpha_i^A(D)$ is a linear function of $D$ and $\beta_i^A(D)$ is 
a quadratic function of $D$. In particular, when $\rank(A^{(i)}) \leq 1$ 
then the same determinant simplifies
to $\det(A^{(i)} + t {\dot A}^{(i)} ) = t \alpha_i^A(D) + t^2 \beta_i^A(D)$.

\begin{lemma}\label{prop:zero-alpha-forces-beta}
Let $M \in \M_{3,2}^{p\times q}$ and let 
 $A^{(i)}\in \mathcal{S}^2_+$ be a factor of a size-$2$ psd factorization of $M$
 where $\rank(A^{(i)}) = 1$. 
If a matrix $D \in \R^{3 \times 3}$ gives a 2-infinitesimal motion such that $\alpha_i^A(D)=0$, then $\beta_i^A(D) = 0$.
\end{lemma}
\begin{proof}
    Let $A^{(i)} = a_ia_i^T$ where $a_i = (a_{i1},a_{i2})$. 
    By using Proposition \ref{prop:inf-motion-rotation} with an appropriate $S \in \operatorname{GL}(2)$  we can assume that $a_{i1} \neq 0$ and $a_{i2} \neq 0$. We note that in this case 
    $$ \det(\tilde{A}^{(i)} + t \tilde{\dot A}^{(i)} ) = \det(S^T(A^{(i)} + t {\dot A}^{(i)})S) = \det(S)^2 \det(A^{(i)} + t {\dot A}^{(i)}).$$ This means that ${\tilde \alpha}_i^A$ and ${\tilde \beta}_i^B$ have the same sign as $\alpha_i^A$ and $\beta_i^A$, respectively. Hence our assumption is justified. 
    Consider the $8$-dimensional subspace defined by the equation $\alpha_i^A(D)=0$. A basis for this subspace is given by 
    \begin{align*}
    &\left(2,0,0,0,0,0,0,0,1\right), 
        \left(-\frac{\sqrt{2} a_{i1}}{a_{i2}},0,0,0,0,0,0,1,0\right), 
        \left(-\frac{\sqrt{2} a_{i2}}{a_{i1}},0,0,0,0,0,1,0,0\right),\\
     &   \left(\frac{\sqrt{2} a_{i2}}{a_{i1}},0,0,0,0,1,0,0,0\right),
    \left(-1,0,0,0,1,0,0,0,0\right),
        \left(-\frac{a_{i2}^2}{a_{i1}^2},0,0,1,0,0,0,0,0\right),\\
      & \left(\frac{\sqrt{2} a_{i1}}{a_{i2}},0,1,0,0,0,0,0,0\right),
        \left(-\frac{a_{i1}^2}{a_{i2}^2},1,0,0,0,0,0,0,0\right).
        \end{align*}
    Now if $D$ corresponds to a $2$-infinitesimal motion where $\alpha_i^A(D)=0$, then $\underline{D}$ is a linear combination of these vectors, 
    \[\underline{D} = \left(-\frac{s_8 a_{i1}^2}{a_{i2}^2}+\frac{\sqrt{2} s_7 a_{i1}}{a_{i2}}-\frac{\sqrt{2} s_2 a_{i1}}{a_{i2}}-\frac{s_6 a_{i2}^2}{a_{i1}^2}-\frac{\sqrt{2} s_3 a_{i2}}{a_{i1}}+\frac{\sqrt{2} s_4 a_{i2}}{a_{i1}}+2 s_1-s_5,s_8, \ldots,s_1\right)\] where $s_1,\ldots,s_8 \in \R$. In this case, $\beta_i^A(D)$ is a quadratic form in $s_1, \ldots, s_8$ represented by the symmetric matrix 
    \[\left(
\begin{array}{cccccccc}
 -a_{i1}^2 a_{i2}^2 & \sqrt{2} a_{i1}^3 a_{i2} & 0 & -\frac{a_{i1} a_{i2}^3}{\sqrt{2}} & a_{i1}^2 a_{i2}^2 & 0 & -\frac{a_{i1}^3 a_{i2}}{\sqrt{2}} & a_{i1}^4 \\
 \sqrt{2} a_{i1}^3 a_{i2} & -2 a_{i1}^4 & 0 & a_{i1}^2 a_{i2}^2 & -\sqrt{2} a_{i1}^3 a_{i2} & 0 & a_{i1}^4 & -\frac{\sqrt{2} a_{i1}^5}{a_{i2}} \\
 0 & 0 & 0 & 0 & 0 & 0 & 0 & 0 \\
 -\frac{a_{i1} a_{i2}^3}{\sqrt{2}} & a_{i1}^2 a_{i2}^2 & 0 & -\frac{a_{i2}^4}{2} & \frac{a_{i1} a_{i2}^3}{\sqrt{2}} & 0 & -\frac{1}{2} a_{i1}^2 a_{i2}^2 & \frac{a_{i1}^3 a_{i2}}{\sqrt{2}} \\
 a_{i1}^2 a_{i2}^2 & -\sqrt{2} a_{i1}^3 a_{i2} & 0 & \frac{a_{i1} a_{i2}^3}{\sqrt{2}} & -a_{i1}^2 a_{i2}^2 & 0 & \frac{a_{i1}^3 a_{i2}}{\sqrt{2}} & -a_{i1}^4 \\
 0 & 0 & 0 & 0 & 0 & 0 & 0 & 0 \\
 -\frac{a_{i1}^3 a_{i2}}{\sqrt{2}} & a_{i1}^4 & 0 & -\frac{1}{2} a_{i1}^2 a_{i2}^2 & \frac{a_{i1}^3 a_{i2}}{\sqrt{2}} & 0 & -\frac{a_{i1}^4}{2} & \frac{a_{i1}^5}{\sqrt{2} a_{i2}} \\
 a_{i1}^4 & -\frac{\sqrt{2} a_{i1}^5}{a_{i2}} & 0 & \frac{a_{i1}^3 a_{i2}}{\sqrt{2}} & -a_{i1}^4 & 0 & \frac{a_{i1}^5}{\sqrt{2} a_{i2}} & -\frac{a_{i1}^6}{a_{i2}^2} \\
\end{array}
\right).\]
This matrix has rank one  with a single non-zero eigenvalue $-\frac{\left(a_{i1}^2+a_{i2}^2\right)^2 \left(2 a_{i1}^2+a_{i2}^2\right)}{2 a_{i2}^2} <0$. Therefore it  is negative semidefinite, and $\beta_i^A(D) \leq 0$ for any  $D$ such that $\alpha_i^A(D)=0$. On the other hand, because $D$ gives a 2-infinitesimal motion, we have $\det(A^{(i)}+t\dot{A}^{(i)})=t^2\beta_i^A(D) \geq 0$ and hence $\beta_i^A(D) = 0$.
\end{proof}

\begin{lemma} \label{lem:rank-2-factors-do-not-matter}
 
Let $M \in \M_{3,2}^{p \times q}$ and let a size-$2$ psd factorization of $M$ be given by the factors $A^{(i)} \in \mathcal{S}_+^2, \, i\in [p]$, and $B^{(j)} \in \mathcal{S}_+^2, \, j \in [q]$. Then the  psd factorization of $M$  is $s$-infinitesimally rigid if and only if the subset of matrices $D \in \mathbb{R}^{3 \times 3}$ where $\dot \A = \A D$ and $\dot \B = -D \B$ are  satisfying the condition (2) in~\Cref{def:psd-infinitesimal} for rank-one factors is precisely the set of $s$-trivial motions. 
\color{black}
\end{lemma}

\begin{proof}
 
Assume that the subset of matrices $D \in \mathbb{R}^{3 \times 3}$ satisfying the condition (2) in~\Cref{def:psd-infinitesimal} for rank-one factors is precisely the set of $s$-trivial motions. Then 
the same is true if we consider the condition (2) for all factors. This means that the original factorization is $s$-infinitesimally rigid. Conversely, assume that the set defined in~\Cref{def:psd-infinitesimal} where in condition (2) we consider only rank-one factors $A^{(i_1)},\ldots,A^{(i_s)},B^{(j_1)},\ldots,B^{(j_t)}$ contains $s$-infinitesimal motions that are not $s$-trivial motions. We claim that the remaining factors of rank two do not introduce any further constraints on matrices $D$ representing $s$-infinitesimal motions: If $A^{(i)}$ has rank two, then it is positive definite. Hence its determinant and diagonal entries are positive. By \cref{rmk:only-zeros-matter}, the inequality $T^1([A^{(i)} + t \dot A^{(i)} ]_{I}) \geq 0$ is satisfied for all $I \subseteq [2]$, $t \in [0,\varepsilon)$ and $\varepsilon > 0$ small enough.
The same is true if a factor $B^{(j)}$ has rank two.
\color{black}
\end{proof}

\section{$s$-Infinitesimal rigidity in $\M_{3,2}^{p\times q}$} \label{sec:infinitesimal-rigidity}

In the rest of the paper, we will consider 
matrices $M \in \R_+^{p \times q}$ with 
rank $3$ and psd rank $2$.  In this section, we first characterize $1$- and $2$-infinitesimally rigid size-2 psd factorizations for matrices with positive entries (\Cref{sec:no-orthogonal-pairs}). 
Then we characterize $2$-infinitesimally rigid size-2 psd factorizations if we allow zeros in $M$ (\Cref{sec:one-orthogonal-pair} and \Cref{sec:two-orthogonal-pairs}). The main results for the corresponding cases are \Cref{thm:no-orthogonal-pairs}, \Cref{thm:one-orthogonal-pair}, and \Cref{thm:two-orth-inf-rigid}. 

The reason why we do not consider $1$-infinitesimal rigidity for matrices with zeros is that being $1$-infinitesimally rigid is very restrictive in the presence of zeros. For example, it can be seen from a dimension count that a $3 \times 3$ matrix with a zero cannot have a $1$-infinitesimally rigid size-$2$ psd factorization. Nevertheless such matrices can have a unique size-$2$ psd factorization up to $\text{GL}(2)$-action, which is the case for the derangement matrix
\[
\begin{pmatrix}
0 & 1 & 1\\
1 & 0 & 1\\
1 & 1 & 0
\end{pmatrix}
\]
as discussed in~\cite[Example 11]{Fawzi:2015}. The unique size-$2$ psd factorization of this matrix is, however, $2$-infinitesimally rigid.

In our characterization, the ranks of the matrices in a psd factorization will play a significant role. In particular, we will pay close attention to the subset of the factors that have rank one instead of the full rank two. 

  This is because for determining $s$-infinitesimal rigidity, we can consider in the second condition of $s$-infinitesimal motions only rank-one factors according to \cref{lem:rank-2-factors-do-not-matter}. 
If $A^{(1)}$, $\dots$, $A^{(\bar{p})}$, $B^{(1)}$, $\dots$, $B^{(\bar{q})} \in \mathcal{S}^2_+$ are matrices of rank one, then $A^{(i)} = a_i a_i^T$ and $B^{(j)} = b_j b_j^T$ for $a_i, b_j \in \R^2$. Besides the rank of these matrices, the pairwise linear independence and orthogonality of these factors also play an important role for deriving our results. We note that the linear independence of $a_i,a_j$ is equivalent to $\det(a_i,a_j) \neq 0$, the linear independence of $b_i,b_j$ is equivalent to $\det(b_i,b_j) \neq 0$, and the non-orthogonality of $a_i,b_j$ is equivalent to $\langle a_i,b_j \rangle \neq 0$. Furthermore, $\langle A^{(i)}, B^{(j)} \rangle = (a_i^Tb_j)^2$,  and hence $\langle A^{(i)}, B^{(j)} \rangle \neq 0$ if and only if $\langle a_i, b_j \rangle \neq 0$.
If $\langle a_i, b_j \rangle = 0$ we say that the factors $A^{(i)}$ and $B^{(j)}$ are \textit{orthogonal}.

\begin{remark}\label{rmk:rank-1-factors-determinant-form}
    Let $A^{(i)}$ and $B^{(j)}$ be factors of rank one of a size-2 psd factorization as above, and let $D$ be a matrix corresponding to a $2$-infinitesimal motion.
    As we observed in Section \ref{sec:prelim-results} 
    $$
        \det(A^{(i)} + t \dot A^{(i)})  = t\alpha_i^A(D)  + t^2\beta_i^A(D) \mbox{  and   } 
        \det (B^{(j)} + t \dot B^{(j)}) = t\alpha_j^B(D) + t^2\beta_j^B(D).
    $$
    We note that  $T^1([A^{(i)} + t \dot{A}^{(i)}]_{\{1,2\}}) = t\alpha_i^A(D)$ and $T^2([A^{(i)} + t \dot{A}^{(i)}]_{\{1,2\}}) =t\alpha_i^A(D)  + t^2\beta_i^A(D)$. Furthermore, when $S \in \operatorname{GL}(2)$ acts on the factorization, our computations in \cref{prop:inf-motion-rotation} imply that $\tilde{\alpha}_i^A$ and $\tilde{\alpha}_j^B$ are positive multiples of $\alpha_i^A$ and $\alpha_j^B$, respectively.
\end{remark}

\subsection{No orthogonal pairs} \label{sec:no-orthogonal-pairs}

This subsection is devoted to the case where no factor  $A^{(i)}$ of rank one is orthogonal to a factor  $B^{(j)}$ of rank one. This means that $M$ has only positive entries. 

\begin{lemma} \label{lem:matrix-C}
Let $A^{(1)}$, $\dots$, $A^{(\bar{p})}$, $B^{(1)}$, $\dots$, $B^{(\bar{q})} \in \mathcal{S}^2_+$ be matrices of rank one, where $A^{(i)} = a_i a_i^T$ and $B^{(j)} = b_j b_j^T$ for $a_i, b_j \in \R^2$. Define
    \begin{align*}
        \alpha_i^A & := (a_{i1}^2 a_{i2}^2, a_{i1}^4, -\sqrt{2} a_{i1}^3 a_{i2}, a_{i2}^4, a_{i1}^2 a_{i2}^2, -\sqrt{2} a_{i1} a_{i2}^3, \sqrt{2} a_{i1} a_{i2}^3, \sqrt{2} a_{i1}^3 a_{i2}, -2 a_{i1}^2 a_{i2}^2), \\
        \alpha_j^B & := (-b_{j1}^2 b_{j2}^2, -b_{j2}^4, -\sqrt{2} b_{j1} b_{j2}^3, -b_{j1}^4, -b_{j1}^2 b_{j2}^2, -\sqrt{2} b_{j1}^3 b_{j2}, \sqrt{2} b_{j1}^3 b_{j2}, \sqrt{2} b_{j1} b_{j2}^3, 2 b_{j1}^2 b_{j2}^2). \\
    \end{align*}
Here the entries of $\alpha_{i}^A$ and $\alpha_j^B$ correspond to the entries of $D_{11}, D_{12} \dots ,D_{33}$ of the matrix $D$.
Consider the $(\bar{p}+\bar{q}) \times 9$ matrix
\begin{align} \label{eqn:matrix-C}
    C_{(\mathcal{A}, \mathcal{B})} := \begin{pmatrix}
        \rowvec{\alpha_{1}^A} \\
        \vdots \\
        \rowvec{\alpha_{\bar{p}}^A}\\
        \rowvec{\alpha_{1}^B}\\
        \vdots \\
        \rowvec{\alpha_{\bar{q}}^B}
    \end{pmatrix}.
\end{align}
\begin{enumerate}
\item The rank of $C_{(\mathcal{A},\mathcal{B})}$ is at most five. 

\item The right kernel of $C_{(\mathcal{A},\mathcal{B})}$ contains the subspace $L := \operatorname{span}(v_1, v_2, v_3, v_4)$, where
\begin{equation*}
\begin{split}
    v_1 &:= (2, 0, 0, 0, 0, 0, 0, 0, 1), \\
    v_2 &:= (0, 0, 0, 0, 2, 0, 0, 0, 1), \\
    v_3 &:= (0, 0, 1, 0, 0, 0, 0, 1, 0), \\
    v_4 &:= (0, 0, 0, 0, 0, 1, 1, 0, 0).
\end{split}
\end{equation*}
If $\rank(C_{(\mathcal{A},\mathcal{B})}) = 5$, then the right kernel is equal to $L$.  

\item If $\bar{p}+\bar{q}=6$, then the left kernel of $C_{(\mathcal{A},\mathcal{B})}$ contains the vector $v \in \R^{\bar{p}+\bar{q}}$ with entries 
\begin{equation} \label{eqn:left-kernel-C}
    v_k := 
    (-1)^{k + \mathbbm{1}_{k \geq \bar{p}+1}} \prod_{\substack{1 \leq i < j \leq \bar{p},\\a_i,a_j \neq c_k}} \det(a_i,a_j) \prod_{\substack{1 \leq i < j \leq \bar{q},\\b_i,b_j \neq c_k}} \det(b_i,b_j) \prod_{\substack{1 \leq i \leq \bar{p},\\ 1 \leq j \leq \bar{q},\\a_i,b_j \neq c_k}} \langle a_i, b_j \rangle, 
\end{equation}
where $c_1=a_1,\ldots,c_{\bar{p}}=a_{\bar{p}},c_{\bar{p}+1}=b_1,\ldots,c_{\bar{p}+\bar{q}}=b_{\bar{q}}$.
If $\rank(C_{(\mathcal{A},\mathcal{B})}) = 5$, then the left kernel is spanned by $v$. 

\item If   $\bar{p}+\bar{q} \geq 5$,  all $\det(a_i,a_j) \neq 0, \det(b_i,b_j) \neq 0$ and $\langle a_i,b_j \rangle \neq 0$, then $\rank(C_{\mathcal{A}, \mathcal{B}}) = 5$.
\end{enumerate}
\end{lemma}

\begin{proof}
We have $\rank(C_{(\mathcal{A},\mathcal{B})}) \leq 5$ since the columns $1,5,9$ are multiples of each other, as are $3,8$ and $6,7$.
Since column $9$ is $-2$ times the first column of $C_{(\mathcal{A},\mathcal{B})}$, it follows that $C_{(\mathcal{A},\mathcal{B})} v_1 = 0$.
Similarly, column $9$ is $-2$ times column $5$, hence $C_{(\mathcal{A},\mathcal{B})} v_2 = 0$.
Finally, since the column $8$ is always the negative of the column $3$, and
the column $6$ is always the negative of the column $7$, it follows that $C_{(\mathcal{A},\mathcal{B})} v_3 = 0$ and $C_{(\mathcal{A},\mathcal{B})} v_4 = 0$.
Hence, $L \subseteq \ker C_{(\mathcal{A},\mathcal{B})}$. If $\rank(C_{(\mathcal{A},\mathcal{B})}) = 5$, then $\ker C_{(\mathcal{A}, \mathcal{B})}$ is $4$-dimensional, and thus $L = \ker C_{(\mathcal{A},\mathcal{B})}$.

Finally assume that $\bar{p}+\bar{q}=6$,
and consider the $6 \times 5$ submatrix $C'$ consisting of the columns $1, 2, 3, 4$ and $6$ of $C_{(\mathcal{A},\mathcal{B})}$. The vector $(-1)^{\bar{p}}(C_1', -C_2', C_3', -C_4', C_5, -C_6')$ where $C_i'$ is the $5 \times 5$ minor of $C'$ obtained by deleting row $i$ is in the left kernel of $C'$ and hence of $C_{(\mathcal{A},\mathcal{B})}$. For all possible cases of $\bar{p}$ and $\bar{q}$ one can computationally check that this vector is  equal to the coordinates~\eqref{eqn:left-kernel-C} scaled by $\frac{1}{2}$. If   $\bar{p}+\bar{q} \geq 5$ and  none of the $\det(a_i,a_j),\det(b_i,b_j)$ and $\langle a_i, b_j \rangle$ are zero, then none of the $5 \times 5$ minors vanish and hence rank of $C_{(\mathcal{A},\mathcal{B})}$ is five.
\end{proof}

\begin{lemma}\label{lem:no-orthogonal-pairs-positive-entries-make-alphas-zero}
    Let $C \in \R^{m \times n}$  and let $P = \{x \in \R^n \mid C x \geq 0 \}$.
    If there exists a $\lambda = (\lambda_1, \dots, \lambda_m)$ such that $\lambda^T C = 0$ and $\lambda_1, \dots, \lambda_m >0$, then $Cx = 0$ for all $x \in P$.
\end{lemma}

\begin{proof}
    Suppose that $\lambda = (\lambda_1, \dots ,\lambda_m)$ has strictly positive entries, and that $\lambda^TC = 0$.
    Denote by $\mathrm{c}_i$ the $i$th row of $C$.
    This implies that $\lambda_1  \mathrm{c}_1 + \lambda_2 \mathrm{c}_2 + \cdots + \lambda_m \mathrm{c}_m = 0$, that is,
    \begin{equation*}
        - \mathrm{c}_1 = \frac{1}{\lambda_1} (\lambda_2 \mathrm{c}_2 + \cdots + \lambda_m \mathrm{c}_m).
    \end{equation*}
    Let $x \in P$ be arbitrary.
    Entries of $\lambda$ being strictly positive implies that 
    \begin{equation*}
        - \mathrm{c}_1 x= \frac{1}{\lambda_1} (\lambda_2 c_2 + \cdots + \lambda_m \mathrm{c}_m)x \geq 0.
    \end{equation*}
    However, as $x \in P$, then $\mathrm{c}_1 x \geq 0$ so it must follow that $\mathrm{c}_1 x = 0$.
    With similar reasoning we may conclude that $Cx = 0$ for all $x \in P$.
\end{proof}

\begin{definition} \label{def:CAB-and-PAB}
Let $M \in \M_{3,2}^{p \times q}$ and consider a size-$2$ psd factorization of $M$ given by the factors $A^{(1)}$, $\dots$, $A^{(p)}$, $B^{(1)}$, $\dots$, $B^{(q)} \in \mathcal{S}_+^2$. Let $A^{(1)}$, $\dots$, $A^{(\bar{p})}$, $B^{(1)}$, $\dots$, $B^{(\bar{q})}$  be the factors that have rank one. We define the matrix $C_{(\mathcal{A},\mathcal{B})} \in \mathbb{R}^{(\bar{p}+\bar{q}) \times 9}$ associated to the matrices $A^{(1)}$, $\dots$, $A^{(p)}$, $B^{(1)}$, $\dots$, $B^{(q)} \in \mathcal{S}_+^2$ to be the same matrix that is defined in \cref{lem:matrix-C} for rank-1 matrices $A^{(1)}$, $\dots$, $A^{(\bar{p})}$, $B^{(1)}$, $\dots$, $B^{(\bar{q})}$. We let  $P_{(\mathcal{A},\mathcal{B})}$
be the polyhedron 
$\{\underline{D} \in \R^9 \mid C_{(\mathcal{A},\mathcal{B})} \underline{D} \geq 0 \}$.    
\end{definition}

\Cref{lem:rank-2-factors-do-not-matter} justifies using only rank-1 factors in the definitions of $C_{(\mathcal{A},\mathcal{B})}$ and $P_{(\mathcal{A},\mathcal{B})}$. 
The next lemma states that for $1$-infinitesimal motions it is enough to consider conditions arising from the determinants of the matrices $A^{(i)} + t \dot A^{(i)}$ and $B^{(j)} + t \dot B^{(j)}$.

\begin{lemma} \label{lemma:1-inf-motions-enough-to-consider-determinants}
  Let $M \in \M_{3,2}^{p \times q}$ and consider a size-$2$ psd factorization of $M$ given by factors $A^{(1)}$, $\dots$, $A^{(p)}$, $B^{(1)}$, $\dots$, $B^{(q)} \in S^2_+$. Let $A^{(1)}$, $\dots$, $A^{(\bar{p})}$, $B^{(1)}$, $\dots$, $B^{(\bar{q})}$  be the factors that have rank one. 
     The cone $P_{(\A,\B)}$ is equal to the set of all $1$-infinitesimal motions. In particular,  
    if $\underline{D} \in P_{(\A, \B)}$, then $T^{1}([A^{(i)} + t \dot A^{(i)}]_{\{ \ell \}}) \geq 0$ and $T^{1}([B^{(j)} + t \dot B^{(j)} ]_{\{\ell\}}) \geq 0$ for $\ell \in \{1,2\}$ and for all $i \in [\bar{p}]$ and $j \in [\bar{q}]$.
\end{lemma}
    
\begin{proof}
       A vector  $\underline{D}$ corresponds to a $1$-infinitesimal motion   if and only if there exists $\varepsilon >0$ such that for $t \in [0,\varepsilon)$,  $T^1([A^{(i)}+t \dot A^{(i)}]_{\{1,2\}}) = 
     \alpha_i^A(D)t \geq 0$ for all $i \in [\bar p]$
     and 
     $T^1([B^{(j)}+t \dot B^{(j)}]_{\{1,2\}}) = 
     \alpha_j^B(D)t \geq 0$ for all $j \in [\bar q]$  (by \cref{rmk:only-zeros-matter} we do not have to consider rank-2 factors). These conditions are equivalent to $\alpha_i^A(D) \geq 0$ and $\alpha_j^B(D) \geq 0$, respectively, and hence equivalent to $\underline{D} \in P_{(\A, \B)}$. 
     
     We also need to show that $T^{1}([A^{(i)} + t \dot A^{(i)}]_{\{ \ell \}}) \geq 0$ and $T^{1}([B^{(j)} + t \dot B^{(j)} ]_{\{\ell\}}) \geq 0$ for $\ell \in \{1,2\}$ and for all $i \in [\bar{p}]$ and $j \in [\bar{q}]$. 
     As stated in \cref{rmk:only-zeros-matter}, if $[A^{(i)}]_I > 0$, then $T^{s}([A^{(i)} + t \dot A^{(i)} ]_I) \geq 0$ for small enough $t>0$.
    In particular, if both diagonal entries of $A^{(i)}$ are nonzero, then for all $\underline{D} \in P_{(\A, \B)}$ the inequality $T^1([A^{(i)} + t \dot A^{(i)} ]_{\{\ell\}}) \geq 0$ is satisfied for $t \in [0,\varepsilon)$ and $\varepsilon > 0$ small enough. 

    Suppose then that a diagonal entry of $A^{(i)}$ is zero and consider $\underline{D} \in P_{(\A, \B)}$. 
    Without loss of generality, assume that the $(2,2)$ entry of $A^{(i)}$ is zero. 
    This means that $A^{(i)}$ is of the form
\[
A^{(i)} = 
\begin{pmatrix}
a & 0\\
0 & 0
\end{pmatrix},
\]
where $a>0$.
Then
\[
A^{(i)} + t \dot{A}^{(i)} = 
\begin{pmatrix}
a + taD_{11} & \frac{1}{\sqrt{2}}taD_{13}\\
\frac{1}{\sqrt{2}}taD_{13} & taD_{12}
\end{pmatrix}.
\]
Hence $T^1([A^{(i)} + t \dot{A}^{(i)}]_{\{2\}})=taD_{12}$ and $T^1([A^{(i)} + t \dot{A}^{(i)}]_{\{1,2\}})=ta^2D_{12}$. 
Since the inequality $a^2D_{12} \geq 0$ is one of the defining inequalities for the cone $P_{(\mathcal{A},\mathcal{B})}$ and $a>0$ by assumption, the inequality $aD_{12} \geq 0$ is satisfied for all points $\underline{D}$ in the cone $P_{(\mathcal{A},\mathcal{B})}$. Therefore $T^1([A^{(i)} + t \dot{A}^{(i)}]_{\{2\}})=taD_{12} \geq 0$ for all $t \geq 0$.
Furthermore, since $a > 0$ the inequality $T^1([A^{(i)} + t \dot{A}^{(i)}]_{\{1\}}) = a + taD_{11} \geq 0$ is satisfied for small enough $t>0$.
Thus, every $\underline{D} \in P_{(\A, \B)}$ satisfies the conditions imposed on the diagonal entries of the matrices $A^{(i)} + t \dot A^{(i)}$ and $B^{(j)} + t \dot B^{(j)}$   for all $i \in [\bar{p}]$ and $j \in [\bar{q}]$. 

\color{black} 
\end{proof}

\begin{proposition} \label{prop:1-inf-rigidity-implies-2-inf-rigidity}
Let $M \in \M_{3,2}^{3 \times 3}$  and consider a size-$2$ psd factorization of $M$ given by factors $A^{(1)}$, $A^{(2)}$, $A^{(3)}$, $B^{(1)}$, $B^{(2)}$, $B^{(3)}$ of rank one, where $A^{(i)} = a_i a_i^T$ and $B^{(j)} = b_j b_j^T$ for $a_i, b_j \in \R^2$. Assume all $\det(a_i,a_j) \neq 0$ and  $\det(b_i,b_j) \neq 0$. 
If the psd factorization is $1$-infinitesimally rigid, then it is also $2$-infinitesimally rigid. 
\end{proposition}

\begin{proof}
Assume that the psd factorization is $1$-infinitesimally rigid.  Hence, all $1$-infinitesimal motions  are $1$-trivial. 
The set of all $1$-trivial motions is precisely $L$ as in \cref{lem:matrix-C}, so that $P_{(\mathcal{A},\mathcal{B})} = L$ by~\Cref{lemma:1-inf-motions-enough-to-consider-determinants}. We may assume that $a_1 = (1,0)$ since by~\cref{lemma:rotation-gives-psd-factorization} the $\text{GL}(2)$-action maps a psd factorization to another psd factorization, and by \cref{cor:bijection-between-inf-motions} there is a linear isomorphism between the matrices $D$ corresponding to $1$- and $2$-infinitesimal motions and those obtained by the $\operatorname{GL}(2)$-action. We note that since $2$-trivial motions are given by $D=dI_3$, \cref{cor:bijection-between-inf-motions} implies that after the $\operatorname{GL}(2)$-action a $2$-trivial motion is given by the same $D$. Also, \cref{rmk:rank-1-factors-determinant-form} implies that the cone $P_{(\mathcal{A},\mathcal{B})}$
stays the same. 

    Next we aim to characterize all 2-infinitesimal motions. Consider a matrix $D$ corresponding to a 2-infinitesimal motion of the factorization.
    Since every 2-infinitesimal motion is a 1-infinitesimal motion by~\cref{lemma:t-inf-motion-is-s-inf-motion}, it can be written as a linear combination of the form $\sum_{i=1}^4 r_i v_i$, where $r_1, r_2, r_3, r_4 \in \R$, i.e., 
    \begin{equation*}
        \underline{D} = (2r_1, 0, r_3, 0, 2r_2, r_4, r_4, r_3, r_1+r_2).
    \end{equation*}
    Furthermore, $\beta_1^A(D)t^2 = -\frac{1}{2}r_3^2 t^2 \geq 0$, from which it follows that $r_3 = 0$.
    Then 
    \begin{equation}\label{eqn:no-orthogonal-pairs-remaining-betas}
        \begin{split}
           \beta_2^A(D) t^2 &= - \frac{1}{2} a_{22}^2 (\sqrt{2}a_{21} (r_1-r_2) +a_{22}r_4)^2 t^2 \geq 0, \\
            \beta_3^A(D)t^2 &= - \frac{1}{2} a_{32}^2 (\sqrt{2}a_{31} (r_1-r_2) +a_{32}r_4)^2 t^2 \geq 0, \\
            \beta_1^B(D)t^2 &= - \frac{1}{2} b_{11}^2 (\sqrt{2}b_{12} (r_1-r_2)  - b_{11}r_4)^2 t^2 \geq 0, \\
            \beta_2^B(D)t^2 &= - \frac{1}{2} b_{21}^2 (\sqrt{2}b_{22} (r_1-r_2)  - b_{21}r_4)^2 t^2 \geq 0, \\
            \beta_3^B(D)t^2 &= - \frac{1}{2} b_{31}^2 (\sqrt{2}b_{32} (r_1-r_2)  - b_{31}r_4)^2 t^2 \geq 0. \\
        \end{split}
    \end{equation}

    From the inequalities \eqref{eqn:no-orthogonal-pairs-remaining-betas}, we get the system of equations
    \begin{equation}\label{eqn:r1-r2-system}
        \begin{pmatrix}
            \sqrt{2}a_{21} & a_{22} \\
            \sqrt{2}a_{31} & a_{32} \\
            \sqrt{2}b_{12} & -b_{11} \\
            \sqrt{2}b_{22} & -b_{21} \\
            \sqrt{2}b_{32} & -b_{31}
        \end{pmatrix} \begin{pmatrix}
            r_1-r_2 \\
            r_4
        \end{pmatrix} = 0.
    \end{equation}
    The rank of the above matrix is $2$
    since the submatrix
$S = \begin{pmatrix}
            \sqrt{2}a_{21} & a_{22} \\
            \sqrt{2}a_{31} & a_{32}
        \end{pmatrix}$ has determinant 
 $\det(S) =  \sqrt{2} \det(a_2, a_3)$
    where $\det(a_2, a_3) \neq 0$ by assumption.
    Hence, the only possible solution to the system \eqref{eqn:r1-r2-system} is $r_1 = r_2$ and $r_4=0$, and the only 2-infinitesimal motions correspond to
    \begin{equation*}
        \underline{D} = (2r_1, 0, 0, 0, 2r_1, 0, 0, 0, 2r_1),
    \end{equation*}
    i.e., the factorization is 2-infinitesimally rigid.
\end{proof}

\begin{lemma}\label{lem:full-dimensional-cone-implies-inf-rig}
Let $M \in \M_{3,2}^{p \times q}$ and consider a size-$2$ psd factorization of $M$ given by factors $A^{(1)}$, $\dots$, $A^{(p)}$, $B^{(1)}$, $\dots$, $B^{(q)} \in S^2_+$.
 If the cone $P_{(\mathcal{A},\mathcal{B})}$ is full-dimensional, then the psd factorization is $s$-infinitesimally flexible for $s \in \{1,2\}$. 
\end{lemma}

\begin{proof}
Assume $\dim(P_{(\mathcal{A},\mathcal{B})}) = 9$. Let $C_{(\mathcal{A},\mathcal{B})}$   be as in~\Cref{def:CAB-and-PAB} and $L$ as in \cref{lem:matrix-C}. 
We will first show that the psd factorization is not $1$-infinitesimally rigid.
Each $\underline{D}\in P_{(\mathcal{A},\mathcal{B})}$ corresponds to a $1$-infinitesimal motion by~\Cref{lemma:1-inf-motions-enough-to-consider-determinants}. 
Since all $1$-trivial motions of the factorization are in $L \subset P_{(\mathcal{A},\mathcal{B})}$ which is just $4$-dimensional, it follows that the factorization is $1$-infinitesimally flexible.

Next we will show that the psd factorization is not 2-infinitesimally rigid.   By~\Cref{lem:rank-2-factors-do-not-matter} we can consider only rank-one factors among $A^{(1)}$, $\dots$, $A^{(p)}$, $B^{(1)}$, $\dots$, $B^{(q)}$  to study 2-infinitesimal rigidity.  There exists an interior point in the cone $P_{(\mathcal{A},\mathcal{B})}$, i.e., there exists some $\underline{D} \in P_{(\mathcal{A},\mathcal{B})}$ such that $C_{(\mathcal{A},\mathcal{B})}\underline{D} > 0$.   Recall that   if  $A^{(i)},B^{(j)}$ are of rank one,  by \cref{rmk:rank-1-factors-determinant-form} the determinants of the matrices $A^{(i)} + t \dot A^{(i)}$ and $B^{(j)}+ t \dot B^{(j)}$ are of the form $\alpha(D)t + \beta(D) t^2$. The inequality $C_{(\mathcal{A},\mathcal{B})}\underline{D} > 0$ is equivalent to $\alpha^A_i \underline{D} >0$ and $\alpha^B_j \underline{D} >0$ for all $i,j$ by the definition of $C_{(\mathcal{A},\mathcal{B})}$.
It follows that $\alpha^A_i(D)t + \beta^A_i(D) t^2 \geq 0$ and $\alpha^B_j(D)t + \beta^B_j(D) t^2 \geq 0$ for all $t \in [0, \varepsilon)$ for $\varepsilon > 0$ small enough.
Therefore $\underline{D}$ corresponds to a 2-infinitesimal motion of the psd factorization. 
It is not a scalar multiple of the 2-trivial motion $(1,0,0,0,1,0,0,0,1)$ because this 2-trivial motion is in the right kernel of the matrix $C_{(\mathcal{A},\mathcal{B})}$ by~\Cref{lem:matrix-C} but $\underline{D}$ is not. Hence $\underline{D}$ is a non-trivial 2-infinitesimal motion and the psd factorization is 2-infinitesimally flexible. 
\end{proof}

\begin{proposition}\label{thm:three-and-three-no-orthogonal-pairs-case}
Let $M \in \M_{3,2}^{3 \times 3}$ and consider a size-$2$ psd factorization of $M$ given by factors $A^{(1)}$, $A^{(2)}$, $A^{(3)}$, $B^{(1)}$, $B^{(2)}$, $B^{(3)}$ of rank one where $A^{(i)} = a_i a_i^T$ and $B^{(j)} = b_j b_j^T$ for $a_i, b_j \in \R^2$. Assume all $\det(a_i,a_j) \neq 0, \det(b_i,b_j) \neq 0$ and $\langle a_i,b_j \rangle \neq 0$.
Then the factorization is $1$-infinitesimally rigid if and only if 
\begin{equation} \label{eqn:15-factors}
\begin{split}    
\frac{\det(b_1,b_3)\det(b_2,b_3)\langle a_2,b_3 \rangle \langle a_3,b_3 \rangle}{\det(a_1,a_2)\det(a_1,a_3)\langle a_1,b_1 \rangle \langle a_1,b_2 \rangle} >0 \\
-\frac{\det(b_1,b_3)\det(b_2,b_3)\langle a_1,b_3 \rangle \langle a_3,b_3 \rangle}{\det(a_1,a_2)\det(a_2,a_3)\langle a_2,b_1 \rangle \langle a_2,b_2 \rangle} >0 \\
\frac{\det(b_1,b_3)\det(b_2,b_3)\langle a_1,b_3 \rangle \langle a_2,b_3 \rangle}{\det(a_1,a_3)\det(a_2,a_3)\langle a_3,b_1 \rangle \langle a_3,b_2 \rangle} >0\\
\frac{\det(b_2,b_3) \langle a_1,b_3 \rangle \langle a_2,b_3 \rangle \langle a_3,b_3 \rangle}{\det(b_1,b_2)\langle a_1,b_1 \rangle \langle a_2,b_1 \rangle \langle a_3,b_1 \rangle} >0 \\
-\frac{\det(b_1,b_3) \langle a_1,b_3 \rangle \langle a_2,b_3 \rangle \langle a_3,b_3 \rangle}{\det(b_1,b_2)\langle a_1,b_2 \rangle \langle a_2,b_2 \rangle \langle a_3,b_2 \rangle} >0.
\end{split}
\end{equation}
The factorization is $1$-infinitesimally flexible   if and only if  the cone $P_{(\mathcal{A},\mathcal{B})}$ is full-dimensional.
\end{proposition}
\begin{proof}
Consider the cone $P_{(\mathcal{A},\mathcal{B})}$.
By Farkas' lemma there exists $y \geq 0, y \neq 0$ such that $y^T C_{(\mathcal{A},\mathcal{B})} = 0$ if and only if there does not exist an element $\underline{D} \in \R^9$ such that $C_{(\mathcal{A},\mathcal{B})}\underline{D} >0$. As all $\det(a_i,a_j) \neq 0, \det(b_i,b_j) \neq 0$ and $\langle a_i,b_j \rangle \neq 0$, we have $\rank(C_{(\mathcal{A},\mathcal{B})}) = 5$ by~\Cref{lem:matrix-C}. 
The matrix $C_{(\mathcal{A},\mathcal{B})}$ has a $1$-dimensional left kernel, whose generator is given by~\eqref{eqn:left-kernel-C}.
We can divide all coordinates of~\eqref{eqn:left-kernel-C} by the last coordinate resulting in the coordinates appearing in~\eqref{eqn:15-factors} and one as the last coordinate. Thus having $y \geq 0, y \neq 0$ such that $y^T C_{(\mathcal{A}, \mathcal{B})} = 0$ is equivalent to the inequalities~\eqref{eqn:15-factors} being satisfied because all factors in~\eqref{eqn:15-factors} are nonzero by the hypothesis.

First assume that \eqref{eqn:15-factors} holds.  
Then by \cref{lem:no-orthogonal-pairs-positive-entries-make-alphas-zero}, we have $\alpha_i^A \underline{D} = \alpha_j^B \underline{D} = 0$ for all $\underline{D} \in P$ and hence $P_{(\mathcal{A},\mathcal{B})}$ is contained in the right kernel of $C_{(\mathcal{A},\mathcal{B})}$. The right kernel of $C_{(\mathcal{A},\mathcal{B})}$ is equal to $L$  by \cref{lem:matrix-C}. Thus $P_{(\mathcal{A},\mathcal{B})} \subseteq L$ and the psd factorization is $1$-infinitesimally rigid.

If \eqref{eqn:15-factors} does not hold, then every nonzero $y$ satisfying $y^T C_{(\mathcal{A},\mathcal{B})} = 0$ has both positive and negative entries. 
By Farkas' lemma there must exist an element $\underline{D} \in \R^9$ such that $C_{(\mathcal{A},\mathcal{B})}\underline{D} >0$.
By \cref{lem:full-dimensional-cone-implies-inf-rig}, the factorization is $1$-infinitesimally flexible.
\end{proof}

\begin{remark}
There are $2^{15}$ different sign configurations for the factors $\det(a_i,a_j)$, $\det(b_i,b_j)$ and $\langle a_i, b_j \rangle$, and $2^{11} = 2048$  of these result in all positive entries in the generator $y$.
The rest result in both positive and negative entries in $y$.
Both of these cases are realizable.

    Consider the psd factorization given by the factors
    \begin{align*}
        A^{(1)} &= \begin{pmatrix}
            1 & 0 \\
            0 & 0
        \end{pmatrix}, \quad A^{(2)} = \begin{pmatrix}
            ~\frac{1}{4} & -\frac{1}{4} \\
            -\frac{1}{4} & ~\frac{1}{4}
        \end{pmatrix}, \quad A^{(3)} = \begin{pmatrix}
            0 & 0 \\
            0 & 1
        \end{pmatrix}, \\
        B^{(1)} &= \begin{pmatrix}
            \frac{1}{4} & \frac{3}{4} \\
            \frac{3}{4} & \frac{9}{4}
        \end{pmatrix}, \quad B^{(2)} = \begin{pmatrix}
            ~\frac{1}{4} & -\frac{1}{4} \\
            -\frac{1}{4} & ~\frac{1}{4}
        \end{pmatrix}, \quad B^{(3)} = \begin{pmatrix}
            1 & \frac{1}{4} \\
            \frac{1}{4} & \frac{1}{16}
        \end{pmatrix}.
    \end{align*}
    This factorization satisfies \eqref{eqn:15-factors}, and thus it is $1$-infinitesimally rigid by \cref{thm:three-and-three-no-orthogonal-pairs-case}.

    On the other hand, the psd factorization given by the factors 
    \begin{align*}
        A^{(1)} & = \begin{pmatrix}
            1 & 2 \\
            2 & 4
        \end{pmatrix}, \quad
        A^{(2)} = \begin{pmatrix}
            1 & 3 \\
            3 & 9
        \end{pmatrix}, \quad
        A^{(3)} = \begin{pmatrix}
            1 & 4 \\
            4 & 16
        \end{pmatrix} \\
        B^{(1)} & = \begin{pmatrix}
            1 & 5 \\
            5 & 25
        \end{pmatrix}, \quad
        B^{(2)} = \begin{pmatrix}
            1 & 6 \\
            6 & 36
        \end{pmatrix}, \quad
        B^{(3)} = \begin{pmatrix}
            1 & 7 \\
            7 & 49
        \end{pmatrix}.
    \end{align*}
    does not satisfy \eqref{eqn:15-factors}, and thus  it is $1$-infinitesimally flexible by \cref{thm:three-and-three-no-orthogonal-pairs-case}.
\end{remark}

\begin{lemma} \label{lem:4A-and-2B}
     Let $A^{(1)},\dots , A^{(\bar{p})},B^{(1)}, \dots, B^{(\bar{q})} \in \mathcal{S}_+^2$, $\bar{p} + \bar{q} \leq 6$,   $(\bar{p},\bar{q}) \neq (3,3)$,  be psd matrices of rank one, where $A^{(i)} = a_ia_i^T$ and $B^{(j)} = b_jb_j^T$ for $a_i = (a_{i1}, a_{i2}), b_j = (b_{j1}, b_{j2}) \in \R^2$. Assume all $\det(a_i,a_j) \neq 0, \det(b_i,b_j) \neq 0$ and $\langle a_i,b_j \rangle \neq 0$.
     Then the cone $P_{(\mathcal{A},\mathcal{B})}$ is full-dimensional.     
\end{lemma}

\begin{proof}
By Farkas' lemma there exists a $y \geq 0, y \neq 0$ such that $y^T C = 0$ if and only if there does not exist an element $\underline{D} \in \R^9$ such that $C_{(\mathcal{A},\mathcal{B})}\underline{D} >0$.   First assume $\bar{p}+\bar{q} =6$. As all $\det(a_i,a_j) \neq 0, \det(b_i,b_j) \neq 0$ and $\langle a_i,b_j \rangle \neq 0$, it follows from~\Cref{lem:matrix-C} that $\rank(C_{(\mathcal{A},\mathcal{B})}) = 5$ and the left kernel of $C_{(\mathcal{A},\mathcal{B})}$ is spanned by the vector $v$ with coordinates given by~\eqref{eqn:left-kernel-C}. For $\bar{p}+\bar{q}=6$ and $\bar{p} \not \in \{0,3,6\}$, we have computationally checked that there are no non-zero nonnegative left kernel vectors. Hence the cone $P_{(\mathcal{A},\mathcal{B})}$ is full-dimensional.

If $\bar p=6$, then second coordinates of the rows of $C_{(\mathcal{A},\mathcal{B})}$ are of the form $a_{i1}^4$ and fourth coordinates are of the form $a_{i2}^4$. For a non-trivial conic combination of these vectors to be the zero vector, all the corresponding $a_{i1}$ and $a_{i2}$ would have to be zero. But this would mean that $\det(a_i,a_j)=0$ for all $j$ which is a contradiction to our assumptions. An analogous argument works for $\bar q = 6$. We conclude that the zero vector cannot be written as a non-trivial conic combination of the rows of $C_{(\mathcal{A},\mathcal{B})}$, hence the cone $P_{(\mathcal{A},\mathcal{B})}$ is full-dimensional.

Finally, if $\bar p + \bar q < 6$, then it follows from the previous cases that the cone $P_{(\mathcal{A},\mathcal{B})}$ is full-dimensional, because even after adding rank-one factors until there are in total of six rank-one factors, the resulting cone is full-dimensional. 
\end{proof}

\begin{lemma} \label{lemma:rank_three_matrix}
Let $M \in \mathbb{R}_+^{3 \times 3}$ with a size-$2$ psd factorization given by factors $A^{(1)}$, $A^{(2)}$, $A^{(3)}$, $B^{(1)}$, $B^{(2)}$, $B^{(3)}$ of rank one, where $A^{(i)} = a_i a_i^T$ and $B^{(j)} = b_j b_j^T$ for $a_i, b_j \in \R^2$. Assume all $\det(a_i,a_j) \neq 0, \det(b_i,b_j) \neq 0$ and $\langle a_i,b_j \rangle \neq 0$. Then $M$ is a rank-three matrix.
\end{lemma}

\begin{proof}
The determinant of matrix $\mathcal{A}$ is $\sqrt{2} \det(a_1,a_2) \det(a_1,a_3) \det(a_2,a_3) $. Since  all $\det(a_i,a_j) \neq 0$, then this determinant is non-zero and $\mathcal{A}$ is a rank-three matrix. Similarly one can argue that $\mathcal{B}$ is a rank-three matrix. It follows that $M = \mathcal{A} \mathcal{B}$ is a rank-three matrix.
\end{proof}
\color{black}

All this brings us to our characterization of $1$- and $2$-infinitesimally rigid factorizations of a positive matrix $M \in \M_{3,2}^{p\times q}$.

\begin{theorem} \label{thm:no-orthogonal-pairs}
Let $M \in \M_{3,2}^{p \times q}$ and consider a size-$2$ psd factorization of $M$ given by the factors $A^{(1)}$, $\dots$, $A^{(p)}$, $B^{(1)}$, $\dots$, $B^{(q)} \in \mathcal{S}_+^2$. Let $A^{(1)}$, $\dots$, $A^{(\bar{p})}$, $B^{(1)}$, $\dots$, $B^{(\bar{q})}$  be the collection of the factors that have rank one. For these,  let $A^{(i)} = a_i a_i^T$ and $B^{(j)} = b_j b_j^T$ for $a_i, b_j \in \R^2$. Assume that all $\det(a_i,a_j) \neq 0, \det(b_i,b_j) \neq 0$ and $\langle a_i,b_j \rangle \neq 0$.
The factorization is $s$-infinitesimally rigid for $s\in \{1,2\}$ if and only if there exist  factors of rank one $A^{(i_1)},A^{(i_2)},A^{(i_3)},B^{(j_1)},B^{(j_2)},B^{(j_3)}$ such that 
\begin{equation}\label{eqn:Farkas-lemma-conditions-pxq}
\begin{split}    
\frac{\det(b_{j_1},b_{j_3})\det(b_{j_2},b_{j_3})\langle a_{i_2},b_{j_3} \rangle \langle a_{i_3},b_{j_3} \rangle}{\det(a_{i_1},a_{i_2})\det(a_{i_1},a_{i_3})\langle a_{i_1},b_{j_1} \rangle \langle a_{i_1},b_{j_2} \rangle} >0 \\
-\frac{\det(b_{j_1},b_{j_3})\det(b_{j_2},b_{j_3})\langle a_{i_1},b_{j_3} \rangle \langle a_{i_3},b_{j_3} \rangle}{\det(a_{i_1},a_{i_2})\det(a_{i_2},a_{i_3})\langle a_{i_2},b_{j_1} \rangle \langle a_{i_2},b_{j_2} \rangle} >0 \\
\frac{\det(b_{j_1},b_{j_3})\det(b_{j_2},b_{j_3})\langle a_{i_1},b_{j_3} \rangle \langle a_{i_2},b_{j_3} \rangle}{\det(a_{i_1},a_{i_3})\det(a_{i_2},a_{i_3})\langle a_{i_3},b_{j_1} \rangle \langle a_{i_3},b_{j_2} \rangle} >0\\
\frac{\det(b_{j_2},b_{j_3}) \langle a_{i_1},b_{j_3} \rangle \langle a_{i_2},b_{j_3} \rangle \langle a_{i_3},b_{j_3} \rangle}{\det(b_{j_1},b_{j_2})\langle a_{i_1},b_{j_1} \rangle \langle a_{i_2},b_{j_1} \rangle \langle a_{i_3},b_{j_1} \rangle} >0 \\
-\frac{\det(b_{j_1},b_{j_3}) \langle a_{i_1},b_{j_3} \rangle \langle a_{i_2},b_{j_3} \rangle \langle a_{i_3},b_{j_3} \rangle}{\det(b_{j_1},b_{j_2})\langle a_{i_1},b_{j_2} \rangle \langle a_{i_2},b_{j_2} \rangle \langle a_{i_3},b_{j_2} \rangle} >0.
\end{split}
\end{equation}
\end{theorem}

\begin{proof}
  Since by assumption all $\det(a_i,a_j) \neq 0, \det(b_i,b_j) \neq 0$, then by~\Cref{lemma:rank_three_matrix} the submatrix of $M$ corresponding to rows $i_1,i_2,i_3$ and columns $j_1,j_2,j_3$ is of rank three and we are in the setting of~\Cref{thm:three-and-three-no-orthogonal-pairs-case}.  If the inequalities \eqref{eqn:Farkas-lemma-conditions-pxq} are satisfied, then   by~\Cref{thm:three-and-three-no-orthogonal-pairs-case} the psd factorization $A^{(i_1)},A^{(i_2)},A^{(i_3)},B^{(j_1)},B^{(j_2)},B^{(j_3)}$ is 1-infinitesimally rigid.   By~\cref{prop:1-inf-rigidity-implies-2-inf-rigidity}, the psd factorization $A^{(i_1)},A^{(i_2)},A^{(i_3)}$, $B^{(j_1)},B^{(j_2)},B^{(j_3)}$ is $2$-infinitesimally rigid. By~\Cref{lem:sub-facorization-rigid}, the psd factorization $A^{(1)}$, $\dots$, $A^{(p)}$, $B^{(1)}$, $\dots$, $B^{(q)}$  is $1$- and $2$-infinitesimally rigid.

Assume that there do not exist factors of rank one $A^{(i_1)}, A^{(i_2)}, A^{(i_3)}, B^{(j_1)}, A^{(j_2)}, A^{(j_3)}$ such that the inequalities \eqref{eqn:Farkas-lemma-conditions-pxq} are satisfied. Consider the $(\bar{p} + \bar{q}) \times 9$ matrix $C_{(\mathcal{A},\mathcal{B})}$ as in~\Cref{def:CAB-and-PAB}. We will show that the cone $P_{(\mathcal{A},\mathcal{B})}$ is full-dimensional and then by~\Cref{lem:full-dimensional-cone-implies-inf-rig} the psd factorization is  $1$- or $2$-infinitesimally flexible.
Let $C_{(\mathcal{A},\mathcal{B})}'$ be the $(\bar{p} + \bar{q}) \times 5$ submatrix of $C_{(\mathcal{A},\mathcal{B})}$ where only the linearly independent columns $1,2,3,4,6$ of $C_{(\mathcal{A},\mathcal{B})}$ are kept. The cone $P_{(\mathcal{A},\mathcal{B})}$ is full-dimensional if and only if the cone $P_{(\mathcal{A},\mathcal{B})}' = \{\underline{D'} \in \R^5 \mid C_{(\mathcal{A},\mathcal{B})}' \underline{D'} \geq 0 \}$ is full-dimensional. 
 
By contradiction, assume that the cone $P_{(\mathcal{A},\mathcal{B})}'$ is not full-dimensional. This means that $(0,0,0,0,0)$ can be written as a non-trivial conic combination of the rows of the matrix $C_{(\mathcal{A},\mathcal{B})}'$. By Caratheodory’s theorem, the zero vector can be written as a non-trivial conic combination of at most  six rows of the matrix $C_{(\mathcal{A},\mathcal{B})}'$.   Hence if the cone $P_{(\mathcal{A},\mathcal{B})}'$ is not full-dimensional, then there exist at most six rows of $C_{(\mathcal{A},\mathcal{B})}'$ such that the cone defined by these rows is not full-dimensional.

We will show that no matter how many of the at most six rows correspond to $A$ and $B$ factors,   then the cones defined by the corresponding rows of $C_{(\mathcal{A},\mathcal{B})}'$ and $C_{(\mathcal{A},\mathcal{B})}$ are full-dimensional and hence by the previous paragraph the cones $P_{(\mathcal{A},\mathcal{B})}'$ and $P_{(\mathcal{A},\mathcal{B})}$ are full-dimensional.  
If three rows correspond to $A$ factors and three rows correspond to $B$ factors,   then by~\Cref{lemma:rank_three_matrix} the submatrix of $M$ corresponding to these rows and columns is of rank three and we are in the setting of~\Cref{thm:three-and-three-no-orthogonal-pairs-case}. Since the inequalities \eqref{eqn:Farkas-lemma-conditions-pxq} are not satisfied, then by~\Cref{thm:three-and-three-no-orthogonal-pairs-case} the cones corresponding to these six factors are full-dimensional. In all other cases of at most six factors, the cones corresponding to these factors are full-dimensional by~\Cref{lem:4A-and-2B}. Hence no matter how we choose at most six rank-one factors, the cone $P_{(\mathcal{A},\mathcal{B})}$ is full-dimensional.   
The factorization is neither $1$-infinitesimally rigid nor $2$-infinitesimally rigid by~\Cref{lem:full-dimensional-cone-implies-inf-rig}.
\end{proof}

\subsection{One orthogonal pair} \label{sec:one-orthogonal-pair}

In the subsequent subsections we consider $2$-infinitesimal motions of matrices with zero entries.
This subsection is devoted to the case when precisely one of the factors $A^{(i)}$ of rank one is orthogonal to a factor $B^{(j)}$ of rank one.
We give a complete characterization of the $2$-infinitesimal motions of size-$2$ psd factorizations of matrices with precisely one zero entry.

\begin{lemma} \label{lem:matrix-C2}
Let $   A^{(1)}, A^{(2)}$, $\dots$, $A^{(\bar{p})}$, $   B^{(1)},  B^{(2)}$, $\dots$, $B^{(\bar{q})} \in \mathcal{S}^2_+$ be matrices of rank one, where $A^{(i)} = a_i a_i^T$ and $B^{(j)} = b_j b_j^T$ for $a_i, b_j \in \R^2$,   $a_1=(1,0)$ and $b_1=(0,1)$.  Define
    \begin{align*}
        \alpha_i^A & := (a_{i1}^2 a_{i2}^2,  a_{i2}^4, a_{i1}^2 a_{i2}^2, -\sqrt{2} a_{i1} a_{i2}^3, \sqrt{2} a_{i1} a_{i2}^3, -2 a_{i1}^2 a_{i2}^2), \\
        \alpha_j^B & := (-b_{j1}^2 b_{j2}^2,  -b_{j1}^4, -b_{j1}^2 b_{j2}^2, -\sqrt{2} b_{j1}^3 b_{j2}, \sqrt{2} b_{j1}^3 b_{j2}, 2 b_{j1}^2 b_{j2}^2).        
    \end{align*}
Consider the $(\bar{p}+\bar{q}-2) \times 6$ matrix $\bar{C}_{(\mathcal{A},\mathcal{B})}$ where we label the rows by $a_2,\ldots,a_{\bar p},b_2,\ldots,b_{\bar q}$
\begin{align} \label{eqn:matrix-C2}
    \bar{C}_{(\mathcal{A},\mathcal{B})} := \begin{pmatrix}
        \rowvec{\alpha_{2}^A} \\
        \vdots \\
        \rowvec{\alpha_{\bar{p}}^A}\\
        \rowvec{\alpha_{2}^B}\\
        \vdots \\
        \rowvec{\alpha_{\bar{q}}^B}
    \end{pmatrix}.
\end{align}
  Rows corresponding to $A^{(1)},B^{(1)}$ are left out from $\bar{C}_{(\mathcal{A},\mathcal{B})}$ because they are zero. 
\begin{enumerate}
\item The rank of $\bar{C}_{(\mathcal{A},\mathcal{B})}$ is at most three. 

\item The right kernel of $\bar{C}_{(\mathcal{A},\mathcal{B})}$ contains the subspace $\bar{L} := \operatorname{span}(w_1, w_2, w_3)$, where
\begin{align*}
    w_1 &= (2, 0, 0, 0, 0, 1),\\
    w_2 &= (0, 0, 2, 0, 0, 1),\\
    w_3 &= (0, 0, 0, 1, 1, 0).
\end{align*}
If $\rank(\bar{C}_{(\mathcal{A},\mathcal{B})}) = 3$, then the right kernel is equal to $\bar{L}$.  

\item  If $\bar{p}+\bar{q}=6$, then the left kernel of $\bar{C}_{(\mathcal{A},\mathcal{B})}$ contains the vector $v \in \R^{\bar{p}+\bar{q}-2}$ with entries 
\begin{equation} \label{eqn:left-kernel-C2} 
    v_{c_k} := 
    (-1)^{k + \mathbbm{1}_{k \geq \bar{p}+1}} \prod_{\substack{1 \leq i < j \leq \bar{p},\\a_i,a_j \neq c_k}} \det(a_i,a_j) \prod_{\substack{1 \leq i < j \leq \bar{q},\\b_i,b_j \neq c_k}} \det(b_i,b_j) \prod_{\substack{1 \leq i \leq \bar{p},\\ 1 \leq j \leq \bar{q},\\i+j>2,\\a_i,b_j \neq c_k}} \langle a_i, b_j \rangle.
\end{equation}
where $c_2=a_2,\ldots,c_{\bar{p}}=a_{\bar{p}},c_{\bar{p}+2}=b_2,\ldots,c_{\bar{p}+\bar{q}}=b_{\bar{q}}$.
If $\rank(\bar{C}_{(\mathcal{A}, \mathcal{B})}) = 3$, then the left kernel is spanned by $v$. 

\item If   $\bar{p}+\bar{q} \geq 5$,   $\det(a_i,a_j) \neq 0, \det(b_i,b_j) \neq 0$   for all $i,j$  and $\langle a_i,b_j \rangle \neq 0$   for $(i,j) \neq (1,1)$ , then $\rank(\bar{C}_{(\mathcal{A}, \mathcal{B})}) = 3$.
\end{enumerate}
\end{lemma}
\begin{proof}
We have $\rank(\bar{C}_{(\mathcal{A}, \mathcal{B})}) \leq 3$ since the columns $1,3,6$ are multiples of each other, as are $4,5$. It can be easily checked from these observations that the right kernel contains the subspace $\bar{L}$. If $\rank(\bar{C}_{(\mathcal{A}, \mathcal{B})}) = 3$, then $\ker \bar{C}_{(\mathcal{A}, \mathcal{B})}$ is $3$-dimensional, and since $\bar{L}$ is 3-dimensional, then $\bar{L} = \ker \bar{C}_{(\mathcal{A}, \mathcal{B})}$.

Finally assume that $\bar{p}+\bar{q}=6$, and consider the $4 \times 3$ submatrix $C'$ consisting of the columns $1,  2$ and $5$ of $\bar{C}_{(\mathcal{A},\mathcal{B})}$. The vector $(C'_2,-C'_3, C'_5, -C'_6)$ where $C'_i$ is the $3 \times 3$ minor of $C'$ obtained by deleting row labeled by $i$ is in the left kernel of $C'$ and hence of $\bar{C}_{(\mathcal{A},\mathcal{B})}$. For all possible cases of $\bar{p}$ and $\bar{q}$ one can computationally check that this vector is  equal to the vector given by~\eqref{eqn:left-kernel-C2} scaled by $-\sqrt{2}$.
If    $\bar{p}+\bar{q} \geq 5$,   $\det(a_i,a_j) \neq 0, \det(b_i,b_j) \neq 0$ for all $i,j$ and $\langle a_i,b_j \rangle \neq 0$ for $(i,j) \neq (1,1)$,  then none of the $3 \times 3$ minors vanish and hence rank of $\bar{C}_{(\mathcal{A},\mathcal{B})}$ is three.
\end{proof}

\begin{remark}
The expression~\eqref{eqn:left-kernel-C2} is almost the same as~\eqref{eqn:left-kernel-C} only with the factor $\langle a_1,b_1 \rangle$ removed. 
\end{remark}

\begin{proposition}\label{prop:one-orth}
Let $M\in \M_{3,2}^{3\times 3}$ and consider a size-$2$ psd factorization of $M$ given by factors $A^{(1)},A^{(2)},A^{(3)},B^{(1)},B^{(2)},B^{(3)}$ of rank one where $A^{(i)} = a_ia_i^T$ and $B^{(j)} = b_jb_j^T$, and $\langle a_1,b_1 \rangle =0$. Assume $\det(a_i,a_j) \neq 0, \det(b_i,b_j) \neq 0$ for all $i,j$ and $\langle a_i,b_j \rangle \neq 0$ for   all $(i,j) \neq (1,1)$. 
Then the factorization is $2$-infinitesimally rigid if and only if 
    \begin{equation}\label{eqn:4-factors}
\begin{split}
-\frac{\det(b_1,b_3) \det(b_2,b_3) \langle a_1,b_3 \rangle \langle a_3,b_3 \rangle}{\det(a_1,a_2) \det(a_2,a_3) \langle a_2,b_1 \rangle \langle a_2,b_2 \rangle} >0, \\
\frac{\det(b_1,b_3) \det(b_2,b_3) \langle a_1,b_3 \rangle \langle a_2,b_3 \rangle}{\det(a_1,a_3) \det(a_2,a_3) \langle a_3,b_1 \rangle \langle a_3,b_2 \rangle} >0,\\
-\frac{\det(b_1,b_3) \langle a_1,b_3 \rangle\langle a_2,b_3 \rangle \langle a_3,b_3 \rangle}{\det(b_1,b_2) \langle a_1,b_2 \rangle \langle a_2,b_2 \rangle \langle a_3,b_2 \rangle} >0.
\end{split}
    \end{equation}
\end{proposition}
\begin{proof}
First we prove the claim for  $a_1 = (1, 0)$ and $b_1 = (0,1)$. 
Consider a $2$-infinitesimal motion of the factorization given by the factors  $A^{(1)}$, $A^{(2)}$, $A^{(3)}$, $B^{(1)}$, $B^{(2)}$, $B^{(3)}$  corresponding to the matrix $D$.
In order for this factorization to be psd, we need 1-minors (diagonals) and 2-minors (determinants) of $A^{(i)}+ t\dot A^{(i)}$ and $B^{(j)}+ t\dot B^{(j)}$ for all $i=1,\ldots,3$ and $j=1,\ldots,3$ to be nonnegative. We examine the diagonals of the factors $A^{(1)}(t)$ and $B^{(1)}(t)$ to conclude $D_{12}=0$:  \begin{align*}
A^{(1)}+ t\dot A^{(1)}= \left(
\begin{array}{cc}
 D_{11} t+1 & \frac{D_{13} t}{\sqrt{2}} \\
\frac{ D_{13} t}{\sqrt{2}} &  D_{12} t \\
\end{array}
\right) 
\\
B^{(1)} + t\dot B^{(1)}= \left(
\begin{array}{cc}
- D_{12}t & -\frac{ D_{32} t}{\sqrt{2}} \\
-\frac{ D_{32} t}{\sqrt{2}} & 1- D_{22} t \\
\end{array}
\right)
\end{align*} 
We then substitute $D_{12}=0$ into the following determinants:
 \begin{align*}
   \text{det}(A^{(1)}+ t\dot A^{(1)}) &=  D_{12} t + \left( D_{11} D_{12}-\frac{1}{2}  D_{13}^2\right)t^2 = -\frac{1}{2}  D_{13}^2 t^2\\
   \text{det}(B^{(1)}+ t\dot B^{(1)}) &= - D_{12} t + \left( D_{12} D_{22}-\frac{1}{2}  D_{32}^2\right)t^2 = -\frac{1}{2}  D_{32}^2t^2 
\end{align*}\color{black}
Since we require that the determinants are nonnegative, we can also conclude that $D_{13}=D_{32}=0$. Then $\text{det}(A^{(1)}+ t\dot A^{(1)})= \text{det}(B^{(1)}+ t\dot B^{(1)}) = 0$. Therefore $D_{12} = D_{13} = D_{32} = 0$ for any $D$ corresponding to a $2$-infinitesimal motion.

Consider the cone $\bar{P}_{(\mathcal{A},\mathcal{B})} = \{\underline{D} \in \R^6 \mid \bar{C}_{(\mathcal{A},\mathcal{B})} \underline{D} \geq 0 \}$ for the matrix $\bar{C}_{(\mathcal{A},\mathcal{B})}$ in~\eqref{eqn:matrix-C2}. This matrix is obtained from matrix~\eqref{eqn:matrix-C} by removing columns corresponding to $D_{12},D_{13},D_{32}$. 
By Farkas' lemma there exists $y \geq 0, y \neq 0$ such that $y^T \bar{C}_{(\mathcal{A},\mathcal{B})} = 0$ if and only if there does not exist an element $\underline{D} \in \R^6$ such that $\bar{C}_{(\mathcal{A},\mathcal{B})}\underline{D} >0$.
  
As now $\det(a_i,a_j) \neq 0, \det(b_i,b_j) \neq 0$ for all $i,j$ and $\langle a_i,b_j \rangle \neq 0$ for $(i,j) \neq (1,1)$,  we have $\rank(\bar{C}_{(\mathcal{A}, \mathcal{B})}) = 3$ by~\Cref{lem:matrix-C2}. 
The matrix $\bar{C}_{(\mathcal{A},\mathcal{B})}$ has a $1$-dimensional left kernel, whose generator is given by~\eqref{eqn:left-kernel-C2}.
Since all factors of all coordinates of~\eqref{eqn:left-kernel-C2} are non-zero, we can divide all coordinates of~\eqref{eqn:left-kernel-C2} by the last coordinate resulting in coordinates appearing in~\eqref{eqn:4-factors} and one as the last coordinate. Thus having $y \geq 0, y \neq 0$ such that $y^T \bar{C}_{(\mathcal{A}, \mathcal{B})} = 0$ is equivalent to the inequalities~\eqref{eqn:4-factors} being satisfied.

First assume that \eqref{eqn:4-factors} holds.
Then, all of the entries of the generator $y$ are strictly positive.
By Farkas' lemma, there does not exist an $\underline{D} \in \R^6$ such that $\bar{C}_{(\mathcal{A},\mathcal{B})} \underline{D} > 0$.
\cref{lem:no-orthogonal-pairs-positive-entries-make-alphas-zero} shows that then for all $\underline{D} \in \bar{P}_{(\mathcal{A},\mathcal{B})}$ we have $\bar{C}_{(\mathcal{A},\mathcal{B})} \underline{D} = 0$,
and hence $\bar{P}_{(\mathcal{A},\mathcal{B})}$ is contained in the right kernel of $\bar{C}_{(\mathcal{A},\mathcal{B})}$.
Thus we consider motions in the kernel of $\bar{C}_{(\mathcal{A},\mathcal{B})}$. 
By~\Cref{lem:matrix-C2}, a basis for the right kernel of $\bar{C}_{(\mathcal{A},\mathcal{B})}$ is given by the vectors 
\begin{align*}
    w_1 &= (2, 0, 0, 0, 0, 1),\\
    w_2 &= (0, 0, 2, 0, 0, 1),\\
    w_3 &= (0, 0, 0, 1, 1, 0).
\end{align*}
By adding zeroes for the coordinates corresponding to $D_{12},D_{13},D_{32}$, we get 
\begin{align*}
    v_1 &= (2, 0, 0, 0, 0, 0, 0, 0, 1),\\
    v_2 &= (0, 0, 0, 0, 2, 0, 0, 0, 1),\\
    v_3 &= (0, 0, 0, 0, 0, 1, 1, 0, 0).
\end{align*}
The motion $\underline{D}$ can be written as a linear combination of the form $\sum_{i = 1}^3 r_i v_i$, where $r_1, r_2, r_3 \in \R$. It is of the form
\begin{equation*}
    \underline{D} =(2 r_1, 0, 0, 0, 2 r_2, r_3, r_3, 0, r_1 + r_2).
\end{equation*}

A $2$-infinitesimal motion needs to satisfy 
    \begin{equation}\label{eqn:one-orth-pair-remaining-betas}
        \begin{split}
           \beta_2^A(D) t^2 &= - \frac{1}{2} a_{22}^2 (\sqrt{2}a_{21} (r_1-r_2) +a_{22}r_3)^2 t^2 \geq 0, \\
            \beta_3^A(D)t^2 &= - \frac{1}{2} a_{32}^2 (\sqrt{2}a_{31} (r_1-r_2) +a_{32}r_3)^2 t^2 \geq 0, \\
            \beta_2^B(D)t^2 &= - \frac{1}{2} b_{21}^2 (\sqrt{2}b_{22} (r_1-r_2)  - b_{21}r_3)^2 t^2 \geq 0, \\
            \beta_3^B(D)t^2 &= - \frac{1}{2} b_{31}^2 (\sqrt{2}b_{32} (r_1-r_2)  - b_{31}r_3)^2 t^2 \geq 0. \\
        \end{split}
    \end{equation}
From \eqref{eqn:one-orth-pair-remaining-betas}, we get the system of equations
\begin{equation} \label{eqn:one-orth-pair-system-of-ineqs}
    \begin{pmatrix}
            \sqrt{2}a_{21} & a_{22} \\
            \sqrt{2}a_{31} & a_{32} \\
            \sqrt{2}b_{22} & -b_{21} \\
            \sqrt{2}b_{32} & -b_{31}
        \end{pmatrix} \begin{pmatrix}
            r_1-r_2 \\
            r_3
        \end{pmatrix} = 0.
\end{equation}
The matrix in \eqref{eqn:one-orth-pair-system-of-ineqs} has rank $2$, since its submatrix $S = \begin{pmatrix}
    \sqrt{2}a_{21} & a_{22} \\
            \sqrt{2}a_{31} & a_{32}
\end{pmatrix}$
has determinant $\det(S) = \sqrt{2} \det(a_2, a_3)$ where $\det(a_2, a_3) \neq 0$ by assumption.
Hence, the only possible solution to the system \eqref{eqn:one-orth-pair-remaining-betas} is $r_1=r_2$ and $r_3 = 0$, and the only $2$-infinitesimal motions correspond to 
\begin{equation*}
    \underline{D} = (2 r_1, 0, 0, 0, 2 r_1, 0, 0, 0, 2 r_1), 
\end{equation*}
i.e., the factorization is $2$-infinitesimally rigid.

If \eqref{eqn:4-factors} does not hold, then  every nonzero $y$ satisfying $y^T \bar{C}_{(\mathcal{A},\mathcal{B})} = 0$ has  both positive and negative entries. 
By Farkas' lemma there must exist an element $\underline{D} \in \R^6$ such that $\bar{C}_{(\mathcal{A},\mathcal{B})}\underline{D} >0$.
Recall that by \cref{rmk:rank-1-factors-determinant-form} the determinants of the matrices $A^{(i)} + t \dot A^{(i)}$ and $B^{(j)}+ t \dot B^{(j)}$ are of the form $\alpha(D)t + \beta(D) t^2$. 
Under the assumption $D_{12} = D_{13} = D_{32} = 0$, we have $\alpha^A_1(D)t + \beta_1^A(D)t^2 = \alpha^B_1 (D)t + \beta_1^B(D)t^2 = 0$, and the inequality $\bar{C}_{(\mathcal{A},\mathcal{B})}\underline{D} > 0$ is equivalent to $\alpha^A_i (D) >0$ and $\alpha^B_j (D) >0$ for all $i,j \geq 2$ by the definition of $\bar{C}_{(\mathcal{A},\mathcal{B})}$.
It follows that $\alpha^A_i(D)t + \beta^A_i(D) t^2 \geq 0$ and $\alpha^B_j(D)t + \beta^B_j(D) t^2 \geq 0$ for all $t \in [0, \varepsilon)$ for $\varepsilon > 0$ small enough. 
Therefore $\underline{D}$ corresponds to a $2$-infinitesimal motion of the psd factorization.
It is not a scalar multiple of a $2$-trivial motion because if it were a $2$-trivial motion, then the truncated vector with the columns corresponding to $D_{12},D_{13}$ and $D_{32}$ removed would be in the right kernel of $\bar{C}_{(\mathcal{A},\mathcal{B})}$.
Hence $D$ is a non-trivial $2$-infinitesimal motion and the psd factorization is $2$-infinitesimally flexible. This proves the proposition for the special case when $a_1 = (1, 0)$ and $b_1 = (0,1)$.

In the general case,  by \cref{lemma:rotation-gives-psd-factorization}, we can find an orthogonal matrix $S \in O(2) \subset \mathrm{GL}(2)$, such that 
        \begin{align*}
            {\tilde a}_1 = S^Ta_1 = \begin{pmatrix}
                1 \\ 0
            \end{pmatrix} \mbox{ and }
            {\tilde b}_1 = S^{-1} b_1 = \begin{pmatrix}
                0 \\ 1
            \end{pmatrix}.
        \end{align*}
Then $\tilde A^{(i)} = S^T A^{(i)} S, \ \tilde B^{(j)} = S^{-1} B^{(j)} S^{-T}$ is a psd factorization of $M$. Moreover, there is a linear isomorphism between the $2$-infinitesimal motions of $A^{(1)},\ldots,B^{(3)}$ and the $2$-infinitesimal motions of $\tilde A^{(1)},\ldots, \tilde B^{(3)}$  by~\cref{cor:bijection-between-inf-motions}. In fact, this isomorphism restricts to identity in the case of $2$-infinitesimally rigid factorizations. 
Hence the psd factorization $A^{(1)},\ldots,B^{(3)}$ is $2$-infinitesimally rigid if and only if the psd factorization $\tilde A^{(1)},\ldots, \tilde B^{(3)}$ is $2$-infinitesimally rigid.

Finally, we note that the inequalities \eqref{eqn:4-factors} hold for $a_i,b_j$ if and only if they hold for $\tilde a_i, \tilde b_j$. Dot products are invariant under the elements of $O(2)$. Determinants can change sign, but if one determinant changes sign, then all determinants change signs, and since each of the inequalities has even number of determinants, then the sign of the inequality does not change. This completes the proof.
\end{proof}

\begin{lemma} \label{lem:one-orth-pair-six-factors}
     Let $A^{(1)},\dots , A^{(\bar{p})},B^{(1)}, \dots, B^{(\bar{q})} \in \mathcal{S}_+^2$, $\bar{p} + \bar{q}  \leq 6$,   $(\bar{p}, \bar{q}) \neq (3,3)$,  $\bar{p}, \bar{q} \geq 1$ be psd matrices of rank one with the form $A^{(i)} = a_ia_i^T$ and $B^{(j)} = b_jb_j^T$, and $a_1 = (1, 0)$ and $b_1 = (0,1)$. Assume $\det(a_i,a_j) \neq 0, \det(b_i,b_j) \neq 0$ for all $i,j$ and $\langle a_i,b_j \rangle \neq 0$ for   $(i,j)  \neq (1,1)$. 
     Then the cone $\bar{P}_{(\mathcal{A},\mathcal{B})} = \{\underline{D} \in \R^6 \mid \bar{C}_{(\mathcal{A},\mathcal{B})} \underline{D} \geq 0 \}$ is full-dimensional.     
\end{lemma}

\begin{proof}
By Farkas' lemma there exists a $y \geq 0, y \neq 0$ such that $y^T \bar{C}_{(\mathcal{A},\mathcal{B})} = 0$ if and only if there does not exist an element $\underline{D} \in \R^6$ such that $\bar{C}_{(\mathcal{A},\mathcal{B})}\underline{D} >0$.   First assume $\bar{p}+\bar{q}=6$.  As $\det(a_i,a_j) \neq 0, \det(b_i,b_j) \neq 0$ for all $i,j$ and $\langle a_i,b_j \rangle \neq 0$ for   $(i,j) \neq (1,1)$,  it follows from~\Cref{lem:matrix-C2} that $\rank(\bar{C}_{(\mathcal{A},\mathcal{B})}) = 3$ and the left kernel of $\bar{C}_{(\mathcal{A},\mathcal{B})}$ is spanned by the vector $v$ with coordinates given by~\eqref{eqn:left-kernel-C2}. For $\bar{p}+\bar{q}=6$ and $\bar{p} \in \{2,4\}$, we have computationally checked that there are no non-zero nonnegative left kernel vectors. Hence the cone $\bar{P}_{(\mathcal{A},\mathcal{B})}$ is full-dimensional.

Assume, then that $\bar p = 5$ and $\bar q = 1$. 
Suppose there exists some $y \geq 0, y \neq 0$ such that $y^T \bar{C}_{(\mathcal{A},\mathcal{B})} = 0$.
Then we have
\begin{equation*}
    y_1 a_{22}^2 + y_2 a_{32}^2 + y_3 a_{42}^2 + y_4 a_{52}^2 = 0.
\end{equation*}
This implies that at least one of the elements $a_{22}, a_{32}, a_{42}, a_{52}$ is zero, which is a contradiction with the assumptions. 
Thus there exists no such $y$ and the cone $\bar{C}_{(\mathcal{A},\mathcal{B})}$ is full-dimensional. A similar argument works for $\bar p = 1$ and $\bar q = 5$.

Finally, if $\bar p + \bar q < 6$, then it follows from the previous cases that the cone $\bar{P}_{(\mathcal{A},\mathcal{B})}$ is full-dimensional, because even after adding rank-one factors until there are in total six rank-one factors, the resulting cone is full-dimensional. 
\end{proof}

\begin{theorem} \label{thm:one-orthogonal-pair}
Let $M \in \M_{3,2}^{p \times q}$ and consider a size-$2$ psd factorization of $M$ given by the factors $A^{(1)}$, $\dots$, $A^{(p)}$, $B^{(1)}$, $\dots$, $B^{(q)} \in \mathcal{S}_+^2$. Let $A^{(1)}$, $\dots$, $A^{(\bar{p})}$, $B^{(1)}$, $\dots$, $B^{(\bar{q})}$, $\bar p, \bar q \geq 1$  be the collection of the factors that have rank one  with the form $A^{(i)} = a_ia_i^T$ and $B^{(j)} = b_jb_j^T$.
  Assume that $\langle a_{i_1},b_{j_1} \rangle =0$ for $(i_1,j_1) \in [\bar{p}] \times [\bar{q}]$,  $\det(a_i,a_j) \neq 0, \det(b_i,b_j) \neq 0$ for all $(i,j)$ and $\langle a_i,b_j \rangle \neq 0$ for  $(i,j) \in [\bar{p}] \times [\bar{q}] \backslash (i_1,j_1)$. 
The factorization is $2$-infinitesimally rigid if and only if there exist  factors of rank one $A^{(i_1)},A^{(i_2)},A^{(i_3)},B^{(j_1)},B^{(j_2)},B^{(j_3)}$ such that 
\begin{equation}\label{eqn:6-factors}
\begin{split}
-\frac{\det(b_{j_1},b_{j_3}) \det(b_{j_2},b_{j_3}) \langle a_{i_1},b_{j_3} \rangle \langle a_{i_3},b_{j_3} \rangle}{\det(a_{i_1},a_{i_2}) \det(a_{i_2},a_{i_3}) \langle a_{i_2},b_{j_1} \rangle \langle a_{i_2},b_{j_2} \rangle} >0, \\
\frac{\det(b_{j_1},b_{j_3}) \det(b_{j_2},b_{j_3}) \langle a_{i_1},b_{j_3} \rangle \langle a_{i_2},b_{j_3} \rangle}{\det(a_{i_1},a_{i_3}) \det(a_{i_2},a_{i_3}) \langle a_{i_3},b_{j_1} \rangle \langle a_{i_3},b_{j_2} \rangle} >0,\\
-\frac{\det(b_{j_1},b_{j_3}) \langle a_{i_1},b_{j_3} \rangle\langle a_{i_2},b_{j_3} \rangle \langle a_{i_3},b_{j_3} \rangle}{\det(b_{j_1},b_{j_2}) \langle a_{i_1},b_{j_2} \rangle \langle a_{i_2},b_{j_2} \rangle \langle a_{i_3},b_{j_2} \rangle} >0.
\end{split}
\end{equation}
\end{theorem}

\begin{proof}
  Since by assumption all $\det(a_i,a_j) \neq 0, \det(b_i,b_j) \neq 0$, then by~\Cref{lemma:rank_three_matrix} the submatrix of $M$ corresponding to rows $i_1,i_2,i_3$ and columns $j_1,j_2,j_3$ is of rank three and we are in the setting of~\Cref{prop:one-orth}.  If the inequalities \eqref{eqn:6-factors} are satisfied, then   by~\Cref{prop:one-orth} the psd factorization $A^{(i_1)},A^{(i_2)},A^{(i_3)},B^{(j_1)},B^{(j_2)},B^{(j_3)}$ is 2-infinitesimally rigid and  by~\Cref{lem:sub-facorization-rigid} the psd factorization $A^{(1)}$, $\dots$, $A^{(p)}$, $B^{(1)}$, $\dots$, $B^{(q)}$   is 2-infinitesimally rigid.

Assume that there do not exist factors of rank one $A^{(i_1)}, A^{(i_2)}, A^{(i_3)}, B^{(j_1)}, A^{(j_2)}, A^{(j_3)}$ such that the inequalities \eqref{eqn:6-factors} are satisfied. Let $S \in O(2)$ be such that 
        \begin{align*}
            S^Ta_1 = \begin{pmatrix}
                1 \\ 0
            \end{pmatrix} \mbox{ and }
            S^{-1}b_1 = \begin{pmatrix}
                0 \\ 1
            \end{pmatrix}.
        \end{align*}
Then $\tilde A^{(i)} = S^T A^{(i)} S, \ \tilde B^{(j)} = S^{-1} B^{(j)} S^{-T}$ is a psd factorization of $M$. Consider the $(\bar{p} + \bar{q}   -2 ) \times 6$ matrix   $\bar{C}_{(\tilde{\mathcal{A}},\tilde{\mathcal{B}})}$  corresponding to the rank-one factors  $\tilde A^{(i)},\tilde B^{(j)}$ as in~\Cref{lem:matrix-C2}. We will show that the cone $\bar{P}_{(\tilde{\mathcal{A}},\tilde{\mathcal{B}})} = \{\underline{D} \in \R^6 \mid \bar{C}_{(\tilde{\mathcal{A}},\tilde{\mathcal{B}})} \underline{D} \geq 0 \}$ is full-dimensional. 
Let $\bar{C}_{(\tilde{\mathcal{A}},\tilde{\mathcal{B}})}'$ be the $(\bar{p} + \bar{q}    -2 ) \times 3$ submatrix of $\bar{C}_{(\tilde{\mathcal{A}},\tilde{\mathcal{B}})}$ where only the linearly independent columns $1,2,4$ of $\bar{C}_{(\tilde{\mathcal{A}},\tilde{\mathcal{B}})}$ are kept. The cone $\bar{P}_{(\tilde{\mathcal{A}},\tilde{\mathcal{B}})}$ is full-dimensional if and only if the cone $\bar{P}_{(\tilde{\mathcal{A}},\tilde{\mathcal{B}})}' = \{\underline{D} \in \R^3 \mid \bar{C}_{(\tilde{\mathcal{A}},\tilde{\mathcal{B}})}' \underline{D} \geq 0 \}$ is full-dimensional. 

By contradiction, assume that the cone $\bar{P}_{(\tilde{\mathcal{A}},\tilde{\mathcal{B}})}'$ is not full-dimensional. This means that $(0,0,0)$ can be written as a non-trivial conic combination of the rows of the matrix $\bar{C}_{(\tilde{\mathcal{A}},\tilde{\mathcal{B}})}'$. By Caratheodory’s theorem, the zero vector can be written as a non-trivial conic combination of at most  four  rows of the matrix $\bar{C}_{(\tilde{\mathcal{A}},\tilde{\mathcal{B}})}'$.   Hence if the cone $\bar{P}_{(\tilde{\mathcal{A}},\tilde{\mathcal{B}})}'$ is not full-dimensional, then there exist at most  four  rows of $\bar{C}_{(\tilde{\mathcal{A}},\tilde{\mathcal{B}})}'$ such that the cone defined by these rows is not full-dimensional.

We will show that no matter how many of the at most four rows correspond to $A$ and $B$ factors,   then the cones defined by the corresponding rows of $\bar{C}_{(\tilde{\mathcal{A}},\tilde{\mathcal{B}})}'$ and $\bar{C}_{(\tilde{\mathcal{A}},\tilde{\mathcal{B}})}$ are full-dimensional and hence by the previous paragraph the cones $\bar{P}_{(\tilde{\mathcal{A}},\tilde{\mathcal{B}})}'$ and $\bar{P}_{(\tilde{\mathcal{A}},\tilde{\mathcal{B}})}$ are full-dimensional.  If two rows correspond to $\tilde{A}^{(i_2)},\tilde{A}^{(i_3)}$ factors and two rows correspond to $\tilde{B}^{(j_2)},\tilde{B}^{(j_3)}$ factors, then by~\Cref{lemma:rank_three_matrix} the submatrix of $M$ corresponding to rows $1,i_2,i_3$ and columns $1,j_2,j_3$ is of rank three and we are in the setting of~\Cref{prop:one-orth}. 
Since the inequalities \eqref{eqn:6-factors} are not satisfied  then the cone corresponding to the factors $\tilde{A}^{(i_2)},\tilde{A}^{(i_3)}$, $\tilde{B}^{(j_2)},\tilde{B}^{(j_3)}$ is full-dimensional as shown in the proof of~\cref{prop:one-orth}.  In all other cases, the cone $\bar{P}_{(\tilde{\mathcal{A}},\tilde{\mathcal{B}})}$ is full-dimensional by~\Cref{lem:one-orth-pair-six-factors}. \\

Hence, there exists an interior point $\underline{D} \in \bar{P}_{(\tilde{\mathcal{A}},\tilde{\mathcal{B}})}$.
In particular, $\underline{D}$ is such that $\alpha_i^A(D)>0$ and $\alpha_j^B(D)>0$ for all $\tilde A^{(i)}$ and $\tilde B^{(j)}$ defining the cone $\bar{P}_{(\tilde{\mathcal{A}},\tilde{\mathcal{B}})}$.
Then $\underline{D}$ is not a $2$-trivial motion by an argument similar to the one in the proof of~\cref{prop:one-orth}.

Hence, the psd factorization consisting of rank-$1$ factors is not $2$-infinitesimally rigid, and therefore the full factorization is not $2$-infinitesimally rigid by~\Cref{lem:sub-facorization-rigid}.

Finally, since there is a bijection between the $2$-infinitesimal motions of $(A^{(1)}$, $\dots$, $A^{(p)}$, $B^{(1)}$, $\dots$, $B^{(q)})$ and $(\tilde A^{(1)}$, $\dots$, $\tilde A^{(p)}$, $\tilde B^{(1)}$, $\dots$, $\tilde B^{(q)})$, the factorization $(A^{(1)}$, $\dots$, $A^{(p)}$, $B^{(1)}$, $\dots$, $B^{(q)})$ is $2$-infinitesimally flexible.
\end{proof}

\subsection{Two orthogonal pairs} \label{sec:two-orthogonal-pairs}

\begin{theorem}\label{thm:two-orth-inf-rigid}
    Let $M\in \M_{3,2}^{p \times q}$ and consider a size-$2$ psd factorization of $M$ given by the factors $A^{(1)}$, $\dots$, $A^{(p)}$, $B^{(1)}$, $\dots$, $B^{(q)} \in \mathcal{S}^2_+$.  Let $A^{(1)}$, $\dots$, $A^{(\bar{p})}$, $B^{(1)}$, $\dots$, $B^{(\bar{q})}$  be the collection of the factors that have rank one  with the form $A^{(i)} = a_ia_i^T$ and $B^{(j)} = b_jb_j^T$. Assume $\bar p, \bar q \geq 2$ and  $\langle a_{1}, b_{1} \rangle = \langle a_{2}, b_{2} \rangle = 0$
    and $\det(a_i,a_j) \neq 0, \det(b_i,b_j) \neq 0$ for all $i,j$. 
    Then the psd factorization is $2$-infinitesimally rigid if and only if $\bar p,\bar q \geq 3$.
\end{theorem}

\begin{proof}
First we prove the claim in the case when $a_1 = (a_{11}, 0)$ and $b_1 = (0,b_{12})$.
 Motivated by~\Cref{lem:rank-2-factors-do-not-matter}, we consider constraints on a 2-infinitesimal motion $D$ imposed by the rank-one factors $A^{(1)}$, $\dots$, $A^{(\bar{p})}$, $B^{(1)}$, $\dots$, $B^{(\bar{q})}$.\color{black}
 From the diagonals and determinants of the matrices $A^{(1)} + t \dot A^{(1)}$ and $B^{(1)} + t \dot B^{(1)}$ we conclude that $D_{12} = D_{13} = D_{32} = 0$ as in the proof of \cref{prop:one-orth}.

From here, we consider the factors under a $2$-infinitesimal motion. 
A $2$-infinitesimal motion of the factorization needs to satisfy
\begin{align*}
    T^2([A^{(i)}+t \dot A^{(i)}]_{\{1,2\}}) &=\det(A^{(i)}+ t\dot A^{(i)}) = \alpha^A_i(D)t + \beta^A_i(D)t^2 \geq 0,\\
    T^2([B^{(j)}+t \dot B^{(j)}]_{\{1,2\}}) &= \det(B^{(j)}+ t\dot B^{(j)}) = \alpha^B_j(D)t + \beta^B_j(D)t^2 \geq 0,
\end{align*}
for all $i \in [p]$ and $j \in [q]$ and for $t \in [0,\varepsilon)$ and $\varepsilon > 0$ small enough. 
Using the orthogonality of $a_2$ and $b_2$,
we make the substitution $b_{22}=-\frac{b_{21}a_{21}}{a_{22}}$. 
Then we have  \begin{align*}
    \alpha^A_2  &= \left(
\begin{array}{ccccccccc}
 a_{21}^2 a_{22}^2 & a_{21}^4 & -\sqrt{2} a_{21}^3 a_{22} & a_{22}^4 & a_{21}^2 a_{22}^2 & -\sqrt{2} a_{21} a_{22}^3 & \sqrt{2} a_{21} a_{22}^3 & \sqrt{2} a_{21}^3 a_{22} & -2 a_{21}^2 a_{22}^2 
\end{array}
\right),\\
    \alpha^B_2 &= \left(
\begin{array}{ccccccccc}
 -b_{21}^2 b_{22}^2 & -b_{22}^4 & -\sqrt{2} b_{21} b_{22}^3 & -b_{21}^4 & -b_{21}^2 b_{22}^2 & -\sqrt{2} b_{21}^3 b_{22} & \sqrt{2} b_{21}^3 b_{22} & \sqrt{2} b_{21} b_{22}^3 & 2 b_{21}^2 b_{22}^2 
\end{array}
\right)
    \\ &= \left(
\begin{array}{ccccccccc}
 -\frac{a_{21}^2 b_{21}^4}{a_{22}^2} & -\frac{a_{21}^4 b_{21}^4}{a_{22}^4} & \frac{\sqrt{2} a_{21}^3 b_{21}^4}{a_{22}^3} & -b_{21}^4 & -\frac{a_{21}^2 b_{21}^4}{a_{22}^2} & \frac{\sqrt{2} a_{21} b_{21}^4}{a_{22}} & -\frac{\sqrt{2} a_{21} b_{21}^4}{a_{22}} & -\frac{\sqrt{2} a_{21}^3 b_{21}^4}{a_{22}^3} & \frac{2 a_{21}^2 b_{21}^4}{a_{22}^2}
\end{array}
\right).\end{align*} 
Now observe $\alpha^A_2=-\frac{a_{22}^4}{b_{21}^4}\cdot \alpha^B_2$. Then in order for both $\alpha^A_2(D) \geq 0$ and $\alpha^B_2(D) \geq 0$,  $\alpha^A_2(D) = \alpha^B_2(D) = 0$. We add this to our constraints, along with $D_{12} = D_{13}=D_{32}=0$. 
The system of equations corresponding to these constraints is represented by the matrix 
\begin{equation}\label{eqn:one-orth-constraint-matrix}
    \left(
\begin{array}{ccccccccc}
 0 & 1 & 0 & 0 & 0 & 0 & 0 & 0 & 0 \\
 0 & 0 & 1 & 0 & 0 & 0 & 0 & 0 & 0 \\
 0 & 0 & 0 & 0 & 0 & 0 & 0 & 1 & 0 \\
 a_{21}^2 a_{22}^2 & a_{21}^4 & -\sqrt{2} a_{21}^3 a_{22} & a_{22}^4 & a_{21}^2 a_{22}^2 & -\sqrt{2} a_{21} a_{22}^3 & \sqrt{2} a_{21} a_{22}^3 & \sqrt{2} a_{21}^3 a_{22} & -2 a_{21}^2 a_{22}^2 \\
\end{array}
\right).
\end{equation}
In particular, the $2$-infinitesimal motions are in the right kernel of the matrix in \eqref{eqn:one-orth-constraint-matrix}.
The kernel of the matrix of constraints 
is 5-dimensional with a basis given by the rows of the matrix
\begin{equation}\label{eqn:a2-d12-d13-d32-basis}
    \left(
\begin{array}{ccccccccc}
 2 & 0 & 0 & 0 & 0 & 0 & 0 & 0 & 1 \\
 -\frac{\sqrt{2} a_{22}}{a_{21}} & 0 & 0 & 0 & 0 & 0 & 1 & 0 & 0 \\
 \frac{\sqrt{2} a_{22}}{a_{21}} & 0 & 0 & 0 & 0 & 1 & 0 & 0 & 0 \\
 -1 & 0 & 0 & 0 & 1 & 0 & 0 & 0 & 0 \\
 -\frac{a_{22}^2}{a_{21}^2} & 0 & 0 & 1 & 0 & 0 & 0 & 0 & 0 \\
\end{array}
\right).
\end{equation}

We now consider a motion $R$, formed by a linear combination of the rows of the matrix in \eqref{eqn:a2-d12-d13-d32-basis}. 
Let $r_l$ be the coefficient corresponding to row $l$. 
This gives us 
\begin{equation*}
    R = \left(-\frac{a_{22}^2 r_5}{a_{21}^2}+\frac{\sqrt{2} a_{22} \left(r_3-r_2\right)}{a_{21}}+2 r_1-r_4,0,0,r_5,r_4,r_3,r_2,0,r_1\right).
\end{equation*} 

Since $\alpha^A_2(R) = \alpha^B_2(R)=0$, we check $\beta^A_2(R)$ and $\beta^B_2(R)$ and find we have terms quadratic in each $r_l$. These two quadratics are represented by the $5\times 5$ matrices as shown below:
\begin{align*}
    \beta^A_2(R)&:\left(
\begin{array}{ccccc}
 -a_{21}^2 a_{22}^2 & 0 & -\frac{a_{21} a_{22}^3}{\sqrt{2}} & a_{21}^2 a_{22}^2 & 0 \\
 0 & 0 & 0 & 0 & 0 \\
 -\frac{a_{21} a_{22}^3}{\sqrt{2}} & 0 & -\frac{a_{22}^4}{2} & \frac{a_{21} a_{22}^3}{\sqrt{2}} & 0 \\
 a_{21}^2 a_{22}^2 & 0 & \frac{a_{21} a_{22}^3}{\sqrt{2}} & -a_{21}^2 a_{22}^2 & 0 \\
 0 & 0 & 0 & 0 & 0 \\
\end{array}
\right),\\
\beta^B_2(R)&:\left(
\begin{array}{ccccc}
 -\frac{a_{21}^2 b_{21}^4}{a_{22}^2} & \frac{a_{21} b_{21}^4}{\sqrt{2} a_{22}} & -\frac{\sqrt{2} a_{21} b_{21}^4}{a_{22}} & \frac{a_{21}^2 b_{21}^4}{a_{22}^2} & b_{21}^4 \\
 \frac{a_{21} b_{21}^4}{\sqrt{2} a_{22}} & -\frac{b_{21}^4}{2} & b_{21}^4 & -\frac{a_{21} b_{21}^4}{\sqrt{2} a_{22}} & -\frac{a_{22} b_{21}^4}{\sqrt{2} a_{21}} \\
 -\frac{\sqrt{2} a_{21} b_{21}^4}{a_{22}} & b_{21}^4 & -2 b_{21}^4 & \frac{\sqrt{2} a_{21} b_{21}^4}{a_{22}} & \frac{\sqrt{2} a_{22} b_{21}^4}{a_{21}} \\
 \frac{a_{21}^2 b_{21}^4}{a_{22}^2} & -\frac{a_{21} b_{21}^4}{\sqrt{2} a_{22}} & \frac{\sqrt{2} a_{21} b_{21}^4}{a_{22}} & -\frac{a_{21}^2 b_{21}^4}{a_{22}^2} & -b_{21}^4 \\
 b_{21}^4 & -\frac{a_{22} b_{21}^4}{\sqrt{2} a_{21}} & \frac{\sqrt{2} a_{22} b_{21}^4}{a_{21}} & -b_{21}^4 & -\frac{a_{22}^2 b_{21}^4}{a_{21}^2} \\
\end{array}
\right).
\end{align*} 

By \cref{prop:zero-alpha-forces-beta} 
these quadratic forms $\beta^A_2(R)$ and $\beta^B_2(R)$ are  zero  for any values of $r_1$, $r_2$, $r_3$, $r_4$, $r_5$. 
We now check for values that will make both of these terms zero. 
Each of these matrices has a 4-dimensional kernel. The intersection of these kernels is 3-dimensional with a basis given by the rows of the following matrix: \begin{equation}\label{eqn:q-f-nullspace}
    \left(
\begin{array}{ccccc}
 0 & -\frac{\sqrt{2} a_{22}}{a_{21}} & 0 & 0 & 1 \\
 1 & 0 & 0 & 1 & 0 \\
 -\frac{a_{22}}{\sqrt{2} a_{21}} & 1 & 1 & 0 & 0 \\
\end{array}
\right).
\end{equation}
To find values of $r_1$, $r_2$, $r_3$, $r_4$, $r_5$ such that $\beta^A_2(R)=\beta^B_2(R)=0$, we form $S$ given by a linear combination of the rows of this matrix. Let $s_1,s_2,s_3$ be the coefficient corresponding to rows $1,2,3$, respectively. This gives us: \begin{equation*}
    S = \left(s_2-\frac{a_{22} s_3}{\sqrt{2} a_{21}},s_3-\frac{\sqrt{2} a_{22} s_1}{a_{21}},s_3,s_2,s_1\right).
\end{equation*} 

We find our new motion, $R'$ in terms of $S$: \begin{equation*}
    R' = \left(\frac{a_{22}^2 s_1}{a_{21}^2}-\frac{\sqrt{2} a_{22} s_3}{a_{21}}+s_2,0,0,s_1,s_2,s_3,s_3-\frac{\sqrt{2} a_{22} s_1}{a_{21}},0,s_2-\frac{a_{22} s_3}{\sqrt{2} a_{21}}\right).
\end{equation*}

If $\bar p = \bar q = 2$, then any assignment of $s_1,s_2,s_3$ gives a 2-infinitesimal motion and the psd factorization is not 2-infinitesimally rigid.

If at least one of $\bar p >2$ or $\bar q >2$, then for a 2-infinitesimal motion we need to ensure that 
$\det(A^{(i)}+ t\dot A^{(i)})\geq 0$ for $i=3,\dots, \bar p$, and that $\det(B^{(j)}+ t\dot B^{(j)})\geq 0$ for $j=3, \dots, \bar q$.
We start with considering the $\alpha$-terms: \begin{align*}
    \alpha^A_i(R') &= s_1\frac{a_{i2}^2 \left(a_{22} a_{i1}-a_{21} a_{i2}\right)^2}{a_{21}^2}\\
    \alpha^B_j(R') &= -s_1\frac{b_{j1}^2 \left(a_{21} b_{j1}+a_{22} b_{j2}\right){}^2}{a_{21}^2}.
\end{align*}

First, we consider the case if $\bar p=2$ or $\bar q=2$. Without loss of generality, we assume that $\bar p \geq 3$ and $\bar q = 2$. We note that $\alpha^A_i(R')$ for $i=3,\dots, \bar p$ contain all squared terms with the exception of $s_1$. Hence if $s_1 \geq 0$, then $\alpha^A_i(R') \geq 0$ for $i=3,\dots, \bar p$. Moreover, if $s_1 > 0$, then $\alpha^A_i(R') >0$ for $i=3,\dots, \bar p$. If $\alpha^A_i(R') >0$, then also $\det(A^{(i)}+ t\dot A^{(i)})\geq 0$ for $t>0$ small enough. Hence any assignment of $s_1>0$ and arbitrary $s_2,s_3$ gives a $2$-infinitesimal motion and hence the psd factorization is not 2-infinitesimally rigid.

If both $\bar p \geq 3$ and $\bar q \geq 3$, we look at the determinants $\det(A^{(3)}+ t\dot A^{(3)})$  and  $\det(B^{(3)}+ t\dot B^{(3)})$. Since $\alpha^A_3(R')$ and $\alpha^B_3(R')$ contain all squared terms with the exception of $s_1$ with opposite signs, we need at least one of these to be zero. 
We consider the terms of $\alpha^A_3(R')$ first: neither $a_{32}$ nor $(a_{22} a_{31}-a_{21} a_{32})$ can be zero, otherwise we have dependent factors. 
We turn our attention to $\alpha^B_3(R')$ and find a similar argument: neither $b_{31}^2$ nor $(a_{21} b_{31}+a_{22} b_{32})$ can be zero otherwise $\det(b_1, b_3) = 0$ or $\det(b_2, b_3) = 0$ because $\langle a_2,b_3 \rangle=0$ together with $\langle a_2,b_2 \rangle=0$ implies $\det(b_2, b_3) = 0$. 
We establish that $s_1=0$. 
We update our motion with this new constraint: \begin{equation*}
    R'' = \left(s_2-\frac{\sqrt{2} a_{22} s_3}{a_{21}},0,0,0,s_2,s_3,s_3,0,s_2-\frac{a_{22} s_3}{\sqrt{2} a_{21}}\right).
\end{equation*}

Now since $\alpha^A_3(R'')=\alpha^B_3(R'')=0$, we check the last terms: \begin{align*}
    \beta^A_3(R'')&= -s_3^2\frac{a_{32}^2 \left(a_{22} a_{31}-a_{21} a_{32}\right)^2}{2 a_{21}^2},\\
    \beta^B_3(R'')&= -s_3^2\frac{b_{31}^2 \left(a_{21} b_{31}+a_{22} b_{32}\right)^2}{2 a_{21}^2}.
\end{align*}

We observe both of these are negative with all squared terms, so we look to make them equal to zero. We can only do this by making $s_3=0$. 
What we have left is a $2$-trivial motion: 
\begin{equation*}
    R''' = \left(s_2,0,0,0,s_2,0,0,0,s_2\right).
\end{equation*}
Therefore, by \cref{lem:rank-2-factors-do-not-matter} the psd factorization of $M$ given by the factors $A^{(1)}$, $\dots$, $A^{(p)}$, $B^{(1)}$, $\dots$, $B^{(q)}$ is 2-infinitesimally rigid.

For the general case, we proceed as in
\cref{prop:one-orth} by using 
an orthogonal matrix $S \in O(2)$, such that 
        \begin{align*}
            S^Ta_1 = \begin{pmatrix}
                1 \\ 0
            \end{pmatrix}, \mbox{ and }
            S^{-1} b_1 = \begin{pmatrix}
                0 \\ 1
            \end{pmatrix}.
        \end{align*}
This leads to the conclusion 
that $(A^{(1)},A^{(2)},A^{(3)},B^{(1)},B^{(2)},B^{(3)})$ is $2$-infinitesimally rigid and therefore 
the entire factorization
$(A^{(1)},\ldots,A^{(p)},B^{(1)}, \ldots, B^{(q)})$ is $2$-infinitesimally rigid. 
\end{proof}

\section{Rigidity and boundaries} \label{sec:rigidity-and-boundaries}
Now we turn to the relationship between unique psd factorizations and $s$-infinitesimal rigidity. The uniqueness is only up to the action of $\operatorname{GL}(k)$, and following the literature in rigidity theory we will introduce two related concepts: local and global rigidity of psd factorizations. It is easily seen that the latter implies the former, but the converse is not true in general. However, we will prove that they are 
equivalent in the case of matrices in $\M_{3,2}^{p \times q}$.  First  we will show that $1$-infinitesimal rigidity implies local rigidity  (\Cref{lem:inf-rigidity-implies-local-rigidity}). Then for positive matrices in $\M_{3,2}^{p \times q}$ we will argue that $2$-infinitesimal rigidity implies local rigidity. All these will culminate in \Cref{thm:equivalence-of-rigidities}. 

In the rest of the paper, we denote by $\mathcal{SF}(M)$ the space of all size-$k$ psd factorizations of 
$M \in \M_{\binom{k+1}{2}, k}^{p \times q}$. For $(A^{(1)},\ldots,B^{(q)}) \in \mathcal{SF}(M)$ and for $S \in GL(k)$, we refer to psd factorizations of the form $(S^T A^{(1)} S,\ldots,S^{-1} B^{(q)} S^{-T})$ as the orbit of $(A^{(1)},\ldots,B^{(q)})$.

\subsection{1-Infinitesimal rigidity implies local rigidity}

In this subsection, we define local rigidity and show that 1-infinitesimal rigidity implies local rigidity for matrices of psd rank $k$ and rank  $\binom{k+1}{2}$. The main result of the subsection is~\Cref{lem:inf-rigidity-implies-local-rigidity}. Lemmas~\ref{claim1}-\ref{claim3} follow closely~\cite[Section 4.2]{connelly1996second}. For the sake of completeness, we include all the details.

\begin{definition}
A size-$k$ psd factorization $(A^{(1)},\ldots,A^{(p)},B^{(1)},\ldots,B^{(q)})$ of a matrix $M \in \R_+^{p \times q}$ is called {\it locally rigid} if 
there exists a neighborhood of $(A^{(1)},\ldots,A^{(p)}$, $B^{(1)}, \ldots,B^{(q)})$ in $\mathcal{SF}(M)$ such that any other psd factorization of $M$ in this neighborhood 
is in the orbit of $(A^{(1)},\ldots,A^{(p)},B^{(1)},\ldots,B^{(q)})$.
\end{definition}

First we need to introduce some definitions used in the proofs. We recall 
from Section \ref{section:1-trivial-motions} that $\mathcal{M}_{\text{GL}(k)}$ is the variety that corresponds to the $\operatorname{GL}(k)$-action on size-$k$ psd factorizations and its elements are called rigid motions.

\begin{definition}
A trivial flex is a path $T(t)$, where $t \in [0,1]$, such that $T(0)=I$, each $T(t) \in \mathcal{M}_{\text{GL}(k)}$ is a rigid motion and $T(t)$ is an analytic function of $t$.
\end{definition}

\begin{definition}
The matrices $C',C'',\ldots,C^{(m)} \in \mathbb{R}^{\binom{k+1}{2} \times \binom{k+1}{2}}$ form an $m$-trivial flex if there exists a trivial flex $T(t)$  satisfying 
 \begin{equation*}
 \left(\frac{d}{dt}\right)^{\ell}  T(t)  |_{t=0} = C^{(\ell)}   \text{ for } \ell \in \{1,2,\ldots,m\}.
 \end{equation*}
 \end{definition}

Next we prove a sequence of three lemmas which will be used in the proof of Proposition \ref{prop:construction-of-nontrivial-1-inf-motion}. 

\begin{lemma} \label{claim1}
Let $C(t)$ for $ t \in [0,1]$ be an analytic path with $C(0)=I$ and let  $C^{(\ell)}:=\left(\frac{d}{dt}\right)^{\ell}  C(t) |_{t=0}$. If $C',C'',\ldots,C^{(m)}$ is an $m$-trivial flex, then there exists a trivial flex $T(t)$ such that
\[
\left(\frac{d}{dt}\right)^{\ell} \left(C(t)T(t) \right)|_{t=0} = 0 \text{ for } \ell \in \{1,\ldots,m\}.
\]
\end{lemma}

\begin{proof}
The variety $\mathcal{M}_{\text{GL(k)}}$ forms a Lie group whose Lie algebra is the tangent space at the identity. The exponential map 
\[
\mathcal{D} \mapsto e^{\mathcal{D}}
\]
takes the Lie algebra to the Lie group. The exponential map is an analytic diffeomorphism in some neighborhood of the identity. Therefore any analytic path $\mathcal{U}(t)$  in the Lie   with $\mathcal{U}(0)=I$  pulls back to an analytic path $\mathcal{D}(t)$ in the Lie algebra. Therefore $\mathcal{U}(t)=e^{\mathcal{D}(t)}$. Since
\[
I= \mathcal{U}(0) = e^{\mathcal{D}(0)},
\]
then $\mathcal{D}(0)=0$. This observation together with $\mathcal{D}(t)$ being an analytic path allows us to write
\[
\mathcal{D}(t) = tD_1+\frac{t^2}{2}D_2 + \frac{t^3}{6} D_3 + \cdots,
\]
where each $D_i$ belongs to the Lie algebra. Then the Taylor series for the matrix exponential at the zero matrix gives
\[
\mathcal{U}(t) = I + \left(tD_1+\frac{t^2}{2}D_2 + \frac{t^3}{6} D_3 + \cdots\right) + \frac{1}{2}\left(tD_1+\frac{t^2}{2}D_2 + \frac{t^3}{6} D_3 + \cdots\right)^2 + \cdots.
\]
Since this power series is absolutely covergent, we are allowed to rearrange the terms:
\begin{equation} \label{eqn:power-series}
\mathcal{U}(t) = I + tD_1 + \frac{t^2}{2}(D_2+D_1^2) + \frac{t^3}{6}(D_3+\frac{3}{2} D_1 D_2 + \frac{3}{2} D_2 D_1 + D_1^3) + \cdots.
\end{equation}
The $\ell$-th derivative of $\mathcal{U}(t)$ evaluated at $t=0$ is equal to the coefficient of $\frac{t^{\ell}}{\ell!}$ in the above power series.  Therefore the first $m$ coefficients form an $m$-trivial flex. 

We are now ready to prove the statement of the lemma that there exists a trivial flex $T(t)$ such that the first $m$ derivatives $\left(\frac{d}{dt}\right)^{\ell} \left(C(t)T(t) \right)|_{t=0}$ are zero. We prove this using  induction on $m$. For $m=0$, there is nothing to prove since no derivatives have to be zero.
We have $C(0) = I$ and we assume 
\[
\left(\frac{d}{dt}\right)^{\ell} \left(C(t) \right)|_{t=0} = 0 \text{ for } \ell \in \{1,\ldots,m-1\}.
\]
 Note that $C(t)$ in the induction hypothesis above is not the original $C(t)$, but 
it is the one already multiplied with a trivial flex to make the first $m-1$ derivatives equal to zero. 
We will show that there exists $T(t)$ such that $\left(\frac{d}{dt}\right)^{\ell} \left(C(t)T(t) \right)|_{t=0} = 0$ for all $\ell \in \{1, \ldots, m\}$. 

Since $C',C'',\ldots,C^{(m)}$ form an $m$-trivial flex, there exists an analytic path $\mathcal{U}(t)$ in the Lie group which can be expressed as in~\eqref{eqn:power-series} with the coefficients of the first $m$ terms equal to $C',C'',\ldots,C^{(m)}$. Since the first $m-1$ derivatives are zero, by~\eqref{eqn:power-series} we have $D_{\ell}=0$ for $\ell \in \{1,\ldots,m-1\}$ and  $\left(\frac{d}{dt}\right)^{m} \left(C(t) \right) = D_m$,  which is in the Lie algebra.  We define
\[
T(t):=e^{-\frac{t^m}{m!}D_m}= I - \frac{t^m}{m!}D_m + \cdots .
\]
By general Leibniz rule, we get
\[
\left(\frac{d}{dt}\right)^{\ell}(C(t)T(t)) = \sum_{\alpha = 0}^{\ell} \binom{l}{\alpha} \left(\frac{d}{dt}\right)^{\ell-\alpha}C(t) \left(\frac{d}{dt}\right)^{\alpha}T(t).
\]
Since
\[
\left(\frac{d}{dt}\right)^{\alpha}T(t)|_{t=0} = \begin{cases}
0 & \alpha=1,\ldots,m-1\\
-D_{m} & \alpha =m
\end{cases},
\]
we get
\[
\left(\frac{d}{dt}\right)^{\ell} \left(C(t)T(t) \right)|_{t=0}  = 
\begin{cases}
0 & \ell \in\{1,\ldots,m-1\}\\
D_m - D_m =0 & \ell=m
\end{cases}.
\]
\end{proof}

\begin{lemma} \label{claim2}
Let $C',\ldots,C^{(m)}$ be such that $C'=\ldots=C^{(m-1)}=0$. Then $C^{(m)}$ is a 1-trivial flex if and only if $C',\ldots,C^{(m)}$ is an $m$-trivial flex. 
\end{lemma}

\begin{proof}
This follows from the definition of an $m$-trivial flex and the equation~(\ref{eqn:power-series}).
\end{proof}

\begin{lemma} \label{claim3}
Let $C(t)$ be an analytic path such that  $C(0) = I$ and  $C' = \left(\frac{d}{dt}\right)  C(t)  |_{t=0},\ldots$, $C^{(m)} = \left(\frac{d}{dt}\right)^{m}  C(t)  |_{t=0}$. If $T(t)$ is a trivial flex, then $\left(\frac{d}{dt}\right)  C(t) T(t) |_{t=0}$, $\ldots$, $\left(\frac{d}{dt}\right)^{m}  C(t) T(t) |_{t=0}$ is an $m$-trivial flex if and only if $C',\ldots,C^{(m)}$ is an $m$-trivial flex.
\end{lemma}
\begin{proof}
It is enough to show only one direction of the statement since if $T(t)$ is a trivial flex, then the same is true for $T(t)^{-1}$.
Assume $\left(\frac{d}{dt}\right)  C(t) T(t) |_{t=0}$, $\ldots, \left(\frac{d}{dt}\right)^{m}  C(t) T(t) |_{t=0}$ is an $m$-trivial flex. Then there exists a trivial flex $U(t)$  such that 
\[
\left(\frac{d}{dt}\right)^{\ell} \left(C(t)T(t) \right)|_{t=0} = \left(\frac{d}{dt}\right)^{\ell} \left(U(t) \right)|_{t=0} \text{ for } \ell \in \{1,\ldots,m\}.
\]
Multiplying both $C(t)T(t)$ and $U(t)$ by $T(t)^{-1}$ from the right and applying the general Leibniz rule, gives
\[
\left(\frac{d}{dt}\right)^{\ell} \left(C(t) \right)|_{t=0} = \left(\frac{d}{dt}\right)^{\ell} \left(U(t) T(t)^{-1} \right)|_{t=0} \text{ for } \ell \in \{1,\ldots,m\}.
\]
On the left hand side, the general Leibniz rule is applied to the product of $C(t)T(t)$ and $T(t)^{-1}$. We also use that $C(0)=T(0)=I$. But since $T$ and $U$ are trivial flexes, then also $U(t) T(t)^{-1}$ is a trivial flex and hence $C',\ldots,C^{(m)}$ is an $m$-trivial flex.
\end{proof}
\begin{proposition} \label{prop:construction-of-nontrivial-1-inf-motion}
Let $M \in \mathcal{M}^{p \times q}_{3,2}$ and let $(A^{(1)}, \dots, A^{(p)}, B^{(1)}, \dots, B^{(q)})$ be a size-$2$ psd factorization of $M$. If there exists an analytic path $p(t) = (\mathcal{A}(t),\mathcal{B}(t))=(\mathcal{A}C(t), C(t)^{-1}\mathcal{B})$ where $t \in [0,1]$ with $p(0)=(\mathcal{A},\mathcal{B})$, $C(0)= I$ and each $p(t)$ a psd factorization of $M$ that cannot be obtained from $(\mathcal{A},\mathcal{B})$ by $\operatorname{GL}(k)$-action, then $(A^{(1)}, \dots, A^{(p)}, B^{(1)}, \dots, B^{(q)})$ is not 1-infinitesimally rigid.
\end{proposition}

\begin{proof}
The proof will consist of two parts: In the first part, we will construct a motion that is not 1-trivial using the three lemmas above. In the second part, we will show that this motion is 1-infinitesimal. 

Let $p(t)$ 
be as in the statement of the proposition.  
 Consider the smallest positive integer $m$ such that $\left(\frac{d}{dt}\right)^{1}  C(t)|_{t=0},\ldots,\left(\frac{d}{dt}\right)^{m}  C(t)  |_{t=0}$ is not an $m$-trivial flex.
Such $m$ exists by the assumption that $p(t)$ cannot be obtained from $(\mathcal{A},\mathcal{B})$ by $\operatorname{GL}(k)$-action: 
 If for all $m$ we have a trivial $m$-flex, then there exists a trivial flex $T(t)$
 such that $\left(\frac{d}{dt}\right)^{m}  C(t) |_{t=0} = \left(\frac{d}{dt}\right)^{m}  T(t)  |_{t=0}$ for all $m$.
 Since $C(0) = T(0) = I$, the Taylor expansion of $C(t)-T(t)$ around $t=0$ 
 is equal to $0$, and hence $C(t) = T(t)$. But this is a contradiction, since then $p(t)$
 is obtained by a $\operatorname{GL}(k)$-action. 
By~\Cref{claim1}, there exists a trivial flex $T$ such that
\[
\left(\frac{d}{dt}\right)^{\ell} \left(C(t)T(t) \right)|_{t=0} = 0 \text{ for } \ell \in \{1,\ldots,m-1\}.
\]
We consider the analytic function $\tilde{C}(t):=C(t)T(t)$  and  
define 
\[
\dot{C}:= \left(\frac{d}{dt}\right)^{m} \left(\tilde{C}(t) \right)|_{t=0}.
\]

By~\Cref{claim3}, $\left(\frac{d}{dt}\right)  \tilde{C}(t)  |_{t=0}$, $\ldots, \left(\frac{d}{dt}\right)^{m}  \tilde{C}(t) |_{t=0}$ is not an $m$-trivial flex. By~\Cref{claim2}, $\dot{C}$ is not 1-trivial. This shows that $\tilde{C}(t)$ is not $1$-trivial.

 Finally we show that
$\tilde{C}(t)$ we have constructed is an $1$-infinitesimal motion. This will conclude that $\tilde{C}(t)$ is 
not 1-infinitesimally rigid completing the proof.

First we show that
$\langle \dot{A}^{(i)},B^{(j)}  \rangle + \langle A^{(i)},\dot{B}^{(j)}  \rangle = 0$ for all $(i,j) \in [p] \times [q]$. \label{cond:psd-equality} Since we have $M =  \mathcal{A}(t) \mathcal{B}(t)$ for $t \in [0,1]$ we get
\begin{align*}
&0 = \left(\frac{d}{dt}\right)^m \left( \mathcal{A}(t) \mathcal{B}(t) \right)|_{t=0} = \sum _{\ell=0} ^m \binom{m}{\ell} \left( \left(\frac{d}{dt}\right)^{m-\ell} \mathcal{A}(t) \left(\frac{d}{dt}\right)^{\ell}  \mathcal{B}(t)\right){\bigm |}_{t=0} = \\
&\sum _{\ell=0} ^m \binom{m}{\ell} \left(\mathcal{A} \left(\frac{d}{dt}\right)^{m-\ell} \left(\tilde{C}(t) \right) \left(\frac{d}{dt}\right)^{\ell} \left(\tilde{C}^{-1} (t) \right) \mathcal{B}\right){\bigm |}_{t=0} =\\
& \left(\mathcal{A} \left(\frac{d}{dt}\right)^{m} \left(\tilde{C}(t) \right) \left(\frac{d}{dt}\right)^0 \left(\tilde{C}^{-1} (t) \right) \mathcal{B} + \mathcal{A} \left(\frac{d}{dt}\right)^{0} \left(\tilde{C}(t) \right) \left(\frac{d}{dt}\right)^m \left(\tilde{C}^{-1} (t) \right) \mathcal{B}\right){\bigm |}_{t=0} =\\
&\left(\mathcal{A} \dot{C} \right) \left(C^{-1} \mathcal{B} \right) + \left(\mathcal{A} C \right) \left(-\dot{C} \mathcal{B} \right) = \dot{\mathcal{A}} \mathcal{B} + \mathcal{A} \dot{\mathcal{B}}.
\end{align*}
We used above the identity $\left(\frac{d}{dt}\right)^m \left(\tilde{C}^{-1} (t) \right) = -  \tilde{C}^{-1} (t) \left(\frac{d}{dt}\right)^m \left(\tilde{C} (t) \right) \tilde{C}^{-1} (t)$, which follows from $\left(\frac{d}{dt}\right) \left(\tilde{C}^{-1} (t) \right) = -  \tilde{C}^{-1} (t) \left(\frac{d}{dt}\right) \left(\tilde{C} (t) \right) \tilde{C}^{-1} (t)$ and the assumption that $\left(\frac{d}{dt}\right)^{\ell} \left(\tilde{C}(t) \right)|_{t=0}=0$ for $\ell \in \{1,\ldots,m-1\}$.

Next we show that 
there exists $\varepsilon >0$ such that for $t \in [0,\varepsilon)$, $T^1([A^{(i)} + t \dot{A}^{(i)}]_I) \geq 0$ for all $i \in [p]$ and $I \subseteq [k]$ and $T^1([B^{(j)} + t \dot{B}^{(j)}]_J) \geq 0$ for all $j \in [q]$ and $J \subseteq [k]$.
We note that if $[A^{(i)}]_I>0$ or $[B^{(j)}]_J>0$, then for $t>0$ small enough $T^s([A^{(i)} + t \dot{A}^{(i)}]_I) \geq 0$ for all $i \in [p]$ and $I \subseteq [k]$ and $T^s([B^{(j)} + t \dot{B}^{(j)}]_J) \geq 0$ for $j \in [q]$ and $J \subset [k]$. Let us consider the case when $[A^{(i)}]_I=0$ (similar argument will work for $[B^{(j)}]_J=0$). We have 
\begin{align*}
&\left(\frac{d}{dt}\right)^{\ell} \left( \mathcal{A}(t) \right)|_{t=0} = \mathcal{A} \left(\frac{d}{dt}\right)^{\ell} \left(\tilde{C}(t) \right)|_{t=0} = 0 \text{ for } \ell \in \{1,\ldots,m-1\},\\
&\left(\frac{d}{dt}\right)^m \left( \mathcal{A}(t) \right)|_{t=0} = \mathcal{A} \left(\frac{d}{dt}\right)^{m} \left(\tilde{C}(t) \right)|_{t=0} = \mathcal{A} \dot{C}=\dot{\mathcal{A}}.
\end{align*}
We will denote $\mathcal{A}^{(\ell)}(t) := \left(\frac{d}{dt}\right)^{\ell}  \left( \mathcal{A}(t) \right)$.
Since $\mathcal{A}(t)$ is analytic and $\mathcal{A}^{(1)}(0)=\ldots = \mathcal{A}^{(m-1)}(0) = 0$, we can write
\[
\mathcal{A}(t) = \mathcal{A} + \sum_{\ell = m}^{\infty}  \frac{\mathcal{A}^{(\ell)}(0)}{\ell!} t^{\ell}.
\]
Then $[A^{(i)}(t)]_I = [A^{(i)}]_I + t^m g(t)$ for some analytic function $g(t)$. Because $[A^{(i)}]_I=0$, we conclude that $[A^{(i)}(t)]_I = t^m g(t)$. 
Now, $\mathcal{A}(t) \succeq 0$ for $t \in [0,1]$ gives 
\[
[A^{(i)}(t)]_I = t^m g(t) \geq 0 \text{ for } t \in [0,1].
\]
This inequality implies $g(t) \geq 0$ for $t \in (0,1]$ and since $g(t)$ is analytic, also  $g(0) \geq 0$.
Finally,
\begin{align*}
& T^1([A^{(i)} + t \dot{A}^{(i)}]_I) = T^1([A^{(i)} + t \mathcal{A}_i^{(m)}(0)]_I) =\\
& T^m([A^{(i)} + s^m \frac{\mathcal{A}_i^{(m)}(0)}{m!}]_I) = T^m([ \mathcal{A}_i + \sum_{\ell = m}^{\infty}  \frac{\mathcal{A}_i^{(\ell)}(0)}{\ell!} s^{\ell}]_I) =\\
& T^m([A^{(i)}(s)]_I) = T^m(s^m g(s)) = s^m g(0) \geq 0
\end{align*}
for $t \in [0,1]$ and $s \in [0,\sqrt[m]{m!}]$.
At the second equality we did a change of variables $t=s^m/m!$.
\end{proof}

\begin{theorem} \label{lem:inf-rigidity-implies-local-rigidity}
If a size-$2$ psd factorization of a matrix $M \in \mathcal{M}^{p \times q}_{3,2}$ is $1$-infinitesimally rigid, then it is locally rigid. 
\end{theorem}

\begin{proof}
We will prove the contrapositive of the statement following the idea from~\cite[Proposition 1]{asimow1978rigidity}. Assume that a size-$2$ psd factorization $(A^{(1)}, \dots, A^{(p)}$, $B^{(1)}, \dots, B^{(q)})$ of $M$ is not locally rigid. This means that in every neighborhood of this psd factorization, there exists another size-$2$ psd factorization of $M$ that is not obtained by the $\operatorname{GL}(k)$-action. By~\cite[Lemma 18.3]{wallace1958algebraic}, there exists an analytic path $p(t) = (A^{(1)}(t), \dots, A^{(p)}(t), B^{(1)}(t), \dots, B^{(q)}(t))$ where $t \in [0,1]$ with $p(0)=(A^{(1)}, \dots, A^{(p)}, B^{(1)}, \dots, B^{(q)})$ and each $p(t)$ a psd factorization of $M$ that cannot be obtained from $(A^{(1)}, \dots, A^{(p)}, B^{(1)}, \dots, B^{(q)})$ by $\operatorname{GL}(k)$-action. 
By~\Cref{prop:construction-of-nontrivial-1-inf-motion}, there exists a non-trivial $1$-infinitesimal motion of the factorization $(A^{(1)},\ldots,A^{(p)},B^{(1)},\ldots,B^{(q)})$.
\end{proof}

\subsection{Rigidity for positive matrices} \label{sec:local-and-global-rigidity-for-positive-matrices}
Here we turn to {\it global rigidity}. Our goal is to show that for positive $M \in \mathcal{M}_{3,2}^{p \times q}$, the four rigidity conditions, namely, $1$- and $2$-infinitesimal rigidity, and local and global rigidity coincide. This is the content of \Cref{thm:equivalence-of-rigidities}.

\begin{definition}
A size-$k$ psd factorization $(A^{(1)}, \dots, A^{(p)}, B^{(1)}, \dots, B^{(q)})$ of a matrix $M \in \R_+^{p \times q}$ is called {\em globally rigid} if all other psd factorizations of $M$ 
are in the orbit of 
$(A^{(1)},\ldots,A^{(p)},B^{(1)},\ldots,B^{(q)})$.
\end{definition}

Global rigidity is the same as uniqueness of a psd factorization,  up to modifications induced by the action of $\operatorname{GL}(k)$. 
 From these definitions the following immediately follows. 
\begin{lemma}
If a size-$k$ psd factorization of a matrix $M \in \R_+^{p \times q}$ is globally rigid, then it is locally rigid.
\end{lemma}

In general we do not expect the converse of the above statement to hold: as soon as a matrix $M \in \R_+^{p \times q}$ has a locally rigid size-$k$ psd factorization together with another size-$k$ psd factorization not in the same $\operatorname{GL}(k)$-orbit as the first one, then $M$ cannot have a unique psd factorization even up to the action of $\operatorname{GL}(k)$. However, local and global rigidity happen to be the same for matrices 
in $\M_{3,2}^{p \times q}$.

\begin{lemma} \label{lem:local-rigidity-implies-global-rigidity}
Let $M \in \M_{3,2}^{p \times q}$, and let $A^{(i)} \in \mathcal{S}_+^k, \, i\in [p]$ and $B^{(j)} \in \mathcal{S}_+^k, \, j \in [q]$ give a size-$2$ psd factorization of $M$. Then this psd factorization is locally rigid if and only if it is globally rigid.
\end{lemma}

\begin{proof}
As we saw above, global rigidity always implies local rigidity. For the converse implication, assume that $(A^{(1)},\ldots,A^{(p)},B^{(1)},\ldots,B^{(q)})$ is locally rigid. By the definition of local rigidity, there exists a neighborhood of $(A^{(1)},\ldots,A^{(p)}$, $B^{(1)},\ldots,B^{(q)})$ where $\mathcal{SF}(M)/GL(k)$ consists of a single point. By~\cite[Proposition 9]{Fawzi:2015}, the space $\mathcal{SF}(M)/GL(k)$ is connected  in the case when $M \in \M_{3,2}^{p \times q}$. Hence $\mathcal{SF}(M)/GL(k)$ must be equal to $(A^{(1)},\ldots,A^{(p)},B^{(1)},\ldots,B^{(q)})$ and this psd factorization is globally rigid.
\end{proof}

\begin{theorem} \label{thm:equivalence-of-rigidities}
Let $M \in \mathcal{M}_{3,2}^{p \times q}$ 
be a positive matrix.
Suppose $A^{(i)} \in \mathcal{S}_+^2, \, i\in [p]$ and $B^{(j)} \in \mathcal{S}_+^2, \, j \in [q]$ give a size-$2$ psd factorization of $M$. Then the following are equivalent:
\begin{itemize}
    \item[a)] $(A^{(1)},\ldots,A^{(p)},B^{(1)},\ldots,B^{(q)})$ is $1$-infinitesimally rigid.
    \item[b)] $(A^{(1)},\ldots,A^{(p)},B^{(1)},\ldots,B^{(q)})$ is $2$-infinitesimally rigid. 
    \item[c)] $(A^{(1)},\ldots,A^{(p)},B^{(1)},\ldots,B^{(q)})$ is locally rigid.
    \item[d)] $(A^{(1)},\ldots,A^{(p)},B^{(1)},\ldots,B^{(q)})$ is globally rigid.
\end{itemize}
\end{theorem}

The items a) and b) in \cref{thm:equivalence-of-rigidities} are not equivalent for matrices with zeroes as discussed at the beginning of~\Cref{sec:infinitesimal-rigidity}. We conjecture that the items b), c) and d) are equivalent for all matrices $M \in \M_{3,2}^{p\times q}$.

As needed by the context, we will consider rows $\mathcal{A}_i$ of $\mathcal{A}$ 
and columns $\mathcal{B}_j$ of $\mathcal{B}$ 
both as vectors in $\mathbb{R}^{\binom{k+1}{2}}$ and as corresponding $k\times k$ symmetric matrices using the translation introduced in~\cref{sec:infinitestimal-motions}.  We recall that a factorization $(\mathcal{A} \mathcal{C},\mathcal{C}^{-1} \mathcal{B})$ is a psd factorization  if $\mathcal{A} \mathcal{C} \in (\mathcal{S}_+^k)^p$ and $\mathcal{C}^{-1} \mathcal{B} \in (\mathcal{S}_+^k)^q$. 
We let $P_{\mathcal{A}_i} = \{ \mathcal{C} \in \mathbb{R}^{\binom{k+1}{2} \times \binom{k+1}{2}} \, : \, \mathcal{A}_i \mathcal{C} \in \mathcal{S}_+^k\}$
and 
$P_{\mathcal{B}_j} = \{ \mathcal{D} \in \mathbb{R}^{\binom{k+1}{2} \times \binom{k+1}{2}} \, : \, \mathcal{D} \mathcal{B}_j \in \mathcal{S}_+^k\}$. 
 These sets are nonempty convex cones when $(\mathcal{A}, \mathcal{B})$ is a psd factorization since they contain the identity matrix. After fixing a factorization $(\mathcal{A}, \mathcal{B})$,
 the size-$k$ psd factorizations of $M$ are identified with the set $F \cap ((P_{\mathcal{A}_1} \cap \ldots \cap P_{\mathcal{A}_p}) \times (P_{\mathcal{B}_1} \cap \ldots \cap P_{\mathcal{B}_q}))$, where $F$ was defined in~\Cref{section:1-trivial-motions}. 
 This motivates the definition of $U_{(\mathcal{A},\mathcal{B})} = (P_{\mathcal{A}_1} \cap \ldots \cap P_{\mathcal{A}_p}) \times (P_{\mathcal{B}_1} \cap \ldots \cap P_{\mathcal{B}_q})$. This is a semialgebraic set. The psd cone $\mathcal{S}_+^k$ is defined by polynomial inequalities in the entries of the $k \times k$ symmetric matrices. A matrix is on the boundary of $\mathcal{S}_+^k$ if and only if it is psd and has rank smaller than $k$. 
 Therefore $U_{(\mathcal{A}, \mathcal{B})}$ is defined by polynomial inequalities induced by polynomial inequalities defining $\mathcal{S}_+^k$, 
 and  $(\mathcal{C}, \mathcal{D})$ is on the boundary of $U_{(\mathcal{A}, \mathcal{B})}$ if and only if either for some $i$ the matrix corresponding to $\mathcal{A}_i\mathcal{C}$ is psd and has rank smaller than $k$ or for some $j$ 
the matrix corresponding to 
$\mathcal{D}\mathcal{B}_j$ is psd and 
 has rank smaller than $k$. 

We  define $W_{(\mathcal{A},\mathcal{B})}$ to be the set of matrices $\mathcal{D}$ which correspond to $k$-infinitesimal motions of the psd factorization $M = \A \B$. 
The set $W_{(\mathcal{A},\mathcal{B})}$ is equal to the projection on the first factor of the tangent directions $(\mathcal{D},-\mathcal{D})$ such that $(I+t\mathcal{D},I-t\mathcal{D})$ is in the cone $U_{(\mathcal{A},\mathcal{B})}$ for $t \in [0,\varepsilon)$ for some $\varepsilon >0$. The definition of $W_{(\mathcal{A},\mathcal{B})}$ for $k$-infinitesimal motions is analogous to the definition of the cone $P_{(\mathcal{A},\mathcal{B})}$ for $1$-infinitesimal motions.

\begin{lemma} \label{lemma:interior-of-P-is-subset-of-W}
Let $M \in \mathcal{M}^{p \times q}_{3,2}$ and let $(A^{(1)}, \dots, A^{(p)}, B^{(1)}, \dots, B^{(q)})$ be a size-$2$ psd factorization of $M$. Then the interior of $P_{(\mathcal{A},\mathcal{B})}$ is contained in $ W_{(\mathcal{A},\mathcal{B})}$.
\end{lemma}

\begin{proof}
A matrix $D$ in the interior of $P_{(\mathcal{A},\mathcal{B})}$ satisfies for some $\varepsilon >0$ and for all $t \in [0,\varepsilon)$  the inequalities $T^1([A^{(i)} + t \dot A^{(i)}]_I) > 0$ for all $i \in [p]$ and $I \subseteq [k]$ and $T^1([B^{(j)} + t \dot B^{(j)}]_J) > 0$ for all $j \in [q]$ and $J \subseteq [k]$. Therefore it also satisfies for some $\varepsilon' >0$ and for all $t \in [0,\varepsilon')$  the inequalities $T^2([A^{(i)} + t  \dot A^{(i)}]_I) > 0$ for all $i \in [p]$ and $I \subseteq [k]$ and $T^2([B^{(j)}+  t \dot B^{(j)}]_J) > 0$ for all $j \in [q]$ and $J \subseteq [k]$. This is equivalent to $D \in W_{(\mathcal{A},\mathcal{B})}$.
\end{proof}

 \begin{proposition} \label{prop:global-rigidity-implies-infinitesimal-rigidity-for-psd-rank-2-positive-matrices}
Let $M \in \mathcal{M}^{p \times q}_{3,2}$ have positive entries. If a size-$2$ psd factorization  of $M$ is locally rigid, then it is 2-infinitesimally rigid. 
\end{proposition}

\begin{proof}
We will prove the contrapositive: if a size-$2$ psd factorization $(A^{(1)}, \dots, A^{(p)}$, $B^{(1)}, \dots, B^{(q)})$ is $2$-infinitesimally flexible, then it is not locally rigid.
If  $(A^{(1)}, \dots, A^{(p)}$, $B^{(1)}, \dots, B^{(q)})$ is $2$-infinitesimally flexible, then the cone $P_{(\mathcal{A},\mathcal{B})} = \{\underline{D} \in \R^9 \mid C_{(\mathcal{A},\mathcal{B})} \underline{D} \geq 0 \}$, where the matrix $C_{(\mathcal{A},\mathcal{B})}$ is as in~\Cref{def:CAB-and-PAB}, is full-dimensional by the proof of~\Cref{thm:no-orthogonal-pairs}. It follows that $W_{(\mathcal{A},\mathcal{B})}$ is $9$-dimensional since every element in the interior of $P_{(\mathcal{A},\mathcal{B})}$ belongs to $W_{(\mathcal{A},\mathcal{B})}$ by~\Cref{lemma:interior-of-P-is-subset-of-W}. 
We recall that $W_{(\mathcal{A},\mathcal{B})}$ is the projection to the first factor of those tangent directions $(D,-D)$ such that $(I+tD,I-tD)$ is in  $U_{(\mathcal{A},\mathcal{B})}$ for small enough $t>0$. 
Thus the dimension of $U_{(\mathcal{A},\mathcal{B})}$ restricted to the first factor is at least the dimension of $W_{(\mathcal{A},\mathcal{B})}$. Since the dimension of $U_{(\mathcal{A},\mathcal{B})}$ restricted to the first factor is at most nine due to the number of free variables, it is exactly nine.  Thus the tangent space $T_{(I,I)} F$ intersects $U_{(\mathcal{A},\mathcal{B})}$ in its interior (if it would intersect $U_{(\mathcal{A},\mathcal{B})}$ only on the boundary, then the dimension of $W_{(\mathcal{A},\mathcal{B})}$ would be less than nine). 

We will argue that the variety $F$ and each of the maximal dimensional boundaries of $U_{(\mathcal{A},\mathcal{B})}$ intersect transversally at $(I,I)$. The variety $F$ is smooth everywhere. Also every maximal boundary component of $U_{(\mathcal{A},\mathcal{B})}$ is smooth at $(I,I)$ because $M$ has positive entries and thus none of the factors in the psd factorization is the zero matrix. Finally, every maximal boundary component of $U_{(\mathcal{A},\mathcal{B})}$ is a hypersurface defined by an equation in the entries of one factor of $\mathbb{R}^{3 \times 3} \times \mathbb{R}^{3 \times 3}$ only and hence the same is true for its tangent space at $(I,I)$. One can check by direct computation using the Jacobian for the defining equations of $F$ and elimination that $T_{(I,I)} F$ is not contained in any hypersurface defined by an equation in the entries of one factor of $\mathbb{R}^{3 \times 3} \times \mathbb{R}^{3 \times 3}$ only. Therefore $T_{(I,I)} F$ is not contained in the tangent space of the maximal boundary component of $U_{(\mathcal{A},\mathcal{B})}$, and the two tangent spaces must span $\mathbb{R}^{3 \times 3} \times \mathbb{R}^{3 \times 3}$. Therefore, the variety $F$ and each of the maximal dimensional boundaries of $U_{(\mathcal{A},\mathcal{B})}$ intersect transversally at $(I,I)$.  

Finally, this observation together with the tangent space $T_{(I,I)} F$ intersecting $U_{(\mathcal{A},\mathcal{B})}$ in its interior imply that $F$ intersects $U_{(\mathcal{A},\mathcal{B})}$ in its interior as well. This intersection is $9$-dimensional, but the set of size-$2$ psd factorizations obtained by $\operatorname{GL}(2)$-action is just $4$-dimensional, therefore the factorization $(A^{(1)}, \dots, A^{(p)}, B^{(1)}, \dots, B^{(q)})$ is not locally rigid.
\end{proof}

With all the results we have proved above we can complete the proof of~\cref{thm:equivalence-of-rigidities}. 
\begin{proof}[Proof of~\cref{thm:equivalence-of-rigidities}]
The equivalence of a) and b) follows from~\Cref{thm:no-orthogonal-pairs}. The condition a) implies c) by~\Cref{lem:inf-rigidity-implies-local-rigidity}. The equivalence of c) and d) follows from~\Cref{lem:local-rigidity-implies-global-rigidity}. Finally, from~\Cref{prop:global-rigidity-implies-infinitesimal-rigidity-for-psd-rank-2-positive-matrices} we get that c) implies b). 
\end{proof}

\subsection{Rigidity for matrices with one zero} \label{sec:local-and-global-rigidity-one-zero}

In this subsection, we assume that $M$ has one zero or equivalently that the size-2 psd factorization has one orthogonal pair. We restrict to the boundary $M_{11}=0$ and to the corresponding space of factorizations. For this consider the linear space
\[
L_{(\mathcal{A},\mathcal{B})} = \{(C,D) \in \mathbb{R}^{\binom{k+1}{2} \times \binom{k+1}{2}} \times \mathbb{R}^{\binom{k+1}{2} \times \binom{k+1}{2}}: C_{1,2}=C_{1,3}=C_{3,2}=D_{1,2}=D_{1,3}=D_{3,2}=0\}.
\]
We define the sets $F'=F \cap L_{(\mathcal{A},\mathcal{B})}$ and $U'_{(\mathcal{A},\mathcal{B})} = U_{(\mathcal{A},\mathcal{B})}\cap L_{(\mathcal{A},\mathcal{B})}$. The set $W'_{(\mathcal{A},\mathcal{B})}$ is the projection to the first factor of those tangent directions $(D,-D)$  such that $(I+tD,I-tD)$ is in  $U'_{(\mathcal{A},\mathcal{B})}$ for $t \in [0,\varepsilon)$ for some $\varepsilon >0$.  The inverse of an invertible $3 \times 3$ matrix $C$ with $C_{1,2}=C_{1,3}=C_{3,2}=0$ has zeros at the same coordinates, i.e., $C^{-1}_{1,2}=C^{-1}_{1,3}=C^{-1}_{3,2}=0$. Hence the dimensions of $F'$ and $F'$ restricted to the first factor are six. Moreover, since $F$ intersects $L_{(\mathcal{A},\mathcal{B})}$ transversally, then $T_{(I,I)}F'=T_{(I,I)}F \cap \left((I,I) + L_{(\mathcal{A},\mathcal{B})}\right)$.

\begin{proposition} \label{prop:one-zero-global-rigidity-implies-infinitesimal-rigidity}
Let $M \in \mathcal{M}^{p \times q}_{3,2}$ have one zero. If a size-$2$ psd factorization  of $M$ is locally and hence globally rigid, then it is $2$-infinitesimally rigid. 
\end{proposition}

\begin{proof}
Without loss of generality, we assume that $M_{11}=0$.  We will prove the contrapositive of the statement, namely, if a a size-$2$ psd factorization  $(A^{(1)}, \dots, A^{(p)}$, $B^{(1)}, \dots, B^{(q)})$ is $2$-infinitesimally flexible, then it is not locally rigid.  First we prove the claim for $a_1 = (a_{11}, 0)$ and $b_1 = (0,b_{12})$, so that we can work with the same cone $\bar{P}_{(\mathcal{A},\mathcal{B})} = \{\underline{D} \in \R^6 \mid \bar{C}_{(\mathcal{A},\mathcal{B})} \underline{D} \geq 0 \}$ as in the proof of~\Cref{thm:one-orthogonal-pair}.
If a size-$2$ psd factorization  $(A^{(1)}, \dots, A^{(p)}, B^{(1)}, \dots, B^{(q)})$ is $2$-infinitesimally flexible, then the cone $\bar{P}_{(\mathcal{A},\mathcal{B})}$ is full-dimensional by the proof of~\Cref{thm:one-orthogonal-pair}. 

We first show that every element in the interior of $\bar{P}_{(\mathcal{A},\mathcal{B})}$ belongs to $W'_{(\mathcal{A},\mathcal{B})}$.  We recall that $W'_{(\mathcal{A},\mathcal{B})}$ is the projection to the first factor of those tangent directions $(D,-D)$  such that $(I+tD,I-tD)$ is in  $U'_{(\mathcal{A},\mathcal{B})}$ for small enough $t>0$.
A matrix $D$ corresponding to a point in the interior of $\bar{P}_{(\mathcal{A},\mathcal{B})}$ satisfies for some $\varepsilon >0$ and for all $t \in [0,\varepsilon)$  the inequalities $T^1([A^{(i)} + t \dot A^{(i)}]_I) > 0$ for all $i \in \{2,\ldots,p\}$ and $I \subseteq [k]$ and $T^1([B^{(j)} + t \dot B^{(j)}]_J) > 0$ for all $j \in \{2,\ldots,q\}$ and $J \subseteq [k]$. Therefore it also satisfies for some $\varepsilon' >0$ and for all $t \in [0,\varepsilon')$  the inequalities $T^2([A^{(i)} + t  \dot A^{(i)}]_I) > 0$ for all $i \in \{2,\ldots,p\}$ and $I \subseteq [k]$ and $T^2([B^{(j)}+  t \dot B^{(j)}]_J) > 0$ for all $j \in \{2,\ldots,q\}$ and $J \subseteq [k]$. It was shown in the proof of~\Cref{prop:one-orth} that under the assumption $D_{12} = D_{13} = D_{32} = 0$, we have $T^2([A^{(1)} + t \dot A^{(1)}]_I) = T^2([B^{(1)} + t \dot B^{(1)}]_J = 0$ for all $I,J \subseteq [k]$ and for all $t \in \R$. Therefore $D$ gives a $2$-infinitesimal motion and hence it is in $W'_{(\mathcal{A},\mathcal{B})}$.

Since every element in the interior of $\bar{P}_{(\mathcal{A},\mathcal{B})}$ belongs to $W'_{(\mathcal{A},\mathcal{B})}$ and $\bar{P}_{(\mathcal{A},\mathcal{B})}$ is 6-dimensional, then $W'_{(\mathcal{A},\mathcal{B})}$ and $W'_{(\mathcal{A},\mathcal{B})}$ are $6$-dimensional. The dimension of $U'_{(\mathcal{A},\mathcal{B})}$ restricted to the first factor is at least the dimension of $W'_{(\mathcal{A},\mathcal{B})}$. Since the dimension of $U'_{(\mathcal{A},\mathcal{B})}$ restricted to the first factor is at most six due to the number of free variables, it is exactly six.  This means that the tangent space $T_{(I,I)} F$ intersects $U'_{(\mathcal{A},\mathcal{B})}$ in its relative interior (if it would intersect $U'_{(\mathcal{A},\mathcal{B})}$ only on the relative boundary, then the dimension of $W'_{(\mathcal{A},\mathcal{B})}$ would be less than six). Since  $T_{(I,I)} F' = T_{(I,I)} F \cap \left((I,I)+ L_{(\mathcal{A},\mathcal{B})} \right)$, it follows that the tangent space $T_{(I,I)} F'$ intersects $U'_{(\mathcal{A},\mathcal{B})}$ in its relative interior. Hence $F'$ intersects $U'_{(\mathcal{A},\mathcal{B})}$ in its relative interior as well. Thus the factorization $(A^{(1)}, \dots, A^{(p)}, B^{(1)}, \dots, B^{(q)})$ is not locally rigid.

In the general case,  by \cref{lemma:rotation-gives-psd-factorization}, we can find an orthogonal matrix $S \in O(2) \subset \mathrm{GL}(2)$, such that 
        \begin{align*}
            {\tilde a}_1 = S^Ta_1 = \begin{pmatrix}
                1 \\ 0
            \end{pmatrix} \mbox{ and }
            {\tilde b}_1 = S^{-1} b_1 = \begin{pmatrix}
                0 \\ 1
            \end{pmatrix}.
        \end{align*}
Then $\tilde A^{(i)} = S^T A^{(i)} S, \ \tilde B^{(j)} = S^{-1} B^{(j)} S^{-T}$ is a psd factorization of $M$. As it was discussed in the proof of~\Cref{prop:one-orth}, the psd factorization $A^{(1)},\ldots,B^{(3)}$ is $2$-infinitesimally rigid if and only if the psd factorization $\tilde A^{(1)},\ldots, \tilde B^{(3)}$ is $2$-infinitesimally rigid. The same is true for local rigidity by~\Cref{lemma:rotation-gives-psd-factorization}. This completes the proof for the general case.
\end{proof}

Finally we are ready to prove our main theorem. 
\begin{proof}[Proof of~\Cref{main-theorem}]
By~\Cref{thm:equivalence-of-rigidities}, a size-2 psd factorization of $M \in \mathcal{M}_{3,2}^{p \times q}$ with positive entries is globally rigid or equivalently unique up to $\operatorname{GL}(2)$-action if and only if it is $2$-infinitesimally rigid.  The characterization of $2$-infinitesimally rigid factorizations of matrices $M \in \mathcal{M}^{p \times q}_{3,2}$ with no zeros is given in~\Cref{thm:no-orthogonal-pairs}. For matrices in $\mathcal{M}^{p \times q}_{3,2}$ with one zero, \Cref{prop:one-zero-global-rigidity-implies-infinitesimal-rigidity} proves that global rigidity implies $2$-infinitesimal rigidity. 
The characterization of $2$-infinitesimally rigid factorizations of matrices $M \in \mathcal{M}^{p \times q}_{3,2}$ with one zero is given in~\Cref{thm:one-orthogonal-pair}. When we set $\langle a_1,b_1 \rangle=0$ in~\eqref{eqn:Farkas-lemma-conditions-pxq}, the left-hand sides of the inequalities~\eqref{eqn:Farkas-lemma-conditions-pxq} give the left-hand sides of the inequalities~\eqref{eqn:6-factors} together with two infinities. Under our assumptions on the determinants and zero entries, the left-hand sides of~\eqref{eqn:Farkas-lemma-conditions-pxq} and~\eqref{eqn:6-factors} never evaluate to zero. Hence we replace the strict inequalities by non-strict inequalities. Moreover, to avoid division by zero, we switch the denominators and numerators of the expressions. 
\end{proof}

\begin{remark}
If $M$ has at least two zeros, then all its size-2 psd factorizations with at least three rank-$1$ $A$-factors and at least three rank-$1$ $B$-factors are $1$- and $2$-infinitesimally rigid by~\Cref{thm:two-orth-inf-rigid}. 
When we set $\langle a_1, b_1 \rangle = \langle a_2, b_2 \rangle = 0$ in~\eqref{eqn:Farkas-lemma-conditions-pxq}, then all but the left-hand side of the middle inequality of~\eqref{eqn:Farkas-lemma-conditions-pxq} evaluates to infinity. However, it can be checked that the middle inequality can never be negative. Therefore the inequalities~\eqref{eqn:uniqueness-conditions} characterize the $1$- and $2$-infinitesimally rigidity in all cases.
\end{remark}

\subsection{Non-trivial boundaries} \label{section:boundaries}

In this subsection, we study the connection between the boundary of $\M^{p \times q}_{\binom{k+1}{2}, k}$ and rigidity
of psd factorizations. A nonnegative matrix with a zero entry is always on the boundary. In other words, the equality $M_{ij} = 0$ always defines a subset of the boundary of $\M^{p \times q}_{\binom{k+1}{2}, k}$. However, there are matrices $M$ with positive entries that are also on the boundary.

\begin{proposition}\label{prop:psd-factorization-with-size-k-factors}
A positive $M \in \mathcal{M}^{p \times q}_{\binom{k+1}{2},k}$ has a size-$k$ psd factorization where all factors have rank $k$ if and only if the set of size-$k$ psd factorizations of $M$ contains a non-empty subset that is open in the Euclidean subspace topology of $\mu^{-1}(M)$. 
\end{proposition}

\begin{proof}
We first prove that $F$ is not contained in the boundary of $U_{(\mathcal{A},\mathcal{B})}$. The entries of $F$ projected to either factor are algebraically independent, because $F$ was defined as the graph of the function that gives the inverse of $\binom{k+1}{2} \times \binom{k+1}{2}$ invertible matrices. The boundary of $U_{(\mathcal{A},\mathcal{B})}$ is contained in the union of hypersurfaces corresponding to the defining polynomial inequalities of $U_{(\mathcal{A},\mathcal{B})}$. The algebraic independence of the entries of the elements in $F$ implies that no such polynomial can vanish on all of $F$. 
Hence $F$ cannot be contained in the boundary of $U_{(\mathcal{A},\mathcal{B})}$.

Now assume that the set of size-$k$ psd factorizations of $M$ contains a non-empty subset that is open in the Euclidean subspace topology of $\mu^{-1}(M)$. By the above argument, $F$ must intersect $U_{(\mathcal{A},\mathcal{B})}$ in its interior, and hence $M$ has a size-$k$ psd factorization where all factors have rank $k$. Conversely, if $M$ has a size-$k$ psd factorization where all factors have rank $k$, then $F$ intersects $U_{(\mathcal{A},\mathcal{B})}$ in the interior of $U_{(\mathcal{A},\mathcal{B})}$ and $F \cap \text{int} (U_{(\mathcal{A},\mathcal{B})})$ gives the desired open set. 
\end{proof}

\begin{corollary}\label{cor:rank-k-factors-imply-not-globally-rigid}
If a positive $M \in \mathcal{M}^{p \times q}_{\binom{k+1}{2},k}$ has a size-$k$ psd factorization where all factors have rank $k$, then $M$ does not have a size-$k$ globally rigid psd factorization. If $k=2$, then the latter is equivalent to $M$ not having a size-$2$ $1$-infinitesimally, $2$-infinitesimally or locally rigid psd factorization.
\end{corollary}

\begin{proof}
We prove the contrapositive. If $M$ has a size-$k$ psd factorization that is globally rigid, then all size-$k$ psd factorizations of $M$ are obtained by the $\operatorname{GL}(k)$-action. Therefore, the dimension of the set of all size-$k$ psd factorizations of $M$ is $k^2$. But then this set does not contain a non-empty subset that is open in the Euclidean subspace topology of $\mu^{-1}(M)$ because the latter set has dimension $\binom{k+1}{2}^2$. By~\Cref{prop:psd-factorization-with-size-k-factors}, $M$ does not have a size-$k$ psd factorization where all factors have rank $k$.  The last statement follows from \Cref{thm:equivalence-of-rigidities}.
\end{proof}

\begin{proposition} \label{prop:not-inf-rigid-implies-rank-two-factors}
If a positive $M \in \mathcal{M}^{p \times q}_{3,2}$ has a size-$2$ psd factorization that is $1$-infinitesimally flexible, 
then $M$ has a size-$2$ psd factorization where all factors have rank $2$.
\end{proposition}

\begin{proof}
A matrix $M \in \mathcal{M}^{p \times q}_{3,2}$ is positive if and only if there are no $A^{(i)}$ and $B^{(j)}$ among the factors of any psd factorization so that 
$A^{(i)}$ and $B^{(j)}$ are orthogonal.  Hence we are in the setting of~\Cref{sec:no-orthogonal-pairs}. Consider the cone $P_{(\mathcal{A},\mathcal{B})}$  as in~\Cref{def:CAB-and-PAB}. If $(A^{(1)}, \dots, A^{(p)}, B^{(1)}, \dots, B^{(q)})$ is a size-$2$ psd factorization of $M$ that is $1$-infinitesimally flexible, then by the proof of~\Cref{thm:no-orthogonal-pairs}, the cone $P_{(\mathcal{A},\mathcal{B})} $ is full-dimensional. This implies that $W_{(\mathcal{A},\mathcal{B})}$ is full-dimensional by~\Cref{lemma:interior-of-P-is-subset-of-W}.

If $W_{(\mathcal{A},\mathcal{B})}$ has full dimension, then the tangent space $T_{(I,I)} F$ intersects $U_{(\mathcal{A},\mathcal{B})}$ in its interior.  The argument is similar to the one we gave in~\Cref{prop:psd-factorization-with-size-k-factors}: if the intersection were contained in the boundary of $U_{(\mathcal{A},\mathcal{B})}$, then this 
affine space must be completely contained in the boundary. However, entries of elements of this tangent space are algebraically independent. This implies that $F$ intersects $U_{(\mathcal{A},\mathcal{B})}$ in its interior. By~\Cref{prop:psd-factorization-with-size-k-factors}, the matrix $M$ has a size-$2$ psd factorization where all factors have rank $2$.
\end{proof}

\Cref{cor:rank-k-factors-imply-not-globally-rigid} and \Cref{prop:not-inf-rigid-implies-rank-two-factors} together imply that the equivalent conditions in~\Cref{thm:equivalence-of-rigidities} are the same as every size-$2$ psd factorization of $M$ having at least one factor of rank one. Now we will finish with two results about positive matrices on the boundary of $\M_{\binom{k+1}{2}, k}^{p \times q}$ and $\M_{3,2}^{p \times q}$, respectively. 

\begin{proposition}\label{prop:characterization-of-boundaries}
If a positive matrix $M$ is on the boundary of $\mathcal{M}^{p \times q}_{\binom{k+1}{2},k}$ then at least one factor 
in any size-$k$ psd factorization of $M$ has rank smaller than $k$.
\end{proposition}

\begin{proof}
In order to prove the contrapositive, suppose that $M$ has a size-$k$ psd factorization $(A^{(1)},\ldots,A^{(p)},$
$B^{(1)},\ldots,B^{(q)})$ such that 
all factors are strictly positive definite.  
Then this factorization has an open neighborhood $W$ contained in $ (\mathcal{S}_{++}^k)^p \times (\mathcal{S}_{++}^k)^q$ where $\mathcal{S}_{++}^k$ denotes the set of strictly positive definite $k \times k$ matrices. By~\cite[Proposition 5.1]{krone2021uniqueness}, the map $\mu$ is a fiber bundle. Since $\mu$ is an open map as a fibre bundle, $\mu(W)$ is an open neighborhood of $M$ in $\mathcal{M}^{p \times q}_{\binom{k+1}{2},k}$. Hence $M$ is in the interior of $\mathcal{M}^{p \times q}_{\binom{k+1}{2},k}$.
\end{proof}

 We note that the above result is consistent with~\cite[Corollary 3.7]{kubjas2018positive} which states that a positive matrix is on the boundary 
 of $\M_{3,2}^{p \times q}$ if and only if all its size-$2$ psd factorizations have at least three factors  $A^{(i_1)}, A^{(i_2)}, A^{(i_3)}$ and at least three factors 
 $B^{(j_1)}, B^{(j_2)}, B^{(j_3)}$
 with rank one.

We conjecture that the converse of the proposition holds as well: if a positive matrix in $\mathcal{M}^{p \times q}_{\binom{k+1}{2},k}$ does not have a size-$k$ psd factorization where all factors have rank $k$, then it lies on the boundary of $\mathcal{M}^{p \times q}_{\binom{k+1}{2},k}$. Proving the converse would allow us to show that 
a positive matrix $M$ is on the boundary of $\M_{3,2}^{p \times q}$ if and only if  $M$ has a size-$2$ psd factorization satisfying any of the conditions in~\Cref{thm:equivalence-of-rigidities}.

\begin{conjecture}
A positive matrix $M$ is on the topological boundary of $\M_{3,2}^{p\times q}$ if and only if any one of the conditions 
in~\cref{thm:equivalence-of-rigidities} is true. 
\end{conjecture}

Instead, we finish with a weaker result. 

\begin{corollary}
If a positive $M$ is on the boundary of $\mathcal{M}^{p \times q}_{3,2}$
then all size-$2$ psd factorizations 
of $M$ are  $1$-infinitesimally (equivalently, $2$-infinitesimally, locally or globally) rigid.
\end{corollary}

\begin{proof}
If a size-$2$ factorization of $M$ is $1$-infinitesimally flexible, 
by~\Cref{prop:not-inf-rigid-implies-rank-two-factors}, it has a size-$2$ psd factorization where all factors have
rank two. Now~\Cref{prop:characterization-of-boundaries} implies that  $M$ is in the interior of $\mathcal{M}^{p \times q}_{3,2}$.
\end{proof}

\begin{remark}
In~\cite{krone2021uniqueness} the authors study the rigidity of nonnegative matrix factorizations in relation to the boundary of $\tilde{\mathcal{M}}^{p \times q}_{r}$, the set of matrices whose rank and nonnegative rank both equal to $r$. A matrix $M \in \tilde{\mathcal{M}}^{p \times q}_{3}$ is on the boundary of $\tilde{\mathcal{M}}^{p \times q}_{3}$ if and only if $M$ contains a zero entry or all its size-3 nonnegative factorizations are infinitesimally rigid. We believe that an analogous statement is true for $\M_{3,2}^{p\times q}$, namely, $M$ is on the boundary of 
$\M_{3,2}^{p\times q}$ if and only if $M$ has a zero entry or all its size-$2$ psd factorizations are $1$- or $2$-infinitesimally rigid. The statement for the nonnegative rank does not generalize. There are matrices on the boundary of $M \in \tilde{\mathcal{M}}^{p \times q}_{4}$ that have no infinitesimally rigid factorizations. We expect similar phenomena to happen for matrices with higher psd rank.   
\end{remark}

\section*{Acknowledgements}
 We are grateful for the helpful comments we received from the anonymous referee. 
K.\,K.\ and L.\,M.\ were partially supported by the Academy of Finland grant number 323416.

\bibliographystyle{alpha}
\bibliography{sample2}

\bigskip

\end{document}